\documentclass{amsart}

\pdfoutput=1
\usepackage{amssymb,amsfonts,url,graphicx,color,amsthm}
\usepackage{hyperref}
\usepackage{thm-restate}
\usepackage{MnSymbol}
\usepackage[color,import]{xypic}
\newcommand{\cd}[1]{\begin{equation*}{\xymatrix{#1}}\end{equation*}}

\usepackage{tikz-cd}

\usepackage{stmaryrd}

\usepackage{amsmath}

\newcommand{\leftQ}[2]{\left.\raisebox{-.2em}{$#2$}\middle\backslash\raisebox{.2em}{$#1$}\right.}
\newcommand{\smallleftQ}[2]{\left.\raisebox{-0.2em}{\small{$#2$}}\middle\backslash\raisebox{0.1em}{\small{$#1$}}\right.}

\usepackage{enumitem}

\newcommand{\mc}{\mathcal}
\newcommand{\acts}{\curvearrowright}
\newcommand{\Ad}{\mathrm{Ad}}
\newcommand{\co}{\colon\thinspace}

\newcommand{\link}{\mathrm{lk}}
\newcommand{\scwol}{\operatorname{scwol}}
\newcommand{\identity}{\mathbf{1}}
\newtheorem {theorem}{Theorem}[section]
\newtheorem {lemma} [theorem] {Lemma}
\newtheorem {proposition} [theorem] {Proposition}
\newtheorem {example} [theorem] {Example}
\newtheorem {corollary} [theorem] {Corollary}

\newtheorem {claim}{Claim}[theorem]
\newtheorem* {claim*}{Claim}

\newtheorem {question} [theorem] {Question}
\newtheorem {notation} [theorem] {Notation}
\newtheorem {terminology} [theorem] {Terminology}

\newtheorem {observation} [theorem] {Observation}
\newtheorem {convention} [theorem] {Convention}

\newtheorem {assumption} [theorem] {Assumption}
\newtheorem{sa} [theorem] {Standing Assumption}

\DeclareMathOperator\dotnorm{\dot{\lhd}}
\DeclareMathOperator\dotsub{\dot{<}}

\def\Stab {\mathrm Stab}
\def\height {\mathrm{height}}
\def\down {\downarrow}
\def\up {\uparrow}

\def\N {\mathbb N}

\def\bZ {\mathbb Z}
\def\bH {\mathbb H}

\def\Stab {\mathrm{Stab}}
\def\lk {\mathrm{link}}

\def\onto {\twoheadrightarrow}

\def\o {\mathfrak{o}}

\def\Gr {\mathrm{Gr}}
\def\cgyt {\widetilde{CG\left(\mc{Y}\right)}}

\newcommand{\orbit}[1] {\llbracket {#1} \rrbracket}

\newcommand{\backwards}[1]{\overleftharpoon{#1}}

\newtheorem{thmA}{Theorem}

\newtheorem{corA}[thmA]{Corollary}

\theoremstyle{definition}
\newtheorem {definition} [theorem] {Definition}
\newtheorem {remark} [theorem] {Remark}

\begin{document}
\title[Hyperbolic groups acting improperly]{Hyperbolic groups acting improperly}

\author[D. Groves]{Daniel Groves}
\address{Department of Mathematics, Statistics, and Computer Science,
University of Illinois at Chicago,
322 Science and Engineering Offices (M/C 249),
851 S. Morgan St.,
Chicago, IL 60607-7045}
\email{groves@math.uic.edu}

\author[J.F. Manning]{Jason Fox Manning}
\address{Department of Mathematics, 310 Malott Hall, Cornell University, Ithaca, NY 14853}
\email{jfmanning@math.cornell.edu}

\begin{abstract}
 In this paper we study hyperbolic groups acting on CAT$(0)$ cube complexes.  The first main result (Theorem \ref{t:cell qc iff hyp}) is a structural result about the Sageev construction, in which we relate quasi-convexity of hyperplane stabilizers with quasi-convexity of cell stabilizers.
The second main result (Theorem \ref{omnibus}) generalizes both Agol's theorem on cubulated hyperbolic groups and Wise's Quasi-convex Hierarchy Theorem.
\end{abstract}
\thanks{The first author was partially supported by the Simons Foundation, \#342049 to Daniel Groves and the National Science Foundation, DMS-1507067 and DMS-1904913.  The second author was also partially supported by the Simons Foundation \#524176 to Jason Manning and the National Science Foundation, DMS-1462263.}

\maketitle
\setcounter{tocdepth}{1}
\tableofcontents
\section{Introduction}
In recent years, CAT$(0)$ cube complexes have played a central role in many spectacular advances, most notably in Agol's proof of the Virtual Haken and Virtual Fibering Theorems in \cite{VH}.   The main result of \cite{VH} is that a hyperbolic group which acts properly and cocompactly on a CAT$(0)$ cube complex is virtually special.  A key ingredient in Agol's proof was the work of Wise from \cite{WisePUP}, particularly Wise's Quasi-convex Hierarchy Theorem \cite[Theorem 13.3]{WisePUP}.  One of the two main results of the current paper is Theorem \ref{omnibus}, which provides a simultaneous generalization of Agol's theorem and Wise's theorem.
So far this generalization has been applied in \cite{EinsteinG} and \cite{Yen_thesis}.  At the end of the introduction in Subsection \ref{ss:New VH}, we explain how Theorem \ref{omnibus} (together with Theorem \ref{t:cell qc iff hyp}) simplifies the proof of the Virtual Haken and Virtual Fibering Theorems for hyperbolic $3$--manifolds, requiring only a single immersed quasi-Fuchsian surface instead of a ubiquitous family.

Cube complexes in group theory arise via the construction of Sageev \cite{Sageev} which takes as input a group $G$ and a collection of {\em codimension--$1$} subgroups of $G$ and produces a CAT$(0)$ cube complex $X$, equipped with an isometric $G$--action on $X$ with no global fixed point.  The other main result of the current paper is Theorem \ref{t:cell qc iff hyp}, which establishes some fundamental properties about the Sageev construction.

Sageev's construction works in great generality.  However, in order to get more information from the $G$--action on $X$, it is useful to add geometric hypotheses.  For example, if $G$ is a hyperbolic group and the codimension--$1$  subgroups are quasi-convex,
Sageev proved that the associated cube complex is $G$--cocompact \cite[Theorem 3.1]{Sageev-Codim1}.  Achieving a {\em proper} action is harder (see \cite{bergeronwise,hruskawise:finiteness} for conditions which ensure properness).

Even an improper action $G\acts X$ gives a description of $G$ as the fundamental group of a {\em complex of groups} in the sense of Bridson--Haefliger (see \cite[III.$\mc{C}$]{bridhaef:book} or Section \ref{s:complex of groups} below).  In this description, the underlying space is $\leftQ{X}{G}$ and the local groups can be identified with cell stabilizers for the action.

Our first main result links the geometry of the hyperplane stabilizers with that of the cell stabilizers.

\begin{thmA} \label{t:cell qc iff hyp}
  Let $G$ be hyperbolic.  The following conditions on a cocompact $G$--action on a CAT$(0)$ cube complex are equivalent:
  \begin{enumerate}
  \item\label{hyp QC} All hyperplane stabilizers are quasi-convex.
 \item\label{vertex QC} All vertex stabilizers are quasi-convex.
  \item\label{cell QC} All cell stabilizers are quasi-convex.
  \end{enumerate}
\end{thmA}
Intersections of quasi-convex subgroups are quasi-convex, and cell stabilizers are intersections of vertex stabilizers.  Therefore, the equivalence of \eqref{vertex QC} and \eqref{cell QC} is trivial.  We prove the equivalence of \eqref{hyp QC} and \eqref{vertex QC}. 

We remark that we actually prove the direction \eqref{hyp QC} $\implies$ \eqref{vertex QC} in the more general setting of arbitrary finitely generated groups where we assume the relevant subgroups are {\em strongly quasi-convex} in the sense of \cite{Tran}.  Note that in this more general setting, \eqref{vertex QC} and \eqref{cell QC} are still equivalent.  See Section \ref{s:QC proof} for more details.
In Subsection \ref{ss:Cor B} we explain how Theorem \ref{t:cell qc iff hyp} implies the following result.

\begin{restatable}{corA}{cubecomplexhyp}\label{cor:cube complex hyp}
Suppose that $G$ is a hyperbolic group acting cocompactly on a CAT$(0)$ cube complex $X$ with quasi-convex hyperplane stabilizers.  Then
\begin{enumerate}
\item $X$ is $\delta$--hyperbolic for some $\delta$;
\item\label{cond:fixed} there exists a $k \ge 0$ so that the fixed point set of any infinite subgroup of $G$ intersects at most $k$ distinct cells; and
\item\label{cond:acyl} the action of $G$ on $X$ is acylindrical (in the sense of Bowditch \cite[p. 284]{bowditch:tight}).
\end{enumerate}
\end{restatable}
  Anthony Genevois explained to us how conclusion \eqref{cond:fixed} implies acylindricity for actions on hyperbolic CAT$(0)$ cube complexes (see Subsection \ref{ss:Cor B}).  The condition in \eqref{cond:fixed} is not implied by acylindricity since $X$ is not assumed to be locally compact.

Without the conclusion of $\delta$--hyperbolicity, a more general version of Corollary \ref{cor:cube complex hyp} holds just as for Theorem \ref{t:cell qc iff hyp}.  See Remark \ref{rem:middleconclusion} for more details.

In Sageev's construction, the stabilizers in $G$ of hyperplanes in the resulting cube complex are commensurable with the chosen codimension--$1$ subgroups of $G$.  Therefore, we have the following result.

\begin{corA} \label{c:Sageev cell QC}
Let $G$ be a hyperbolic group and let $\mc{H} = \{ H_1, \ldots, H_k \}$ be a collection of quasi-convex codimension--$1$ subgroups.  Let $X$ be a CAT$(0)$ cube complex obtained by applying the Sageev construction to $\mc{H}$. 

\begin{enumerate}
\item\label{stabFP} The stabilizers of cells in $X$ are quasi-convex in $G$.  In particular, they are finitely presented.
\item $X$ is $\delta$--hyperbolic for some $\delta$.
\item There exists a $k \ge 0$ so that the fixed set of any infinite subgroup of $G$ intersects at most $k$ distinct cells.
\item The action of $G$ on $X$ is acylindrical.
\end{enumerate}
\end{corA}
As far as we are aware, even the corollary of item \eqref{stabFP} that the cell stabilizers are finitely generated in the above result is new.  We remark that the fact that cell stabilizers are finitely presented implies that the description of $G$ as the fundamental group of the complex of groups associated to $\leftQ{X}{G}$ is a finite description.

Some of the most dramatic uses of CAT$(0)$ cube complexes have come from Haglund and Wise's theory of {\em special} cube complexes \cite{HW08}.   A cube complex is {\em special} if it admits a locally isometric immersion into the Salvetti complex of a right-angled Artin group.  A group $G$ is {\em virtually special} if there is a finite-index subgroup $G_0 \le G$ and a CAT$(0)$ cube complex $X$ so that $G_0$ acts freely and cubically on $X$ and $\leftQ{X}{G_0}$ is a compact special cube complex.  (For some authors the quotient is allowed to be non-compact but have finitely many hyperplanes.)

As shown in \cite{HW08}, virtually special hyperbolic groups have many remarkable properties, such as being residually finite, linear over $\mathbb Z$ and possessing very strong subgroup separability properties.

Agol \cite{VH} proved that if a hyperbolic group $G$ acts properly and cocompactly on a CAT$(0)$ cube complex then $G$ is virtually special.  It is this result that implies the Virtual Haken Conjecture, as well as the Virtual Fibering Conjecture (in the compact case), and many other results.

One of the key ingredients of the proof of Agol's Theorem, and another of the most important theorems in the  area is Wise's Quasi-convex Hierarchy Theorem \cite[Theorem 13.3]{WisePUP} (see also \cite[Theorem 10.2]{AGM_msqt}) which states that if a hyperbolic group $G$ can be expressed as   $A \ast_C$ (respectively $A \ast_C B$) where $C$ is quasi-convex in $G$ and $A$ is (respectively $A$ and $B$ are) virtually special then $G$ is virtually special.  This theorem can be rephrased as saying that if a hyperbolic group acts cocompactly on a {\em $1$-dimensional CAT$(0)$ cube complex} (otherwise known as a `tree') with virtually special and quasi-convex cell stabilizers, then $G$ is virtually special.

Our second main result is a common generalization of Agol's theorem and Wise's Quasi-convex Hierarchy Theorem.

\begin{thmA} \label{omnibus}
 Suppose that $G$ is a hyperbolic group acting cocompactly on a CAT(0) cube complex $X$ so that cell stabilizers are quasi-convex and virtually special.  Then $G$ is virtually special.
\end{thmA}

By Corollary \ref{c:Sageev cell QC}, Theorem \ref{omnibus} has the following immediate consequence.

\begin{corA}
Suppose that $G$ is a hyperbolic group and that $\mc{H} = \{ H_1, \ldots, H_k \}$ is a collection of quasi-convex codimension--$1$ subgroups.  If the vertex stabilizers of the $G$--action on a cube complex obtained by applying the Sageev construction to $\mc{H}$ are virtually special, then $G$ is virtually special. 
\end{corA}

Since finding proper actions of hyperbolic groups on CAT$(0)$ cube complexes is much harder than finding cocompact actions, Theorem \ref{omnibus} is expected to be a powerful new tool for proving that hyperbolic groups are virtually special.  As mentioned above, already Yen Duong \cite{Yen_thesis} has used Theorem \ref{omnibus} to show that random groups in the square model at density $<1/3$ are virtually special.  Theorem \ref{omnibus} (as well as Corollary \ref{t:quotient CAT(0)} below) are also applied in the paper \cite{EinsteinG} to provide a characterization of relatively hyperbolic groups with $2$--sphere boundary in terms of actions on cube complexes.

Theorem \ref{t:cell qc iff hyp} is one of the key ingredients of the proof of Theorem \ref{omnibus}.  We now explain how Theorem \ref{omnibus} is a consequence of the above-mentioned results of Agol and Wise, along with Theorem \ref{t:cell qc iff hyp} and the following result (proved in Section \ref{s:Dehn filling}).

 \begin{restatable}{thmA}{fillproper}\label{t:fill to get proper action}
  Suppose that the hyperbolic group $G$ acts cocompactly on a CAT$(0)$ cube complex $X$ and that cell stabilizers are virtually special and quasi-convex.
  There exists a quotient $\overline{G} = G/K$ so that 
 \begin{enumerate}
  \item The quotient $\leftQ{X}{K}$ is a CAT$(0)$ cube complex;
  \item The group $\overline{G}$ is hyperbolic; and
  \item The action of $\overline{G}$ on $\leftQ{X}{K}$ is proper (and cocompact).
 \end{enumerate}
 \end{restatable}

\begin{proof}[Proof of Theorem \ref{omnibus}]
Consider the hyperbolic group $G$, acting on a CAT$(0)$ cube complex $X$ as in the statement of Theorem \ref{omnibus}.  
By Theorem \ref{t:fill to get proper action} there exists a hyperbolic quotient $\overline{G} = G/K$ of $G$ so that
$\leftQ{X}{K}$ is a CAT$(0)$ cube complex, and the $\overline{G}$--action on $\leftQ{X}{K}$ is proper and cocompact.  Let $Z = \leftQ{X}{K}$.  

By Agol's Theorem \cite[Theorem 1.1]{VH}, there is a finite-index subgroup $\overline{G}_0$ of $\overline{G}$ so that $\leftQ{Z}{\overline{G}_0}$ is special.  Let $G_0$ be the pre-image in $G$ of $\overline{G}_0$.  Clearly, the underlying space of $\leftQ{X}{G_0}$ is the same as that of $\leftQ{Z}{\overline{G}_0}$, and in particular all of the hyperplanes are two-sided and embedded.

We cut successively along these hyperplanes, applying the complex of groups version of the Seifert-van Kampen Theorem \cite[Example III.$\mc{C}$.3.11.(5) and Exercise III.$\mc{C}$.3.12]{bridhaef:book}.  In this way, we obtain a hierarchy of $G_0$ with the following properties:

\begin{enumerate}
\item The edge groups are quasi-convex (since they are stabilizers of hyperplanes, which are quasi-convex by Theorem \ref{t:cell qc iff hyp}); and
\item The terminal groups are virtually special (since they are finite-index subgroups of the vertex stabilizers in $G$).\end{enumerate}

Therefore, $G_0$ admits a quasi-convex hierarchy terminating in virtually special groups, so $G_0$ is virtually special by Wise's Quasi-convex Hierarchy Theorem \cite[Theorem 13.3]{WisePUP} (see \cite[Theorem 10.3]{AGM_msqt} for a somewhat different account).  Since $G_0$ is finite-index in $G$, the group $G$ is virtually special, as required.  This completes the proof of Theorem \ref{omnibus}.
\end{proof}
\begin{remark}
We thank one of the referees for pointing out that one can replace the use of the complex of groups Seifert-van Kampen Theorem in the previous proof with the following argument:  Once all of the hyperplanes in $\leftQ{X}{G_0}$ are two-sided and embedded, lift to $X$ and consider the trees dual to the hyperplanes.  This gives a collection of $G_0$--trees.  Order them in some way.  If $V$ stabilizes a vertex in the first tree, consider the $V$--action on the second tree.  The stabilizers in $V$ for this second action act on the third tree, and so on.  In this way, a quasi-convex hierarchy for $G_0$ is obtained, and the proof finishes as above.
\end{remark}

We now briefly outline the contents of this paper.  In Section \ref{s:complex of groups} we recall those parts of the theory of complexes of groups from \cite{bridhaef:book} which we need.  
  In Section \ref{s:QC proof}, we prove Theorem \ref{t:cell qc iff hyp} and Corollary \ref{cor:cube complex hyp}.  The proof of Theorem \ref{t:cell qc iff hyp} depends on a quasi-convexity criterion (Theorem \ref{t:globally QC}) which is proved separately in Appendix \ref{app:QC}.  We separate out Theorem \ref{t:globally QC} and its proof both because it may be of independent interest and because the methods, unlike in the rest of the paper, are pure $\delta$--hyperbolic geometry.
 In Section \ref{s:X mod K CAT(0)} we investigate conditions on a group $G$ acting on a CAT$(0)$ cube complex $X$ and a normal subgroup $K \unlhd G$ so that the quotient $\leftQ{X}{K}$ is a CAT$(0)$ cube complex.  In Section \ref{s:algebra} we translate these conditions into group-theoretic statements.  In Section \ref{s:Dehn filling} we prove various results about Dehn filling (in particular, Theorem \ref{t:meta} and Corollary \ref{t:quotient CAT(0)} which may be of independent interest) to see that the conditions from Section \ref{s:algebra} are satisfied for certain subgroups $K$ which arise as kernels of long Dehn filling maps.  We use this to deduce Theorem \ref{t:fill to get proper action}.  

\subsection{Notation and conventions}
The notation $A\dotsub B$ indicates that $A$ is a finite index subgroup of $B$; similarly, $A\dotnorm B$ indicates $A$ is a finite index normal subgroup.  We write conjugation as $a^x = xax^{-1}$, or sometimes as $\Ad(x)(a)$.  For $p$ an element of a $G$--set, we denote the $G$--orbit by $\orbit{p}$.

\subsection{Virtual Haken and Fibering with a single surface}  \label{ss:New VH}

Let $M$ be a closed hyperbolic $3$--manifold, and let $\Gamma = \pi_1(M)$.
Agol's proof that $M$ is virtually Haken and virtually fibered in \cite{VH} proceeds via proving that $\Gamma$ is virtually special.  This relies on Bergeron--Wise's theorem that $\Gamma$ acts properly and cocompactly on a CAT$(0)$ cube complex \cite{bergeronwise}.  In turn, Bergeron--Wise rely on work of Kahn--Markovic \cite{KM12}, which provides a ``ubiquitous''\footnote{This terminology is from Cooper and Futer \cite{CooperFuter}.} family of immersed quasi-Fuchsian surfaces in $M$.  That there is such an abundance of surfaces follows from the proofs in \cite{KM12},   
but is not explicitly stated there.

Here we point out that the results in this paper show that the fact that $\Gamma$ is virtually special follows from the existence of a {\em single} immersed quasi-Fuchsian surface in $M$.    It is explained in \cite{WisePUP} how Virtual Haken and Virtual Fibering follow.

\begin{theorem}
Suppose that $M$ is a closed hyperbolic $3$--manifold and that $M$ contains an immersed quasi-Fuchsian surface.  Then $\pi_1(M)$ is virtually special.
\end{theorem}
\begin{proof}
  If $M$ is non-orientable, we replace it by its orientation double cover.
  Let $\Gamma\cong \pi_1M$ be a lattice in $\operatorname{Isom}^+(\bH^3)$, so that $M \cong \leftQ{\bH^3}{\Gamma}$.  We note that in this setting a subgroup $W<\Gamma$ is geometrically finite as a Kleinian group if and only if it is quasi-convex in $\Gamma$ (see \cite[Theorem 2]{Kapovich-Short} or \cite[Theorem 1.1 and Proposition 1.3]{Swarup}).

  Let $H<\Gamma$ be the subgroup corresponding to the immersed quasi-Fuchsian surface.  Since $H$ is quasi-convex and codimension--$1$ in $\Gamma$, we can apply the Sageev construction to obtain a cocompact action of $\Gamma$ on a CAT$(0)$ cube complex $X$ with no global fixed point, and with hyperplane stabilizers conjugate to $H$.  Theorem \ref{t:cell qc iff hyp} implies that the vertex stabilizers for this action are quasi-convex in $\Gamma$.  To apply Theorem \ref{omnibus}, we will show that the vertex stabilizers admit quasi-convex hierarchies and hence are virtually special.

  Let $V<\Gamma$ be a vertex stabilizer. 
  Since $V$ is quasi-convex in $\Gamma$ it is a geometrically finite subgroup of $\operatorname{Isom}^+(\bH^3)$.  As $V$ has infinite index in $\Gamma$, it acts with infinite covolume on $\bH^3$.  An argument of Thurston shows that every finitely generated subgroup of $V$ is also geometrically finite \cite[Proposition 7.1]{MorganUniformization}.

  Since $\Gamma$ contains no parabolics, neither does $V$.  Thus a small closed neighborhood $N$ of the convex core of $\leftQ{\bH^3}{H}$ is a compact $3$--manifold  with nonempty boundary, and hence is irreducible in the sense that every embedded $2$--sphere bounds a ball \cite[Propositions 2.36, 3.1]{mt:hmkg}.  A compact irreducible $3$--manifold with nonempty boundary is Haken (see \cite[Chapter 6]{Hempel}, \cite[Chapter III]{jaco}).  In particular it has a Haken hierarchy \cite[IV.12]{jaco}.  This topological hierarchy of $N$ gives a group-theoretic hierarchy of $V$.  The edge groups in the hierarchy are finitely generated.  The previously mentioned argument of Thurston then implies that the edge groups are geometrically finite and hence quasi-convex in $\Gamma$.  In particular, this is a quasi-convex hierarchy, and we may apply Wise's Quasi-convex Hierarchy Theorem to conclude that $V$ is virtually special.

  Since all vertex stabilizers of the action $\Gamma\acts X$ are quasi-convex and virtually special, we may apply Theorem \ref{omnibus} to conclude that $\Gamma$ is itself virtually special.
\end{proof}

\subsection{Acknowledgments}  
We thank Richard Webb for suggesting that the direction \eqref{hyp QC} $\implies$ \eqref{vertex QC} of Theorem \ref{t:cell qc iff hyp} might hold in a more general setting than that of hyperbolic groups.

  Thanks to Anthony Genevois for pointing out that his work allows us to deduce acylindricity (Corollary \ref{cor:cube complex hyp}.\eqref{cond:acyl}) from Corollary \ref{cor:cube complex hyp}.\eqref{cond:fixed}.

  We also thank Alessandro Sisto for pointing out the applicability of a ``Greendlinger Lemma'' type result in our joint work (see \cite[Lemma 2.26]{GMS}).  This lemma inspired Theorem~\ref{t:greendlinger} in the current paper, which simplified the proof of Theorem~\ref{t:meta}.

Finally, we thank the anonymous referees whose careful readings and comments have led to substantial improvements in the exposition of this paper.

\section{Complexes of groups} \label{s:complex of groups}
In this section we give a brief account of those parts of the theory of complexes of groups which we need.  Much more detail can be found in Bridson--Haefliger \cite[III.$\mc{C}$]{bridhaef:book}.

\subsection{Paths and homotopies in a category} 
The definitions here are mainly taken from \cite[III.$\mc{C}$.A]{bridhaef:book}, though our notation is slightly different.

Let $\mc{C}$ be a category.
For an arrow $a$ of $\mc{C}$, we denote its source by $i(a)$ and its target by $t(a)$.
An \emph{oriented edge} of $\mc{C}$ is a symbol $a^+$ or $a^-$ where $a$ is an arrow of $\mc{C}$.  The source and target of an oriented edge are defined by
\[ i(a^-) = i(a),\quad t(a^-) = t(a),\quad\mbox{and}\quad i(a^+) = t(a),\quad  t(a^+) = i(a).\]
(We caution readers that this may be the opposite of what they expect.  The signs are chosen so that concatenation of $+$ edges is homotopic to composition of the corresponding arrows (see Definition \ref{def:elem homotopy}).)

We now define $\mc{C}$--paths.
A \emph{$\mc{C}$--path $p$ of length $0$} is an object $v$ of $\mc{C}$, with $i(p) = t(p) = v$.  For $j > 0$, a \emph{$\mc{C}$--path of length $j$} is a list $p = e_1\cdot e_2 \cdots  e_j$ where for each $i$ we have $t(e_i) = i(e_{i+1})$.  We have $i(p) = i(e_1)$ and $t(p) = t(e_j)$.

If $p$ is a $\mc{C}$--path of length $j$, $q$ is a $\mc{C}$--path of length $k$, and $t(p) = i(q)$, then the concatenation $p \cdot q$ is a $\mc{C}$--path of length $j+k$ with $i(p \cdot q) = i(p)$ and $t(p \cdot q) = t(q)$.\footnote{For purposes of concatenation, a $\mc{C}$--path of length $0$ is regarded as an empty list.}

The category $\mc{C}$ is {\em connected} if for any two objects $v_0, v_1$ in $\mc{C}$ there is a $\mc{C}$--path $p$ with $i(p) = v_0$ and $t(p) = v_1$.

  If $p$ is a $\mc{C}$--path, then $p$ is \emph{non-backtracking} if it contains no subpath of the form $a^+ \cdot a^-$ or $a^-\cdot a^+$.

\begin{definition} \label{def:elem homotopy}
  \emph{Homotopies} of $\mc{C}$--paths (see  \cite[III.$\mc{C}$.A.11]{bridhaef:book}) are generated by the following \emph{elementary homotopies}, valid whenever both sides are paths:
 \begin{enumerate}
 \item $p\cdot a^+\cdot a^- \cdot q\simeq p\cdot q$ or $p\cdot a^-\cdot a^+\cdot q\simeq p\cdot q$;
 \item $p\cdot a^+\cdot b^+\cdot q \simeq p\cdot (ab)^+\cdot q$ or $p\cdot b^-\cdot a^-\cdot q\simeq p\cdot (ab)^-\cdot q$ (here and below we write $ab$ for the composition $a\circ b$); and
 \item $p\cdot \identity_v^\pm \cdot q\simeq p\cdot q$ (where $\identity_v$ is an identity arrow).
 \end{enumerate}
 \end{definition}

 Any category has a \emph{nerve} which is a simplicial complex whose $0$--cells are the objects of $\mc{C}$, with $1$--cells corresponding to arrows, $2$--cells to composable pairs of arrows, and so on.  The $\mc{C}$--paths we have just defined give edge-paths and the elementary homotopies correspond to simplicial homotopies in this complex.

\subsection{Small categories without loops (scwols)}
By a \emph{scwol} (small category without loops) we mean a small category in which for every object $v$, the set of arrows from $v$ to itself contains only the unit $\identity_v$, and this unit $\identity_v$ cannot be written as a composition of other arrows.
  An arrow is \emph{trivial} if it is equal to $\identity_v$ for some object $v$.
  We sometimes conflate $v$ and $\identity_v$.  A \emph{(non-degenerate) morphism} of scwols $f\co \mc{A}\to \mc{B}$ is a functor which induces, for each object $v$ of $\mc{A}$, a bijection between the arrows $\{a\mid i(a) = v\}$ and the arrows $\{a\mid i(a)=f(v)\}$.

  \begin{definition}[Simple scwol, scwolification]
    \label{def:simplescwol}
    A scwol in which there is at most one arrow with a given source and target will be called a \emph{simple scwol}.  Any small category $\mc{C}$ has a canonical quotient category $\scwol(\mc{C})$ which is a simple scwol.  The objects of  $\scwol(\mc{C})$ are equivalence classes of objects of $\mc{C}$, where $v\sim w$ if there are arrows $a$ and $b$ so that $i(a)=t(b)=v$ and $i(b)=t(a)=w$.  Similarly, the arrows of $\scwol(\mc{C})$ are equivalence classes of arrows of $\mc{C}$, where $a\sim b$ whenever $i(a)\sim i(b)$ and $t(a)\sim t(b)$.  We may refer to $\scwol(\mc{C})$ as the \emph{scwolification of $\mc{C}$.}  The map $\mc{C}\to \scwol(\mc{C})$ taking each object and arrow to its equivalence class will be called the \emph{scwolification functor}.
  \end{definition}

\begin{remark}\label{rem:scwolify is natural}
  The procedure of scwolification is natural.  In particular, a group action on a small category $\mc{C}$ descends to an action on $\scwol(\mc{C})$.
\end{remark}
  
A key example of a scwol is the (opposite) poset of cells of a simplicial or cubical complex, with arrows from each cell to all its faces.  If two faces of some cell are glued together one obtains a non-simple scwol.

\begin{definition}\cite[III.$\mc{C}$.1.3]{bridhaef:book}
  A scwol $\mc{A}$ has a \emph{(geometric) realization} which is a simplicial complex whose $0$--cells are the objects of $\mc{A}$, with $1$--cells corresponding to nontrivial arrows, $2$--cells to composable pairs of such arrows, and so on.  
\end{definition}
The realization of $\mc{A}$ is naturally a sub-complex of the nerve of $\mc{A}$.  
Although the nerve is necessarily infinite dimensional, the realization of a scwol has dimension equal to the length of the longest chain of nontrivial composable arrows.  Every scwol which appears in the current paper has finite dimensional realization.

 \begin{definition}
   If $A$ is the realization of a scwol $\mc{A}$, then there is a canonical correspondence between combinatorial paths in the $1$--skeleton of $A$ and $\mc{A}$--paths without trivial arrows.  If $p$ is a combinatorial path in $A^{(1)}$, and $q$ the corresponding $\mc{A}$--path, we say that $p$ is the \emph{realization} of $q$, and $q$ is the \emph{idealization} of $p$.
\end{definition}

\subsection{Complexes of groups}
\begin{definition} \cite[III.$\mc{C}$.2.1]{bridhaef:book}
  Let $\mc{A}$ be a scwol.  A \emph{complex of groups} $H(\mc{A})$ consists of the following data:
  \begin{enumerate}
  \item For each object $\sigma$ of $\mc{A}$, a \emph{local group} (also called a \emph{cell group}) $H_\sigma$;
  \item For each arrow $a$ of $\mc{A}$, an injective group homomorphism $\psi_a\co H_{i(a)}\to H_{t(a)}$ (if $a$ is a trivial arrow, we require $\psi_a$ to be the identity map); and
  \item For each pair of composable arrows $a,b$ with composition $a b$, a \emph{twisting element} $z(a,b)\in H_{t(a)}$.  (If either $a$ or $b$ is trivial, $z(a,b)=1$.)\footnote{In \cite{bridhaef:book} the notation $g_{a,b}$ is used instead of $z(a,b)$.}
  \end{enumerate}
  These data satisfy the following conditions (continuing to write $ab$ for $a\circ b$) whenever all written compositions of arrows are defined:
\begin{enumerate}
  \item (compatibility) $\Ad(z(a,b))\psi_{ab}=\psi_a\psi_b$;\footnote{Recall $\Ad(z)(x)=zxz^{-1}$.} and
  \item (cocycle) $\psi_a(z(b,c))z(a,bc)=z(a,b)z(ab,c)$.
  \end{enumerate}
\end{definition}
The cocycle condition above applies to any arrangement of arrows of the following form:
 \begin{tikzcd}[row sep = small]
   & \cdot \arrow[drr, "ab" description]\arrow[ddr,"b" description] & & \\
   \cdot \arrow[ur,"c" description] \arrow[drr,"bc" description] & & & \cdot \\
   & & \cdot \arrow[ur, "a" description] & 
 \end{tikzcd}

\begin{definition}[The complex of groups coming from an action]\label{def:actionchoices}
  Suppose $G$ acts on a scwol $\mc{X}$ so that any $g\in G$ fixing an object fixes every arrow from that object.  Suppose further that $\mc{Y} = \leftQ{\mc{X}}{G}$ is a scwol.  We obtain a complex of groups $G(\mc{Y})$ once we have made the following choices \cite[III.$\mc{C}$.2.9]{bridhaef:book}:
  \begin{enumerate}
  \item For each object $v$ of $\mc{Y}$, a lift $\widetilde{v}$ to $\mc{X}$; this lift also determines lifts $\widetilde{a}$ of all arrows $a$ with $i(a) = v$.
  \item For each arrow $a$, a choice of $h_a\in G$ so that $t(h_a(\widetilde{a}))=\widetilde{t(a)}$. (When $a$ is a trivial arrow, we always take $h_a = 1$.)
  \end{enumerate}

  Given these choices, one defines:
  \begin{enumerate}
  \item $G_v$ is the stabilizer of $\widetilde{v}$;
  \item $\psi_a = \Ad(h_a)|_{G_{i(a)}}$;
  \item $z(a,b) = h_ah_bh_{ab}^{-1}$.
  \end{enumerate}
\end{definition}
The complex of groups $G(\mc{Y})$ can be used to recover the group $G$.  There are two different ways of doing this.  The first is explained in \cite[III.$\mc{C}$.3.7]{bridhaef:book}, and involves {\em $G(\mc{Y})$-paths}.  The second way is from \cite[III.$\mc{C}$.A]{bridhaef:book}, and is the way that we proceed.
The advantage to this second way, which uses categories and coverings of categories, is that lifting paths to covers is a canonical procedure (as with usual covering theory).

\subsection{Fundamental groups and coverings of categories}\label{ss:review}

In Definition \ref{def:elem homotopy} we defined homotopy of $\mc{C}$--paths, where $\mc{C}$ is a category.
\begin{definition}
 Given a category $\mc{C}$ and an object $v_0$ of $\mc{C}$, the \emph{fundamental group of $\mc{C}$ based at $v_0$}, denoted $\pi_1(\mc{C},v_0)$, is the set of homotopy classes of $\mc{C}$--loops based at $v_0$, with operation induced by concatenation of $\mc{C}$--paths.
\end{definition}

\begin{definition} \cite[III.$\mc{C}$.A.15]{bridhaef:book}. 
Let $\mc{C}$ be a connected category.  A functor $f \co \mc{C}' \to \mc{C}$ is a \emph{covering} if for each object $\sigma'$ of $\mc{C}'$ the restriction of $f$ to the collection of arrows that have $\sigma'$ as their initial (respectively, terminal) object is a bijection onto the set of arrows which have $f(\sigma')$ as their initial (resp., terminal) object.  
\end{definition}

The \emph{universal cover} $\widetilde{\mc{C}}$ of a connected category $\mc{C}$ is described in \cite[III.$\mc{C}$.A.19]{bridhaef:book}: Fix a base vertex $v_0$ of $\mc{C}$, and define $\mathrm{Obj}(\widetilde{\mc{C}})$ to be the set of homotopy classes of $\mc{C}$--paths starting at $v_0$.  If $[c]$ is a homotopy class of path, and $\alpha$ is an arrow from $t(c)$, then there is an arrow $\widetilde{\alpha}$ of $\widetilde{\mc{C}}$ from $[c]$ to $[c\cdot \alpha^-]$.  The projection $\pi \co \widetilde{\mc{C}} \to \mc{C}$ sets $\pi\left( [p] \right) = t(p)$ and if $\widetilde{\alpha}$ is the arrow described above then $\pi(\widetilde{\alpha}) = \alpha$.  The fundamental group $\pi_1(\mc{C},v_0)$ acts on $\widetilde{\mc{C}}$ by pre-concatenation.  

The theory of coverings of categories is entirely analogous to ordinary covering theory.  In fact it is a special case, as the coverings of a connected category $\mc{C}$ correspond bijectively to the covering spaces of its nerve.

We record the following observation.
\begin{lemma}\label{l:lift paths and homotopies}
  Let $\phi\co \widetilde{\mc{C}}\to \mc{C}$ be a covering of categories, and suppose $\phi(\widetilde{v})=v$, for objects $v$ of $\mc{C}$ and $\widetilde{v}$ of $\widetilde{\mc{C}}$.  Any $\mc{C}$--path $p$ with $i(p) = v$ has a unique lift to a $\widetilde{\mc{C}}$--path $\widetilde{p}$ with $i(\widetilde{p})=\widetilde{v}$.  Moreover any elementary homotopy from $p$ to a path $p'$ gives a unique elementary homotopy of $\widetilde{p}$ to a lift $\widetilde{p}'$ of $p'$ with the same endpoints as $\widetilde{p}$.
\end{lemma}

\subsection{The category associated to a complex of groups} \label{ss:category}
Any complex of groups $G(\mc{Y})$ has an associated category $CG(\mc{Y})$.

\begin{definition} \cite[III.$\mc{C}$.2.8]{bridhaef:book} \label{def:CG(Y)}
The objects of $CG(\mc{Y})$ are the objects of the scwol $\mc{Y}$.  Arrows of $CG(\mc{Y})$ are pairs $(g, a)$ so that $a$ is an arrow of $\mc{Y}$ and $g\in G_{t(a)}$.   Composition is defined by $(g, a)\circ(h, b)=(g\psi_a(h)z(a,b), ab)$.  
 \end{definition}
Recall that if $a$ is a trivial arrow then $\psi_a$ is the identity homomorphism and $z(a,x)$, $z(x,a)$ are always trivial.

\begin{remark}\label{remark:section}
The map $CG(\mc{Y}) \to \mc{Y}$ given by $(g,a) \to a$ is the scwolification functor (see Definition~\ref{def:simplescwol}), and is a bijection on objects.  This functor has an obvious section $a \mapsto (1,a)$.  If there are nontrivial twisting elements this is not a functor, but it does allow $\mc{Y}$--paths to be ``unscwolified'' to $CG(\mc{Y})$--paths.  In Definition \ref{def:scwolify}, we explain how to go back and forth between paths in covers of $CG(\mc{Y})$ and their associated scwols.
\end{remark}

\begin{theorem} \cite[III.$\mc{C}$.3.15 and text before III.$\mc{C}$.A.13]{bridhaef:book} \label{th:G is pi_1 of cat}
  Suppose that the group $G$ acts on the simply connected complex $X$, giving rise to an action of $G$ on the scwol $\mc{X}$ as in Definition~\ref{def:actionchoices}, and that $v_0$ is an object in $\mc{Y} = \leftQ{\mc{X}}{G}$.  Let $CG(\mc{Y})$ be the category associated to $G(\mc{Y})$.  Then there is an isomorphism from $\pi_1(CG(\mc{Y}),v_0)$ to $G$ taking any loop of the form $(g_1,a_1)^{\epsilon_1}\cdots (g_n,a_1)^{\epsilon_1}$ to the product $(g_1h_{a_1})^{\epsilon_1}\cdots (g_nh_{a_1})^{\epsilon_n}$.
\end{theorem}
The exponents $\epsilon_i$ in the statement are taken from the set $\{+,-\}$.  We use the mild abuse of notation that if $g$ is a group element then $g^+ = g$ and $g^- = g^{-1}$.
\begin{definition}\label{def:group and scwol arrows}
  Let $a$ be a nontrivial arrow of $\mc{Y}$.  The arrow $(1,a)$ of $CG(\mc{Y})$ is called a \emph{scwol arrow}.  Let $g\in G_v$ where $v$ is a vertex of the scwol $\mc{Y}$.  The arrow $(g,\identity_v)$ is called a \emph{group arrow at $v$}, or just a \emph{group arrow} if $v$ is unimportant.
  
In later sections we abuse notation and refer to the edge $(g,\identity_v)^+$ (for a group arrow $(g,\identity_v)$) as ``$(g,v)$'' or even just ``$g$''.  We also blur the difference between the scwol arrow $(1,a)$ and the $\mc{Y}$--arrow $a$, and often refer to the scwol arrow by ``$a$''. We also blur the distinction between the $CG(\mc{Y})$--edge $(1,a)^\pm$ and the $\mc{Y}$--edge $a^\pm$.
\end{definition}

\begin{definition}\label{d:label}
  Let $\mc{C}\to CG(\mc{Y})$ be a covering of categories.    We say that an arrow is \emph{labeled by $(g,a)$} if its image in $CG(\mc{Y})$ is $(g,a)$.
  An arrow of $\mc{C}$ is said to be a \emph{scwol (resp. group) arrow} if its label is a scwol (resp. group) arrow of $CG(\mc{Y})$.
\end{definition}

\begin{lemma}\label{lem:nice_paths}
  If $\mc{C}\to CG(\mc{Y})$ is any cover, then every $\mc{C}$--path is homotopic to a concatenation of group and scwol arrows.  
\end{lemma}
\begin{proof}
  Observe that any $CG(\mc{Y})$--arrow $(g,a)$ is a composition of a group arrow and a scwol arrow: $(g,a) = (g,t(a))\circ (1,a)$.  This gives a homotopy in $CG(\mc{Y})$ to a path of the desired form.
  Lemma \ref{l:lift paths and homotopies} says that the homotopy lifts.
\end{proof}

As described at the end of the last subsection, a choice of base vertex $v_0$ determines a universal covering map
$ \cgyt \to CG(\mc{Y}) $
sending a homotopy class of path $[p]$ to its terminal vertex $t(p)$, and the arrow from $[c]$ to $[c\cdot (g,a)^-]$ to the arrow $(g,a)$ of $CG(\mc{Y})$.  

The group $\pi_1(CG(\mc{Y}),v_0) \cong G$ acts on the universal cover $\cgyt$ by pre-concatenation of paths.  The quotient by this action is $CG(\mc{Y})$.
If $H<\pi_1(CG(\mc{Y}),v_0)$ is any subgroup, then $\smallleftQ{\cgyt}{H}$ is an intermediate cover of categories.  Every connected cover of $CG(\mc{Y})$ is of this form.  Indeed, covers of $CG(\mc{Y})$ correspond to coverings of the nerve $N$ of $CG(\mc{Y})$.  Since $\pi_1(CG(\mc{Y},v_0))$ is canonically isomorphic to $\pi_1(N,v_0)$, connected covers of $CG(\mc{Y})$ are all of the form $\smallleftQ{\cgyt}{H}$ for $H < \pi_1(CG(\mc{Y}),v_0)$.  The isomorphism $\pi_1(CG(\mc{Y}),v_0)\cong G$ from Theorem~\ref{th:G is pi_1 of cat} allows us to identify such $H$ as subgroups of $G$.

\begin{proposition}\label{prop:recoverX}
  Suppose that $CG(\mc{Y})$ arises from an action of $G$ on a simply connected scwol $\mc{X}$ via Definitions~\ref{def:actionchoices} and~\ref{def:CG(Y)}.  Let $\phi\co \pi_1(CG(\mc{Y}),v_0)\to G$ be the isomorphism from Theorem~\ref{th:G is pi_1 of cat}.  There is a $\phi$--equivariant functor $\Theta\co \cgyt\to \mc{X}$ which factors through an isomorphism of categories $\hat\Theta\co \scwol(\cgyt)\to \mc{X}$.
\end{proposition}
\begin{proof}[Proof sketch]
  Consider an arrow $x$ labeled by $(g,a)$ from $[\sigma_1]$ to $[\sigma_2]$ in $\cgyt$.  We may suppose
  \[ \sigma_1 = (g_1,a_1)^{\epsilon_1}\cdots (g_n,a_n)^{\epsilon_n}\quad \mbox{and}\quad  \sigma_2 = (g_1,a_1)^{\epsilon_1}\cdots (g_n,a_n)^{\epsilon_n}\cdot (g,a)^-.\]
  We define
  \begin{equation*}
    \Theta(x) = \prod_{i = 1}^n (g_i h_{a_i})^{\epsilon_i} \tilde{a}
  \end{equation*}
  Examining the elementary homotopies, it is not hard to see that $\Theta(x)$ is well-defined.
  
  The map $\phi$ sends the homotopy class of the loop $(h_1,b_1)^{\delta_1}\cdots (h_k,b_k)^{\delta_k}$ to the group element
  \[(h_1h_{b_1})^{\delta_1}\cdots (h_kh_{b_k})^{\delta_k}.\]   Since $\pi_1(CG(\mc{Y}),v_0)$ acts by pre-concatenation of paths, $\Theta$ is clearly $\phi$--equivariant.
   
  To see $\Theta$ is a functor, suppose that $x = y z$ in $\cgyt$, where $x,y,z$ are labeled by $(g,a), (h,b), (k,c)$, respectively, and $i(x)=i(z) = [(g_1,a_1)^{\epsilon_1}\cdots(g_n,a_n)^{\epsilon_n}]$.  Letting $p = \prod_{i=1}^n(g_i h_{a_i})^{\epsilon_i}$, we have $\Theta(x) = p \tilde{a}$, $\Theta(y) = p h_c^{-1}k^{-1}\tilde{b}$ and $\Theta(z) = p \tilde{c}$, and it is easily checked that $\Theta(y)\Theta(z)= \Theta(x)$.

  Any two objects of $\cgyt$ identified under the scwolification map are separated by a group arrow.
  If $x$ is a group arrow then $\Theta(x)$ is a trivial arrow,
  so $\Theta$ factors through a functor 
\[ \hat{\Theta}\co \scwol(\cgyt)\to \mc{X}. \]  
It remains to show that $\hat{\Theta}$ is an isomorphism of scwols.  The fact that the homomorphism from Theorem~\ref{th:G is pi_1 of cat} is surjective for any $v_0\in \mc{Y}$ implies that $\hat{\Theta}$ is also surjective, so the only difficult point is to see injectivity.

Assuming that $\Theta([\sigma]) = \Theta([\tau])$, we consider the images $\hat\sigma$ and $\hat\tau$ of the \emph{paths} $\sigma$ and $\tau$, respectively.  Since $\mc{X}$ is assumed to be simply connected, there is a sequence of elementary homotopies in $\mc{X}$ taking $\hat\sigma$ to $\hat\tau$.  It can be shown that these homotopies all lift (non-uniquely) to homotopies in $\cgyt$, so we may assume that $\hat\sigma = \hat\tau$.  If $\hat\sigma = \hat\tau$ is a degenerate path at $\tilde{v}_0$, then it is clear that $\sigma$ and $\tau$ are separated by a group arrow.  Otherwise, after a further homotopy in $\cgyt$, we can assume that
\begin{align*}
  \sigma & = (g_1,a_1)^{\epsilon_1}\cdots (g_n,a_n)^{\epsilon_n}\\
  \tau & = (h_1,a_1)^{\epsilon_1}\cdots (h_n,a_n)^{\epsilon_n}
\end{align*}
where the signs of the $\epsilon_i$ alternate and each arrow $a_i$ is non-trivial.  If $k$ is the smallest index for which $g_k \ne h_k$ and $k<n$, one can find a short sequence of elementary homotopies taking $\sigma$ to another path $\sigma'$ which agrees with $\tau$ for the first $k$ edges.  If $k = n$, then there is a slightly shorter sequence of elementary homotopies taking $\sigma$ to $\tau\cdot x$ for some group arrow $x$.

To sum up, if $\Theta([\sigma]) = \Theta([\tau])$, then there is a group arrow from $[\sigma]$ to $[\tau]$.  It follows that the map $\hat\Theta\co \scwol(\cgyt)\to \mc{X}$ is injective, and hence an isomorphism.
\end{proof}

The map $\Theta$ also passes to quotients by subgroups of $G$:
\begin{corollary}\label{cor:scwolify quotient}
  Suppose that $CG(\mc{Y})$ arises from an action of $G$ on a simply connected scwol $\mc{X}$ via Definitions~\ref{def:actionchoices} and~\ref{def:CG(Y)}, and that $\mc{Y} = \leftQ{\mc{X}}{G}$ is simple.
    If $H < G$, the scwolification of $\smallleftQ{\cgyt}{H}$ is canonically isomorphic to $\leftQ{\mc{X}}{H}$.
\end{corollary}
\begin{proof}
  For any small category $\mc{C}$, acted on by a group $H$, there is a canonical surjective functor $\smallleftQ{\scwol{\mc{C}}}{H}\to \scwol(\smallleftQ{\mc{C}}{H})$.  This is an isomorphism if and only if $\smallleftQ{\scwol{\mc{C}}}{H}$ is already a simple scwol.  Since $\mc{Y} = \smallleftQ{\mc{X}}{G}$ is a simple scwol, the intermediate quotient $\smallleftQ{\mc{X}}{H}$ is also a simple scwol.
  
  The naturality of scwolification (Remark~\ref{rem:scwolify is natural}) and the equivariance of $\Theta$ together mean that the map $\hat{\Theta}$ from Proposition~\ref{prop:recoverX} is also equivariant, and so we get an isomorphism
  \[ \hat\Theta_H\co \smallleftQ{\scwol(\cgyt)}{H}\to \smallleftQ{\mc{X}}{H} .\]

  Thus $\scwol\left(\smallleftQ{\cgyt}{H}\right)  =  \smallleftQ{\scwol(\cgyt)}{H}$   is isomorphic to $\smallleftQ{\mc{X}}{H}$.
\end{proof}

\begin{remark}
  Although Corollary~\ref{cor:scwolify quotient} does not explicitly appear in Bridson--Haefliger, it can be derived from results there as we outline briefly in this remark.

Suppose  $\mc{C'} = \smallleftQ{\cgyt}{H} \to CG(\mc{Y})$ is a cover.  According to \cite[Proposition III.$\mc{C}$.A.24]{bridhaef:book} there is an associated (category of a) complex of groups $CG(\mc{Y'})$, which is a sub-category of $\mc{C'}$ so that the inclusion is an equivalence $CG(\mc{Y'}) \to \mc{C'}$.  The construction of $CG(\mc{Y'})$ from $\mc{C'}$ is as in \cite[Proposition~III.$\mc{C}$.A.4]{bridhaef:book}, which uses the same equivalence relation as the definition of the scwolification functor as in Definition~\ref{def:simplescwol}.
From the construction of $CG(\mc{Y'})$, and the assumption that $\mc{Y}$ is a simple scwol, it follows that $\mc{Y'}$ is isomorphic to $\scwol\left( \mc{C'} \right)$.  Further, from the correspondence between coverings and subgroups (both for categories and also for complexes of groups), and the identification of $G$ with $\pi_1(CG(\mc{Y}))$ as in Theorem~\ref{th:G is pi_1 of cat}, it follows that $\mc{Y}'$ is isomorphic to $\leftQ{\mc{X}}{H}$, as required.  

Rather than fully develop this approach, we chose to sketch a more direct approach using Proposition~\ref{prop:recoverX} and the naturality of the scwolification functor.
\end{remark}

\begin{definition}\label{def:Theta}
  We denote the scwolification functor from $\smallleftQ{\cgyt}{H}$ to $\leftQ{\mc{X}}{H}$ by $\Theta_H$.  If $H = \{1\}$, we just write $\Theta$, as in Proposition~\ref{prop:recoverX}.
\end{definition}
We observe the following.
\begin{lemma}\label{lem:TransitiveStabilizer}
    Suppose that $CG(\mc{Y})$ arises from an action of $G$ on a simply connected scwol $\mc{X}$ via Definitions~\ref{def:actionchoices} and~\ref{def:CG(Y)}, and that $\mc{Y} = \leftQ{\mc{X}}{G}$ is simple.

    Let $\o$ be an object of $\mc{X}$.  Then $\Stab_G(\o)$ acts freely and transitively on $\Theta^{-1}(\o)$.
\end{lemma}

\begin{definition}\label{def:scwolify}
Let $K < G \cong \pi_1(CG(\mc{Y}),v_0)$, let $\mc{C}_K = \smallleftQ{\cgyt}{K}$ and let $\mc{Z} = \leftQ{\mc{X}}{K}$.  Since $\Theta_K \co \mc{C}_K \to \scwol(\mc{C}_K)= \mc{Z}$ is a functor, it gives a way to turn a $\mc{C}_K$--path $p$ into a $\mc{Z}$--path $p'$.  Deleting all the trivial arrows from $p'$ produces a $\mc{Z}$--path which we call the \emph{scwolification} of $p$.  Abusing the notation slightly, we denote the scwolification of $p$ by $\Theta_K(p)$.
  
  Conversely, if $\sigma$ is a $\mc{Z}$--path, then any $\mc{C}_K$--path $\widehat{\sigma}$ so that $\Theta_K(\widehat{\sigma}) = \sigma$ is called an \emph{unscwolification} of $\sigma$.  The unscwolification is highly non-unique, but always exists.
\end{definition}
The following can be deduced by examining the elementary homotopies.
\begin{lemma}
  Scwolifications of homotopic paths are homotopic.
\end{lemma}

Given a $CG(\mc{Y})$--path $p$ we can lift it to a $\mc{C}_K$--path $\widehat{p}$, and then scwolify $\widehat{p}$ to the $\mc{Z}$--path $\Theta_K(\widehat{p})$.

\begin{lemma} Let $p$ be a $CG(\mc{Y})$--loop at $v$ and let $\widehat{p}$ be a lift to $\mc{C}_K$.  If $\Theta_K(\widehat{p})$ is a loop, then there is a group arrow labeled by an element of $G_v$ joining the endpoints of $\widehat{p}$.\end{lemma}

\subsection{The complex of groups coming from an action on a cube complex}
Let $X$ be a CAT$(0)$ cube complex, and suppose that $G$ acts on $X$ by cubical automorphisms.  The quotient $\leftQ{X}{G}$ may or may not be a cube complex, depending on whether the groups $G_\sigma=\{g\mid g\sigma=\sigma\}$ and $\{g\mid gx=x, \forall x\in \sigma\}$ agree for all cells $\sigma$.

Another way to phrase this issue is to note that, if $\mc{X}_0$ is the scwol of cells of $X$, then $G$ acts by morphisms on $\mc{X}_0$, but the quotient map  $\mc{X}_0\to \leftQ{\mc{X}_0}{G}$ may not be a non-degenerate morphism of scwols, since some isometry of $X$ may fix the center of some cube, but permute faces of that cube.
In order to obtain a complex of groups structure on $G$ from the action $G\acts X$, we need a scwol quotient, so we replace $\mc{X}_0$ with $\mc{X}$, the scwol of cells of the first barycentric subdivision of $X$:

\begin{definition}\label{def:cube idealization}
  If $W$ is a cube complex, the \emph{idealization of $W$} is the scwol of cells of the first barycentric subdivision of $W$.
\end{definition}
If every cube of the cube complex $W$ embeds in $W$, there is another way of thinking of the idealization $\mc{W}$.  Namely, the objects of $\mc{W}$ are in one-to-one correspondence with non-empty nested chains of cubes of $W$ and there is at most one arrow in $\mc{W}$ between two objects: If $c_1$ contains $c_2$ as a sub-chain, there is an arrow from $c_1$ to $c_2$.

For example, if $X$ is a single $1$--dimensional cube $e$ with endpoints $a$ and $b$, the nontrivial arrows of the idealization $\mc{X}$ are as follows:
\begin{equation}\label{eq:edgescwol}
\begin{tikzcd}[row sep = small]
  (a) & (a\subset e) \arrow[r]\arrow[l] & (e) & (b\subset e) \arrow[r]\arrow[l] & (b)  
\end{tikzcd}
\end{equation}
Already a square $\tau$ with $e$ as a face is much more complicated.  The idealization is shown on the left of Figure~\ref{fig:idealsquare} as a graph, with detail shown on the right for the highlighted portion.
\begin{figure}[htbp]
\[
\xy
(30,0)*{\xy \xymatrix{
  (a) & & (a\subset e)\ar[r]\ar[ll] & (e) \\
  & & (a\subset e \subset \tau)\ar[d]\ar[ull]\ar[u]\ar[r]\ar[ur]\ar[ddr] & (e\subset \tau)\ar[u]\ar[dd]\\
  & & (a\subset \tau)\ar[uull]\ar[dr] & \\
  & & & (\tau)
  }
  \endxy};
(-30,0)*{
  \xy
  \xyimport(2,2){\includegraphics{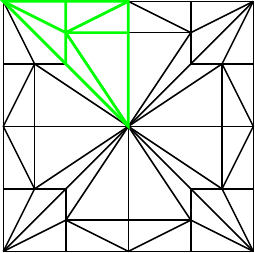}}
  \endxy};

\endxy
\]
  \caption{Idealization of a square.}
  \label{fig:idealsquare}
\end{figure}

Let $X$ be a CAT$(0)$ cube complex, and let $\mc{X}$ be its idealization.
 Any cubical automorphism of $X$ gives a non-degenerate automorphism of $\mc{X}$.  
 Moreover if such an automorphism maps a chain of cubes to itself, then it also preserves all subchains.  In terms of the scwol structure this means that the stabilizer of an object also stabilizes every arrow from that object.
 It follows that the quotient $\mc{Y} = \leftQ{\mc{X}}{G}$ is a simple scwol, and that the quotient map $\mc{X} \to \mc{Y}$ is non-degenerate morphism of scwols.  Similarly, if $K<G$ is any subgroup, then $\leftQ{\mc{X}}{K}$ is a simple scwol and the maps $\mc{X}\to \leftQ{\mc{X}}{K} \to \mc{Y}$ are non-degenerate morphisms of scwols.  In particular Corollary~\ref{cor:scwolify quotient} applies to the covers of $CG(\mc{Y})$.

 \begin{example}\label{ex:CGY}
   Let $X$ be the regular $3$--valent tree dual to the Farey tessellation, and let \[G = \operatorname{PSL}(2,\bZ) \cong \langle x\mid x^2\rangle * \langle y \mid y^3\rangle\] act on $X$ in the standard way, with $y$ rotating around a vertex $v$, and $x$ rotating around the center of an adjacent edge $e$.  Then $CG(\mc{Y})$ has the following form, omitting identity arrows and most arrow labels.
   \begin{equation}\label{eq:psl2Z}
     \begin{tikzcd}[row sep = small]
       (e)\arrow[loop left,"x"] &
       (v\subset e) \arrow[l]\arrow[l,bend right] \arrow[r]\arrow[r,bend left]\arrow[r,bend right]&
       (v) \arrow[loop above,"y"]\arrow[loop right,"y^2"]
     \end{tikzcd}
   \end{equation}
The two-fold cover corresponding to the subgroup generated by $\{y,xyx\}$ is of the following form, omitting identity arrows and all labels. 
\begin{equation}\label{eq:2foldcover}
  \begin{tikzcd}[row sep = small]
    \circ\arrow[loop below]\arrow[loop left] &
    \circ\arrow[l]\arrow[r]\arrow[dr] &
    \circ\arrow[d,bend right] & & & \\
    & &
    \circ\arrow[u,bend right] &
    \circ\arrow[ul]\arrow[l]\arrow[r] &
    \circ\arrow[loop above]\arrow[loop right]
  \end{tikzcd}
\end{equation}
The scwolification of this cover is isomorphic to the scwol shown in~\eqref{eq:edgescwol}.  In both~\eqref{eq:psl2Z} and~\eqref{eq:2foldcover}, the scwol arrows are exactly the horizontal ones.

This example is one-dimensional, so there are no nontrivial compositions of scwol arrows.  The reader is invited to explore a simple two-dimensional example, for example the category of the complex of groups associated to a dihedral group acting on a square with quotient scwol equal to the right-hand part of Figure~\ref{fig:idealsquare}.
 \end{example}

  \begin{remark}\label{rem:idealizing paths}
   Suppose that $\mc{C}$ is the idealization of a cube complex $C$, so that the realization of $\mc{C}$ is the second barycentric subdivision of $C$.  In later sections, we make use of the fact that the following types of paths have canonical idealizations in $\mc{C}$:
   \begin{enumerate}
   \item Combinatorial paths in the $1$--skeleton of the first \emph{cubical} subdivision $C^b$ of $C$ (Section \ref{s:QC proof}).
   \item Combinatorial paths in links of cells of $C$ (Section \ref{s:X mod K CAT(0)}).
   \end{enumerate}
   In both cases, this follows from the fact that subdivisions of these graphs embed naturally in the $1$--skeleton of the second barycentric subdivision.
 \end{remark}

\section{Quasi-convexity in the Sageev construction} \label{s:QC proof}

In this section, we prove Theorem \ref{t:cell qc iff hyp}.  Recall that we have a hyperbolic group $G$ acting cocompactly on a CAT$(0)$ cube complex $X$, and we are required to prove that the vertex stabilizers are quasi-convex if and only if the hyperplane stabilizers are quasi-convex.  

To prepare for this proof it may be useful to think about the case that $X$ is a tree.  In that case, hyperplanes are midpoints of edges, and so the statement is that edge stabilizers are quasi-convex if and only if vertex stabilizers are.  Edge stabilizers are intersections of vertex stabilizers, and intersections of quasi-convex subgroups are quasi-convex, so one direction is clear.  The other direction is not much harder:  Consider a geodesic joining two vertices of a vertex stabilizer.  The vertex stabilizer is coarsely separated from the rest of the Cayley graph by appropriate cosets of edge stabilizers.  The quasi-convexity of these cosets ``traps'' the geodesic close to the vertex stabilizer.

Now remove the assumption that $X$ is a tree, and suppose that vertex stabilizers are quasi-convex.  It still follows that edge stabilizers are quasi-convex, but a hyperplane stabilizer is much bigger than an edge stabilizer.  We will express a hyperplane stabilizer as a union of cosets of edge stabilizers, intersecting in a controlled way, and use a quasi-convexity criterion proved in Appendix \ref{app:QC} to conclude that the hyperplane stabilizer is quasi-convex.

If on the other hand we assume that hyperplane stabilizers are quasi-convex, we will use them as in the tree case to control geodesics joining points in a vertex stabilizer.  We inductively use more and more hyperplanes to corral points on a geodesic in an argument which terminates because of the finite dimensionality of the cube complex.

Throughout this section we suppose that $X$ is a CAT$(0)$ cube complex and that $\mc{X}$ is its idealization (see Definition \ref{def:cube idealization}).  We suppose further that $G$ is a group acting cocompactly on this cube complex.  The quotient $\leftQ{\mc{X}}{G}$ is a scwol $\mc{Y}$.
Making choices as in Definition \ref{def:actionchoices}, we obtain a complex of groups $G(\mc{Y})$, with associated category $CG(\mc{Y})$.  Choosing a vertex $v_0\in \mc{Y}$, Theorem~\ref{th:G is pi_1 of cat} gives an identification of $G$ with $\pi_1(CG(\mc{Y}),v_0)$.  It is helpful to assume (as we may do without loss of generality) that $v_0$ is the orbit of some $0$--cube of $X$.

In Section \ref{ss:review}, we defined $\cgyt$ to be the universal covering of the category $CG(\mc{Y})$.  Recall that the objects of $\cgyt$ are homotopy classes of $CG(\mc{Y})$--paths, starting at the basepoint $v_0\in \mc{Y}$, and arrows are labeled by arrows of $CG(\mc{Y})$ (Definition \ref{d:label}).    The basepoint of $\cgyt$ is $\widetilde{v}_0$, the homotopy class of the constant path at $v_0$.  As in Section~\ref{ss:review} we denote the universal covering map by $\phi \co \cgyt \to CG(\mc{Y})$.  Recall from Proposition~\ref{prop:recoverX} and Definition~\ref{def:Theta} that $\Theta \co \cgyt \to \mc{X}$ is the scwolification functor.

We briefly describe the contents of the remainder of this section.  In Subsection \ref{ss:graphcategory} we explain how we consider subsets of small categories as graphs.  In Subsection \ref{ss:cubical paths} we identify certain subsets of $\cgyt$ which are tuned to the cubical geometry of $X^b$ and associate graphs with these subsets.  In Subsection \ref{ss:(2) implies (1)} we prove the direction \eqref{hyp QC} $\implies$ \eqref{vertex QC}  of Theorem \ref{t:cell qc iff hyp}.  In fact, we prove the more general Theorem \ref{t:strong QC (2) implies (1)}.  In Subsection \ref{ss:(1) implies (2)} we prove the direction \eqref{vertex QC} $\implies$ \eqref{hyp QC} of Theorem \ref{t:cell qc iff hyp}.  In Subsection \ref{ss:generalizations} we consider various possible generalizations of Theorem \ref{t:cell qc iff hyp}.  Finally, in Subsection \ref{ss:Cor B} we prove Corollary \ref{cor:cube complex hyp}.

\subsection{Graphs from subsets of small categories}\label{ss:graphcategory}

  Let $\mc{C}$ be a (small) category, and let $S$ be a subset of the set of arrows of $\mc{C}$.  There is an associated graph (really a $1$--complex), which we denote $\Gr(S)$, with vertex set the set of objects which are either the source or target of some arrow in $S$, and with edges in correspondence with the arrows $S$.

\begin{example}\label{ex:grC}
    Let $S$ be the set of \emph{all} arrows in $\mc{C}$.  Then $\Gr(\mc{C}):=\Gr(S)$ is the $1$--skeleton of the nerve of $\mc{C}$.
\end{example}
\begin{example}    Suppose $\mc{C}$ is a group (i.e. $\mc{C}$ has a single object and each arrow of $\mc{C}$ is invertible), and $S_0\subset\mc{C}$ is a generating set.
    Let $S$ be the set of arrows in the universal cover $\widetilde{\mc{C}}$ with label in $S_0$.  Then $\Gr(S)$ is the Cayley graph of $G$ with respect to $S_0$.
  \end{example}

\subsection{Cubical paths} \label{ss:cubical paths}
The group $\pi_1(CG(\mc{Y}),v_0)$ acts on $\cgyt$, with quotient the category $CG(\mc{Y})$.  As in Theorem \ref{th:G is pi_1 of cat}, we can identify $G$ with $\pi_1(CG(\mc{Y}),v_0)$.  The proof of each direction of Theorem \ref{t:cell qc iff hyp} begins with choosing a certain connected $G$--cocompact subgraph $\Gamma$ of $\Gr(\cgyt)$ (see Lemma~\ref{lem:GammaA}). 
The graph $\Gamma$ admits a $G$--equivariant map to the $1$--skeleton of the cubical subdivision of $X$, which we now recall.

\begin{definition}
Suppose that $X$ is a cube complex.  The {\em (first) cubical subdivision} of $X$, denoted $X^b$, is the cube complex obtained by replacing each $n$--cube in $X$ by $2^n$ $n$--cubes, found by subdividing each coordinate interval into two equal halves, and then gluing in the obvious way induced from the structure of $X$.
\end{definition}
Of course, $X^b$ is canonically homothetic to $X$, and $X^b$ is NPC (respectively, CAT$(0)$) if and only if $X$ is.  We suppose that $X$ is CAT$(0)$, and therefore $X^b$ is also.  




\begin{observation} \label{obs:X^b}
The vertices of $X^b$ are in bijection with the cubes of $X$.

The cells of $X^b$ are in bijection with pairs $(\widetilde{\sigma}_1,\widetilde{\sigma}_2)$ of cubes in $X$ so that $\widetilde{\sigma}_1 \subseteq \widetilde{\sigma}_2$.  The dimension of the cube corresponding to $(\widetilde{\sigma}_1,\widetilde{\sigma}_2)$ is $\dim(\widetilde{\sigma}_2) - \dim(\widetilde{\sigma}_1)$.

Thus, a $1$--cell in $X^b$ corresponds to a pair of cubes $(\widetilde{\sigma}_1, \widetilde{\sigma}_2)$ where $\widetilde{\sigma}_1$ is a codimension one face of $\widetilde{\sigma}_2$.  Moreover, each cell of $X^b$ can be naturally identified with an object of $\mc{X}$.
\end{observation}

As noted in Remark \ref{rem:idealizing paths} any path in the $1$--skeleton of $X^b$ has a canonical idealization in $\mc{X}$.  Each $1$--cell $e$ of $X^b$ corresponds to some pair of cells $(\widetilde{\sigma}_1 \subseteq \widetilde{\sigma}_2)$ with $\widetilde{\sigma}_1$ of codimension one in $\widetilde{\sigma}_2$.  If the path $p$ traverses the $1$--cell $e$, its idealization $\widehat{p}$ contains consecutive arrows labeled  $(\widetilde{\sigma}_1 \subseteq \widetilde{\sigma}_2) \to \widetilde{\sigma}_1$
and
$(\widetilde{\sigma}_1 \subseteq \widetilde{\sigma}_2) \to \widetilde{\sigma}_2$,
and every arrow of $\widehat{p}$ has such a label.
By Lemma~\ref{lem:nice_paths}, every $\cgyt$--path is homotopic to a concatenation of group arrows and scwol arrows.  The graph that we use to prove Theorem~\ref{t:cell qc iff hyp} uses only scwol arrows that occur in pairs corresponding to the above description.  Thus we make the following definition.

\begin{definition}
A \emph{pair of opposable scwol arrows in $CG(\mc{Y})$} is a pair of scwol arrows $\gamma^\down, \gamma^\up$ so that 
\begin{enumerate}
\item $c = i(\gamma^\down) = i(\gamma^\up)$ is an orbit of chain $(\widetilde{\sigma}_1 \subset \widetilde{\sigma}_2)$ where $\widetilde{\sigma}_1$ has codimension one in $\widetilde{\sigma}_2$;
\item  $\gamma^\down = (1,a)$ where $a$ is the arrow in $\mc{Y}$ corresponding to the $G$--orbit of the arrow $(\widetilde{\sigma}_1 \subset \widetilde{\sigma}_2) \to \widetilde{\sigma}_1$ in $\mc{X}$; and
\item $\gamma^\up = (1,b)$, where $b$ is the arrow in $\mc{Y}$ corresponding to the $G$--orbit of the arrow $(\widetilde{\sigma}_1 \subset \widetilde{\sigma}_2) \to \widetilde{\sigma}_2$ in $\mc{X}$.
\end{enumerate}
Suppose $c =i(\gamma^\down) = i(\gamma^\up)$ for a pair of opposable scwol arrows $(\gamma^\down,\gamma^\up)$.
If $\tilde{c}$ is a lift of $c$ to $\cgyt$, there are unique lifts $\tilde{\gamma}^\down,\tilde{\gamma}^\up$ with source $\tilde{c}$.  The pair $(\tilde{\gamma}^\down,\tilde{\gamma}^\up)$ is a \emph{pair of opposable scwol arrows in $\cgyt$}.
\end{definition}
We remark that the image under $\Theta$ of a pair of opposable scwol arrows is quite restricted.  In the right-hand part of Figure~\ref{fig:idealsquare}, for example, the possible images are the horizontal pair of arrows at the top of the diagram or the vertical pair at the right.

\begin{definition}
An object in $CG(\mc{Y})$ (equivalently, in $\mc{Y}$, since the objects of these two categories are the same) is {\em cubical} if it is an orbit of cubes in $X$ (rather than an orbit of chains of cubes of length greater than $1$).  An object in $\cgyt$ is {\em cubical} if its projection to $CG(\mc{Y})$ is cubical.

A path $p$ in $CG(\mc{Y})$ is {\em cubical} if
\begin{enumerate}
\item The initial and terminal objects of $p$ are cubical;
\item $p$ is a concatenation of group arrows and scwol arrows; and
\item The scwol arrows occur in consecutive pairs, as pairs of opposable scwol arrows.
\end{enumerate}

A path in $\cgyt$ is {\em cubical} if its projection to $CG(\mc{Y})$ is cubical.
\end{definition}
It follows from the definition that all group arrows for a cubical path occur at cubical objects.

\begin{proposition}\label{cubicalpaths}
Suppose that $v$ and $w$ are cubical vertices of ${CG}(\mc{Y})$ and $\sigma$ is a  $CG(\mc{Y}$)--path between $v$ and $w$.  Then $\sigma$ is homotopic to a cubical path.
 
In particular, every $g\in G=\pi_1(CG(\mc{Y}),v_0)$ is represented by a cubical $CG(\mc{Y})$--loop starting and ending at $v_0$.
\end{proposition}
\begin{proof}
By path lifting, it suffices to prove the analogous statement for $\cgyt$--paths.
We already observed in the proof of Proposition~\ref{prop:recoverX} that homotopies in $\mc{X}$ can be lifted to homotopies of $\cgyt$--paths, thus a given path can be homotoped to a path whose image under $\Theta$ stays in the idealization of $\left(X^b\right)^{(1)}$.  Lemma~\ref{lem:nice_paths} turns this into a concatenation of group arrows and scwol arrows (without changing its image under $\Theta$).  Any sub-path $(1,a)^+ \cdot (g,\identity_{i(a)})^{\pm}$ is homotopic to $(\psi(a)(g),\identity_{t(a)})^{\pm} \cdot (1,a)^+$.  In particular, after a homotopy, we may assume our path contains no such sub-path, and whenever a scwol arrow occurs it occurs in a pair of opposable scwol arrows.
\end{proof}

\begin{definition} \label{def:Gamma A}	
  Suppose that for each cubical object $\o$ of $\mc{Y}$ we choose a set $\mathbb{A}_\o \subset G_\o$.  These determine a subset $S(\mathbb{A})$ of the arrows of $\cgyt$ which is the union of the following two sets:
  \begin{enumerate}
  \item the set of (group) arrows with label $(g,\identity_\o)$ for some $\o$ and some $g\in \mathbb{A}_\o$; and
  \item the set of scwol arrows occurring in some pair of opposable scwol arrows.
\end{enumerate}
  As discussed in Section \ref{ss:graphcategory}, there is an associated graph $\Gr(S(\mathbb{A}))$ which we denote by $\Gamma(\mathbb{A})$.  A vertex of this graph is called \emph{cubical} if it comes from a cubical object, and otherwise it is called \emph{central}.
\end{definition}
Note that any central vertex of $\Gamma(\mathbb{A})$ only meets opposable scwol arrows and thus has valence exactly two, and each of its neighbors is a cubical vertex of $\Gamma(\mathbb{A})$.  The valence of a cubical vertex coming from the object $\o$ is equal to the number of opposable scwol arrows with terminus $\o$ plus twice the cardinality of $\mathbb A_\o$.  Since $\mc{Y}$ is finite, the graph $\Gamma(\mathbb A)$ is locally finite if and only if every $A_\o$ is finite.

\begin{lemma}\label{lem:GammaA}
  Suppose that $G=\pi_1(CG(\mc{Y}),v_0)$ is finitely generated.  Then there is a choice of $\mathbb A$ so that $\Gamma(\mathbb A)$ is locally finite, connected, and $G$--cocompact.
\end{lemma}
\begin{proof}
  Let $S$ be a finite generating set for $G$.  For each $s\in S$ choose a cubical loop $p(s)$ in $CG(\mc{Y})$ based at $v_0$ representing $s$.  For a cubical object $\o$ of $\mc{Y}$, let $\mathbb A_\o$ consist of those group arrows at $\o$ which occur in some $p(s)$.  There are only finitely many such, so the graph $\Gamma(\mathbb A)$ defined in Definition~\ref{def:Gamma A} is locally finite.

  To see $\Gamma(\mathbb A)$ is connected, let $w$ be any vertex of $\Gamma(\mathbb A)$.  There is a path composed of opposable scwol arrows joining $w$ to a vertex $v$ in the $G$--orbit of $\tilde{v}_0$.  Since $S$ generates $G$, there is a concatenation of the cubical loops $p(s)$ which lifts to a path in $\Gamma(\mathbb A)$ joining $\tilde{v}_0$ to $v$.
  
  The set of $G$--orbits of cubical vertices of $\Gamma(\mathbb A)$ injects into the set of objects of $\mc{Y}$, so it is finite and $\Gamma(\mathbb A)$ is $G$--cocompact.
\end{proof}

\begin{convention}
  For the rest of this section, we fix $\Gamma = \Gamma(\mathbb A)$ as in the conclusion of Lemma~\ref{lem:GammaA}.  
\end{convention}
The functor $\Theta \co \cgyt \to \mc{X}$ induces a map
\[	\Psi \co \Gamma \to \left(X^b \right)^{(1)}	.	\]
This map is simplicial after barycentrically subdividing the target.  Each pair of edges of $\Gamma$ coming from a pair of opposable scwol arrows maps to a single $1$--cell of $\left(X^b\right)^{(1)}$.  Any central vertex of $\Gamma$ maps under $\Psi$ to an intersection of an edge of $X^b$ with a hyperplane of $X^b$.

Note that the map $\Psi$ is continuous, $G$--equivariant and Lipschitz.

\begin{definition}[Cubical neighborhood]
  Let $v$ be a vertex of $X$.  The \emph{cubical neighborhood of $v$} is the union of those cubes of $X^b$ which contain $v$.  It will be denoted below by $N(v)$.
\end{definition}
\begin{proposition}\label{prop:N(v)} 
  Let $v$ be a vertex of $X$.  Then $\Psi^{-1}(N(v))$ is finite Hausdorff distance from $\Psi^{-1}(v)$ in $\Gamma$.
\end{proposition}
\begin{proof}
  Since $G$ acts cocompactly on $X$, there are finitely many $\Stab(v)$--orbits of pairs of cubes $(\sigma \subset \tau)$ so $\sigma$ contains $v$ and is a codimension one face of $\tau$.

  Let $(\sigma \subset \tau)$ be one such pair of cubes.
  By Lemma~\ref{lem:TransitiveStabilizer}, $\Stab(\tau)$ acts freely and transitively on $\Theta^{-1}((\tau))$.  The subgroup $\Stab(\sigma)\cap \Stab(\tau) = \Stab((\sigma\subset \tau))$ preserves the collection of pairs of opposable scwol arrows joining $\Theta^{-1}((\tau))$ to $\Theta^{-1}((\sigma))$.  Moreover $\Stab((\sigma \subset \tau))$ is finite index in $\Stab(\sigma)$.  It follows that there is some $c(\tau,\sigma)>0$ so that every vertex of $\Psi^{-1}((\tau))$ is distance at most $c(\tau,\sigma)$ from a vertex of $\Psi^{-1}((\sigma))$.  As there are only finitely many $\Stab(v)$--orbits of pairs $(\sigma \subset \tau)$ of such faces, there is some $c>0$ which works for every such pair.

  If $x\in \Psi^{-1}(N(v))$ is a cubical vertex, it is therefore distance at most $c \cdot \dim(X)$ from a vertex of $\Psi^{-1}(v)$.  If $x\in \Psi^{-1}(N(v))$ is central, then it is distance $1$ from a cubical vertex of $\Psi^{-1}(N(v))$.  We have shown that $\Psi^{-1}(N(v))$ is contained in the $(c \cdot \dim(X) + 1)$--neighborhood of $\Psi^{-1}(v)$.  Since $\Psi^{-1}(v)\subset \Psi^{-1}(N(v))$, we are finished.
\end{proof}
Part of our reason for working in $X^b$ is that Proposition~\ref{prop:N(v)} would fail if we defined $N(v)$ to be the union of those cubes of $X$ meeting $v$.

In the following statements we use the convention that the empty intersection of hyperplanes of $X$ is $X^b$ and the empty intersection of subgroups is $G$.
\begin{lemma} \label{lem:StabIinStabTau}
Suppose that $W_1, \ldots, W_k$ are  hyperplanes in $X$ and that $I = \bigcap\limits_{i=1}^k W_i$ is non-empty.  For any cell $\tau$ intersecting $I$ the subgroup $\Stab(I) \cap \Stab(\tau)$ is finite-index in $\Stab(\tau)$.
\end{lemma}
\begin{proof}
Finitely many hyperplanes intersect $\tau$ and these are permuted by any element of $\Stab(\tau)$.  Thus, a finite-index subgroup of $\Stab(\tau)$ fixes all the hyperplanes in $I$.
\end{proof}

We observe the following consequence of the cocompactness of $G\acts X$:
\begin{lemma} \label{lem:Int_Stab_Stab_Int}
  There are finitely many $G$--orbits of finite sets $$\{ W_1, \ldots , W_k \}$$ of distinct hyperplanes of $X$ with non-empty intersection $\bigcap\limits_{i=1}^k W_i$.  For each such set, the intersection $\bigcap\limits_{i=1}^k\Stab(W_i)$ is finite index in $\Stab\left(\bigcap\limits_{i=1}^k W_i\right)$.
\end{lemma}

\begin{definition}
  Let $D>0$.  An action of a group on a metric space is \emph{$D$--cobounded} if there is a set of diameter $D$ which meets every orbit.
\end{definition}

\begin{proposition} \label{prop:Stab(W)_on_Psi-1(W)}
  There is a constant $D>0$ so that for any non-empty intersection $I$ of hyperplanes of $X$, the subgroup $\Stab(I)$ acts $D$--coboundedly on $\Psi^{-1}(I)$.
\end{proposition}
\begin{proof}
  Since there are finitely many orbits of non-empty intersections of hyperplanes, it suffices to consider a single such intersection.
  
  Since the action of $G$ on $X$ is cocompact, so is the action of $\Stab(I)$ on $I$.  In particular, there are finitely many $\Stab(I)$ orbits of cubical or central vertices in the idealization of $I$.  Let $\o$ be the object of $\mc{X}$ corresponding to one of these vertices.  By Lemma~\ref{lem:TransitiveStabilizer}, $\Stab_G(\o)$ acts transitively on $\Theta^{-1}(\o)$.  By Lemma~\ref{lem:StabIinStabTau}, there is a finite index subgroup of $\Stab_G(\o)$ in $\Stab(I)$.  Thus we see that $\Psi^{-1}(I)$ contains finitely many $\Stab(I)$--orbits of vertices.  
\end{proof}

\subsection{If hyperplane stabilizers are QC then cell stabilizers are QC} \label{ss:(2) implies (1)}

In this section we prove the direction \eqref{hyp QC} $\implies$ \eqref{vertex QC} of Theorem \ref{t:cell qc iff hyp}. As mentioned in the introduction, we prove this in greater generality than that of a hyperbolic group acting cocompactly on a CAT$(0)$ cube complex with quasi-convex hyperplane stabilizers.  The right general setting for this proof is that of {\em strongly quasi-convex subgroups} of finitely generated groups, as defined by Tran in \cite{Tran}.  (Such subgroups were also studied by Genevois \cite{genevois:hyperbolicities} under the name {\em Morse subgroups}.)

\begin{definition} \cite[Definition 1.1]{Tran}
Let $X$ be a geodesic metric space.  A subset $Q\subseteq X$ is 
{\em strongly quasi-convex} if for every $K \ge 1$, $C \ge 0$ there is some $M = M(K,C)$ so that every $(K,C)$--quasi-geodesic in $X$ with endpoints in $Q$ is contained in the $M$--neighborhood of $Q$.  The function $M(K,C)$ is called a \emph{Morse gauge}.
\end{definition}

Strong quasi-convexity persists under quasi-isometries of pairs.
This is presumably known to the experts, and is closely related to \cite[Proposition 4.2]{Tran}, but we do not see it in the literature so we provide a proof sketch.

\begin{theorem} \label{t:QI of SQC}
Suppose $X$ and $Y$ are geodesic metric spaces, that $A \subset X$ is strongly quasi-convex, that $\phi \co X \to Y$ is a quasi-isometry and that $B \subset Y$ is finite Hausdorff distance from $\phi(A)$.  Then $B$ is a strongly quasi-convex subset of $Y$.
\end{theorem}
\begin{proof}[Proof sketch]
  This is proved essentially in the same way as the corresponding fact about quasi-convex subsets of hyperbolic spaces.  The difference is that instead of a single constant of quasi-convexity, we must produce a Morse gauge.

  Suppose that $\phi\co X\to Y$ and $\psi\co Y\to X$ are $(\lambda,\epsilon)$--quasi-isometries which are $\epsilon$--quasi-inverses, and that $d_{\mathrm{Haus}}(B,\phi(A))\leq \epsilon$.
  
    Any quasi-geodesic $\gamma$ joining points in $B$ can be extended by a pair of geodesic segments of length $\le\epsilon$ to make a quasi-geodesic $\gamma'$ joining points in $\phi(A)$.  The image of $\gamma'$ under $\psi$ can likewise be extended to a quasi-geodesic $\gamma''$ between points of $A$.  If $\gamma$ was a $(K,C)$--quasi-geodesic, then $\gamma''$ is a $(K',C')$--quasi-geodesic where $K',C'$ depend only on $K$, $C$, $\lambda$, and $\epsilon$.  If $M$ is the Morse gauge for $A$ in $X$, then let $M_1 = M(K',C')$.  For any point $p$ on $\gamma$, the point $\psi(p)$ is on $\gamma''$ so it is within $M_1$ of some point in $A$.  Using $\phi$ to move back to $X$ we see that $p$ is within $\lambda M_1 + 3\epsilon$ of some point of $B$.  We can therefore define a Morse gauge $M'$ for $B$ in $Y$ by $M'(K,C)=\lambda M(K',C')+3\epsilon$.
\end{proof}

In particular, the notion of strong quasi-convexity makes sense for subgroups of finitely generated groups.

In this subsection, we prove the following theorem.

\begin{theorem} \label{t:strong QC (2) implies (1)}
Suppose that a finitely generated group $G$ acts cocompactly on a CAT$(0)$ cube complex $X$ and that the hyperplane stabilizers are strongly quasi-convex.  Then the cell stabilizers are strongly quasi-convex.
\end{theorem}

Since quasi-convexity is equivalent to strong quasi-convexity for subgroups of hyperbolic groups, Theorem \ref{t:strong QC (2) implies (1)} immediately implies the direction \eqref{hyp QC} $\implies$ \eqref{vertex QC} of Theorem \ref{t:cell qc iff hyp}.

Note that each cell stabilizer is a finite intersection of vertex stabilizers.  Tran shows that a finite intersection of strongly quasi-convex subgroups is strongly quasi-convex (\cite[Theorem 1.2.(2)]{Tran}) so we only need to show that vertex stabilizers are strongly quasi-convex whenever hyperplane stabilizers are.

We will use the following general statement about intersections of strongly quasi-convex sets, analogous to \cite[III.$\Gamma$.4.13]{bridhaef:book}.

\begin{proposition} \label{prop:intersections}
  For any Morse gauge $M$ and any $D$, $N$, there is a function $R\co [0,\infty)\to[0,\infty)$ so that the following holds.
  Let $X$ be a graph of valence at most $N$ with a free $G$--action, let $A,B<G$ be subgroups acting $D$--coboundedly on $M$--strongly quasi-convex subgraphs $Y_A,Y_B$, respectively.  If $Y_C = Y_A\cap Y_B$ is non-empty then for any $p\in X$,
  \[ d(p,Y_A\cap Y_B) \le R\left( \max\left\{d(p,Y_A),d(p,Y_B)\right\} \right).\]
\end{proposition}
\begin{proof}
  Let $r>0$.  We must describe $R(r)$.
  
  Note that a concatenation of a geodesic of length $r$ with a geodesic of any length is a $(1,2r)$--quasi-geodesic.  Let $M_0=M(1,2r)$.  Let $R(r)$ be a bound for the number of pointed oriented simplicial paths in $X$ of length $\leq 2(M_0+D)$, up to the $G$--action.  (A bound can be chosen depending only on $N$, $M$, $D$, and $r$.)

  Let $p\in X$ be chosen so that $\max\left\{d(p,Y_A),d(p,Y_B)\right\} \le r$.  Let $q$ be a closest point in $Y_C$ to $p$.  Suppose $d(p,q) > R(r)$, and let $\gamma$ be a geodesic from $p$ to $q$. 
  Every vertex on $\gamma$ lies within $M_0$ of both $Y_A$ and $Y_B$.  It follows from $D$--coboundedness that every vertex $v$ on $\gamma$ lies within $M_0 + D$ of some $aq$ and some $b q$ for $a\in A, b\in B$.  By our choice of $R(r)$, there must be a pair of distinct vertices $v_1$, $v_2$ on $\gamma$ and paths $\sigma_i$ joining $v_i$ to $a_i q$, and $\tau_i$ joining $v_i$ to $b_i q$ of length at most $M_0+D$, and an element $h\in G$, so that $hv_1=v_2$, $h\sigma_1=\sigma_2$ and $h\tau_1=\tau_2$.  We may assume that $v_1$ is closer to $q$ than $v_2$ is.

  Note that we have $h Aq \cap Aq \ne \emptyset$.  Since the action of $G$ on $X$ is free, this implies that $h\in A$.  By the same argument $h\in B$, so $h\in C$, and thus $hq \in Y_C$.  But $hq$ is closer to $p$ than $q$ is, contradicting our choice of $q$.
\end{proof}
\begin{remark}
It is straightforward to see, using the above proof, that the set $Y_C$ is $R\circ M$--strongly quasi-convex, and also that $C$ acts coboundedly on $Y_C$ (with constants depending only on $D$, $M$ and $R$).
\end{remark}

Towards proving Theorem \ref{t:strong QC (2) implies (1)}, suppose that $G$ is a finitely generated group acting cocompactly on a CAT$(0)$ cube complex $X$, and suppose that hyperplane stabilizers are strongly quasi-convex in $G$.  Recall the graph $\Gamma$ from Lemma~\ref{lem:GammaA}, and the continuous, $G$--equivariant, Lipschitz map $\Psi \co \Gamma \to \left( X^b \right)^{(1)}$.

\begin{lemma} \label{lem:Stab(I) SQC}
Suppose that $W_1, \ldots, W_k$ are  hyperplanes in $X$ and that $I = \bigcap\limits_{i=1}^k W_i$ is non-empty.  Then $\Stab(I)$ is strongly quasi-convex in $G$.
\end{lemma}
\begin{proof}
By Lemma~\ref{lem:Int_Stab_Stab_Int}, the stabilizer of $I$ is a finite index supergroup of the intersection $\bigcap\limits_{i=1}^k\Stab(W_i)$.  The intersection of strongly quasi-convex subgroups is strongly quasi-convex by \cite[Theorem 1.2.(2)]{Tran}.
\end{proof}

\begin{proposition}\label{prop:Psi-1I SQC}
  There exists a Morse gauge $M$ so that for any collection of hyperplanes in $X$ with non-empty intersection $I$, the set
  $\Psi^{-1}(I)$ is strongly quasi-convex in $\Gamma$ with Morse gauge $M$.
\end{proposition}
\begin{proof}
The $G$--action on $X$ is cocompact, so there are finitely many $G$--orbits of sets $\{ W_1, \ldots , W_k \}$ of hyperplanes of $X$ so that $\bigcap\limits_{i=1}^k W_i \ne \emptyset$.  If $I$ is such a set and $g \in G$ then $\Psi^{-1}(g \cdot I) = g \cdot \Psi^{-1}(I)$.  Therefore, it suffices to consider a (finite) collection of representatives of $G$--orbits of intersections and prove that each is individually strongly quasi-convex, and then take a maximum over the finitely many Morse gauges for these subsets.

By Proposition~\ref{prop:Stab(W)_on_Psi-1(W)} $\Stab(I)$ acts cocompactly on $\Psi^{-1}(I)$ and by Lemma~\ref{lem:Stab(I) SQC} $\Stab(I)$ is strongly quasi-convex in $G$.  Therefore, by considering an orbit map $G \to \Gamma$ and applying Theorem \ref{t:QI of SQC}, we see that each $\Psi^{-1}(I)$ is a strongly quasi-convex subset of $\Gamma$.
\end{proof}

We now give the main part of the argument of the proof of Theorem \ref{t:strong QC (2) implies (1)}, namely that if hyperplane stabilizers are strongly quasi-convex, vertex stabilizers are also strongly quasi-convex.
We therefore fix a vertex $v$ of $X$.

Note that $\Psi^{-1}(v)$ is a non-empty and $\Stab(v)$--invariant set of vertices of $\Gamma$ consisting of finitely many $\Stab(v)$--orbits.
Thus in order to show $\Stab(v)$ is strongly quasi-convex in $G$, it suffices (by Theorem \ref{t:QI of SQC}) to show that the pre-image $\Psi^{-1}(v)$ is a strongly quasi-convex subset of $\Gamma$. 

We fix constants $K \ge 1$ and $C \ge 0$, suppose that $a$ and $b$ are vertices in $\Psi^{-1}(v)$ and let $\gamma$ be a $(K,C)$--quasi-geodesic in $\Gamma$ between $a$ and $b$.   Let $y$ be an arbitrary vertex on $\gamma$.  
We have to show $d(y,\Psi^{-1}(v))$ is bounded independent of $a$ and $b$.  By Proposition~\ref{prop:N(v)} $\Psi^{-1}(v)$ is finite Hausdorff distance from $\Psi^{-1}(N(v))$ (recall $N(v)$ is the cubical neighborhood of $v$), so it is enough to show $d(y,\Psi^{-1}(N(v)))$ is bounded independent of $a$ and $b$.

Here is a description of our bound:  Let $N$ be a bound for the valence of $\Gamma$ and $D$ the bound from Proposition~\ref{prop:Stab(W)_on_Psi-1(W)}.  Let $R_0 = M(K,C)$.  Assuming $R_i$ has been defined, we let $R_{i+1}$ be the maximum of $R_i$ and the number $R(R_i)$, where $R$ is the function from the conclusion of Proposition \ref{prop:intersections}, with the Morse gauge $M$ and the above $D,N$.  We will prove that $d(y,\Psi^{-1}(v)) \le R_{\dim X}$.

We build a sequence of points $y = z_0, z_1,\ldots,z_t$ in $\Gamma$ and hyperplanes $W_1, \ldots , W_t$ in $X$ for some $t\leq \dim X$ so that for each $i$, the following conditions are satisfied:
\begin{equation} \tag{$\ast_i$}
\left\{
\begin{gathered}
     d(y,z_i)\leq R_i;\\
     \Psi(z_i) \in \bigcap\limits_{s=1}^i W_s \mbox{ and }\\
     \bigcap\limits_{s=1}^i W_s \cap N(v) \ne \emptyset
\end{gathered}
\right.
\end{equation}
%
%
%
%

Notice that $(\ast_0)$ holds, since $d(y,z_0) = 0 \le R_0$ and the empty intersection of hyperplanes is $X^b$.

\begin{proposition}\label{prop:find z i plus one}
  Suppose $i \ge 0$ and that $z_0, \ldots z_i$ and $W_1, \ldots W_i$ have been defined and satisfy $(\ast_i)$.  Then either $\Psi(z_i) \in N(v)$ or there is some $z_{i+1}$ and $W_{i+1}$ so that $z_{i+1}$ and $W_1, \ldots , W_{i+1}$ satisfy $(\ast_{i+1})$.
\end{proposition}
\begin{proof}
  Suppose that $\Psi(z_i) \not\in N(v)$.  We are given hyperplanes $W_1, \ldots , W_i$ so that $I = \bigcap\limits_{s=1}^i W_s$ is nonempty and $I \cap N(v) \ne \emptyset$.  Notice that $I$ is a combinatorially convex sub-complex of $X^b$.  Choose a geodesic $p$ in the $1$-skeleton of $I$ from $I \cap N(v)$ to $\Psi(z_i)$.  The first edge of $p$ joins a vertex $u_1$ in $I \cap N(v)$ to a vertex $u_2$ which is not in $I \cap N(v)$.  Thus, the vertex $u_1$ corresponds to a cube $\tau$ of $X$ which contains $v$, and $u_2$ corresponds to a cube $\sigma$ which is a codimension one face of $\tau$ so that $\sigma$ does not contain $v$.  Let $W_{i+1}$ be the hyperplane of $X$ meeting $\tau$ in a mid-cube parallel to $\sigma$.  Since $u_1\in I \cap N(v)$ we see that $I \cap W_{i+1} \cap N(v) \ne \emptyset$, since it contains the vertex $u_1$ of $X^b$.
Also, because $p$ is a geodesic in $(X^b)^{(1)}$, $W_{i+1}$ separates $N(v)$ from $\Psi_{\Gamma}(z_i)$.  It is clear that 
\[ \bigcap\limits_{s=1}^{i+1} W_s \cap N(v) = W_{i+1} \cap I \cap N(v) \ne \emptyset	,	\]
so it remains to find $z_{i+1}$ satisfying the first two conditions of $(\ast_{i+1})$.

\begin{claim*}
$$d(y,\Psi^{-1}(W_{i+1}))\le  R_i.$$
\end{claim*}
\begin{proof}
If $\Psi(y)\in W_{i+1}$ then $d(y,\Psi^{-1}(W_{i+1}))=0$, so we suppose $\Psi(y)\not\in W_{i+1}$.  There are two cases.

 Suppose first $W_{i+1}$ separates $v$ from $\Psi(y)$.  In this case,  we know that $\gamma$ must cross $\Psi^{-1}(W_{i+1})$ between $a$ and $y$.  However, $\Psi_{\Gamma}(\gamma)$ is a loop, so $\gamma$ must also cross $\Psi^{-1}(W_{i+1})$ in the segment of $\gamma$ between $y$ and $b$.  Thus, there is a (quasi-geodesic) subsegment $\gamma_1$ of $\gamma$ which contains $y$ and which starts and finishes on $\Psi^{-1}(W_{i+1})$.  Since $M$ is a Morse gauge for $\Psi^{-1}(W_{i+1})$, and $M(K,C) \le R_i$, the claim follows in this case.

 Now suppose $W_{i+1}$ does not separate $v$ from $\Psi(y)$.  In this case, since $W_{i+1}$ \emph{does} separate $v$ from $\Psi(z_i)$ we know that $W_{i+1}$ must separate 
$\Psi(y)$ from $\Psi(z_i)$.  (In particular $i>0$ in this case.)  Since $d(y,z_i)\leq R_i$, and any path from $y$ to $z_i$ must intersect $\Psi^{-1}(W_{i+1})$, we must have
$d(y,\Psi^{-1}(W_{i+1}))\le R_i$ as required.  The claim is proved.
\end{proof}
Let $I = \bigcap\limits_{s=1}^i W_s$.
It follows immediately from the first two conditions of $(\ast_i)$ that $d(y,\Psi^{-1}(I)) \le R_i$.  Moreover, by the claim $d(y,\Psi_{\Gamma}^{-1}(W_{i+1}) \le R_i$.  Furthermore, we have 
\[	\Psi^{-1}\left(I \cap W_{i+1} \right) = \Psi^{-1}(I) \cap \Psi^{-1}(W_{i+1})	.\]
By Proposition~\ref{prop:Stab(W)_on_Psi-1(W)} each of the sets $\Psi^{-1}(I),\Psi^{-1}(W_{i+1})$ and $\Psi^{-1}(I\cap W_{i+1})$ is $D$--cobounded under the action of its stabilizer.  Proposition~\ref{prop:intersections} thus gives
\[ d\left(y,\Psi^{-1}\left(I \cap W_{i+1} \right) \right) \le R(R_i) \le R_{i+1}.\]
We choose $z_{i+1}$ to be any point of $\Psi^{-1}\left(I \cap W_{i+1} \right)$ which is closest to $y$.  The point $z_{i+1}$ satisfies the first two conditions of $(\ast_{i+1})$ so the proof is complete.
\end{proof}

For $j>\dim X$, there cannot exist a point $z_j$ satisfying $(\ast_j)$, since there are no $j$--tuples of hyperplanes with nonempty intersection.  Therefore, Proposition \ref{prop:find z i plus one} asserts that for some $i\leq \dim X$, $\Psi(z_i)=v$.
We conclude that $d(y,\Psi^{-1}(v))\leq R_i\leq R_{\dim X}$, as desired. 

This completes the proof of Theorem \ref{t:strong QC (2) implies (1)}.

\subsection{If cell stabilizers are QC then hyperplane stabilizers are QC} \label{ss:(1) implies (2)}
In this section we prove the direction \eqref{vertex QC} $\implies$ \eqref{hyp QC} of Theorem \ref{t:cell qc iff hyp}.  Therefore, suppose that $G$ is a hyperbolic group acting cocompactly on a CAT$(0)$ cube complex $X$, and suppose that the vertex stabilizers are quasi-convex in $G$.  

Let $W$ be a hyperplane in $X$.  Then as we have noted $W$ is a sub-complex of $X^b$.  Given $w$ in $W\cap (X^b)^{(0)}$, let $Y(w)$ be the (closed) $1$-neighborhood in $\Psi^{-1}(W)$ of $\Psi^{-1}(w)$.

\begin{lemma}\label{lem:uniform epsilon}
The sets $Y(w)$ are quasi-convex subsets of $\Gamma$ with constants which do not depend on $w$.
\end{lemma}
\begin{proof}
Since $\Stab(W)$ acts cocompactly on $W$, there are finitely many $\Stab(W)$--orbits of sets $Y(w)$, so the uniformity of constants will follow immediately if we can prove each $Y(w)$ is a quasi-convex subset of $\Gamma$.  We therefore fix such a $w$.

The stabilizer $\Stab(w)$ is equal to the stabilizer of some cube $\sigma$ of $X$.  Thus $\Stab(w)$ is virtually the intersection of the stabilizers of the vertices of $\sigma$.  The vertex stabilizers are assumed to be quasi-convex, so $\Stab(w)$ is also quasi-convex.

Since $G$ acts freely and cocompactly on $\Gamma$ and $\Psi$ is equivariant, $\Stab(w)$ acts freely and cocompactly on $\Psi^{-1}(w)$.
By Lemma~\ref{lem:StabIinStabTau}, $\Stab(W) \cap \Stab(w)$ is a finite-index subgroup of $\Stab(w)$, so it is also quasi-convex and acts freely and cocompactly on $\Psi^{-1}(w)$.  It moreover acts cocompactly on $Y(w)$.

  The result follows (for example, by Theorem \ref{t:QI of SQC}).
\end{proof}

We are ready to prove the direction \eqref{vertex QC} $\implies$ \eqref{hyp QC} of Theorem \ref{t:cell qc iff hyp}, which is the content of the following theorem.  For this result, we assume Theorem \ref{t:globally QC}, which is proved in Appendix \ref{app:QC}.

\begin{theorem}
Suppose that the hyperbolic group $G$ acts cocompactly on the cube complex $X$, and that for every vertex $v$ of $X$, the stabilizer $\Stab(v)$ is quasi-convex.  Then, for every hyperplane $W \subset X$, the stabilizer $\Stab(W)$ is a quasi-convex subgroup of $G$.
\end{theorem}
\begin{proof}
As we have already remarked, quasi-convexity of vertex stabilizers implies quasi-convexity of all cell stabilizers.

Let $\Gamma$, $\Psi^{-1}(W)$ and the $Y(w)$ be as discussed above.  Since $G$ acts freely and cocompactly on $\Gamma$, we know that $\Gamma$ is $\delta$-hyperbolic for some $\delta$.  Let $\epsilon$ be a constant so that $Y(w)$ is $\epsilon$--quasi-convex for every $w$ (Lemma \ref{lem:uniform epsilon}).

Since $\Stab(W)$ acts freely and cocompactly on $\Psi^{-1}(W)$, in order to prove the theorem it suffices to prove that $\Psi^{-1}(W)$ is quasi-convex in $\Gamma$, so let $p, q \in \Psi^{-1}(W)$.  

Consider a geodesic $\gamma$ in $(X^b)^{(1)}$ between $\Psi(p)$ and $\Psi(q)$.  Both $\Psi(p)$ and $\Psi(q)$ lie in $W$.
Since $W$ is combinatorially convex in $X^b$, the geodesic $\gamma$ is entirely contained in the $1$-skeleton of $W$ (considered as a sub-complex of $X^b$).  The vertices $w_1,\ldots,w_n$ on $\gamma$ correspond to cells of $X$ contained in $W$.
The sets $Y(w_i)$ corresponding to these cells satisfy the hypotheses of Theorem \ref{t:globally QC} with $m = 2$, $c = 1$, and $\epsilon$ the quasi-convexity constant chosen above.  Theorem \ref{t:globally QC} then implies that $Y(w_1)\cup\cdots\cup Y(w_n)$ is $\epsilon'$--quasi-convex, for a constant $\epsilon'$ depending only on $\epsilon$ and $\delta$.

In particular, a $\Gamma$--geodesic between $p$ and $q$ lies within $\epsilon'$ of $Y(w_1)\cup\cdots \cup Y(w_n)$.
  Since each of these $Y(w_i)$ is contained in $\Psi^{-1}(W)$, the $\Gamma$--geodesic between $p$ and $q$ stays uniformly close to $\Psi^{-1}(W)$, as required.
\end{proof}
Together with Theorem \ref{t:strong QC (2) implies (1)}, this completes the proof of Theorem \ref{t:cell qc iff hyp}.

\subsection{On generalizations of Theorem \ref{t:cell qc iff hyp}}  \label{ss:generalizations}
For a subgroup $H$ of a hyperbolic group $G$, the following three conditions are equivalent:
\begin{enumerate}[label=(\alph*)]
\item\label{Hsqc} $H$ is strongly quasi-convex in $G$.
\item\label{Hqc} $H$ is quasi-convex in $G$.
\item\label{Hqie} $H$ is undistorted in $G$.
\end{enumerate}
Dropping the condition that $G$ is hyperbolic, condition \ref{Hqc} ceases to be a useful notion, but the conditions \ref{Hsqc} and \ref{Hqie} still make sense.

One can ask for versions of Theorem \ref{t:cell qc iff hyp} where the hypothesis of hyperbolicity is removed and condition \ref{Hqc} is replaced by either condition \ref{Hsqc} or \ref{Hqie}.

\subsubsection{Strong quasi-convexity}
Replacing quasi-convexity with strong quasi-convexity we can ask about the following conditions for a finitely generated group $G$ acting cocompactly on a CAT$(0)$ cube complex:
\begin{enumerate}[label=(\arabic*S)]
\item \label{hyp SQC} Hyperplane stabilizers are strongly quasi-convex.
\item \label{vertex SQC} Vertex stabilizers are strongly quasi-convex.
\item \label{cell SQC} All cell stabilizers are strongly quasi-convex.
\end{enumerate}
As remarked earlier \ref{vertex SQC}$\Longleftrightarrow$\ref{cell SQC} follows from \cite[Theorem 1.2.(2)]{Tran}.  
Theorem \ref{t:strong QC (2) implies (1)} states that \ref{hyp SQC}$\implies$\ref{vertex SQC}

The remaining implication \ref{cell SQC}$\implies$\ref{hyp SQC} is \emph{false}, as shown for example by $\mathbb{Z}^2$ acting freely on a cubulated $\mathbb{R}^2$.

\subsubsection{Undistortedness}
The situation when replacing quasi-convexity with quasi-isometric embeddedness is murkier.  We consider the following conditions, for a finitely generated group $G$ acting cocompactly on a CAT$(0)$ cube complex $X$:
\begin{enumerate}[label=(\arabic*U)]
\item \label{hyp U} Hyperplane stabilizers are undistorted.
\item \label{vertex U} Vertex stabilizers are undistorted.
\item \label{cell U} All cell stabilizers are undistorted.
\end{enumerate}
If $X$ is a tree, \ref{hyp U} and \ref{cell U} each implies \ref{vertex U}, but not conversely.  For example, the double of a finitely generated group over a distorted group acts on a tree with undistorted vertex stabilizers but distorted edge/hyperplane stabilizers.

We do not know the relationship between \ref{hyp U} and \ref{cell U} in general, so we ask the question.
\begin{question}
  For finitely generated groups acting cocompactly on CAT$(0)$ cube complexes does \ref{hyp U}$\implies$\ref{cell U}?  Does \ref{cell U}$\implies$\ref{hyp U}?
\end{question}

\subsection{Height of families and the proof of Corollary \ref{cor:cube complex hyp}} \label{ss:Cor B}
The {\em height} of a subgroup was introduced in \cite{gmrs}.  We need a generalization of this notion to families of subgroups.

\begin{definition}\label{d:height} [Height of a family]
 Suppose that $G$ is a group and $\mc{H}$ is a collection of subgroups.  The {\em height} of $\mc{H}$ is the minimum number $n$ so that for every tuple of distinct
 cosets $(g_0H_0, g_1H_1, \ldots , g_nH_n)$ with $H_i \in \mc{H}$ (and $g_i \in G$), the intersection
 $\bigcap\limits_{i=0}^n H_i^{g_i}$ is finite.  If there is no such $n$ then we say the height of $\mc{H}$ is infinite.
\end{definition}
In case $\mc{H} = \{ H \}$ is a single subgroup, we recover the familiar notion of the height of a subgroup from \cite{gmrs}.

The following result for a single subgroup is part of \cite[Main Theorem]{gmrs}.  The proof of that result from \cite{VH} (Corollary A.40 in that paper) can be adapted in the obvious way to prove the result for finite families.  This result was proved in the more general setting of strongly quasi-convex subgroups by Tran \cite[Theorem 1.2.(3)]{Tran}.
\begin{proposition}\label{p:finite height}
 Let $G$ be a hyperbolic group and $\mc{H}$ a finite collection of quasi-convex subgroups of $G$.  Then the height of $\mc{H}$ is finite.
\end{proposition}

We also use the following special case of a theorem of Charney--Crisp \cite[Theorem 5.1]{CharneyCrisp}.
\begin{theorem} \label{t:X QI to coned}
Suppose that $G$ acts cocompactly on a cube complex $X$.  Then $X$ is quasi-isometric to the space obtained from the Cayley graph of $G$ by coning cosets of stabilizers of vertices to points.
\end{theorem}

We now prove Corollary \ref{cor:cube complex hyp}.  For convenience, we recall the statement.

\cubecomplexhyp*

\begin{proof}
  If $G$ is a hyperbolic group acting cocompactly on a CAT(0) cube complex, and if the stabilizers in $G$ of vertices in $X$ are quasi-convex, then \cite[Theorem 7.11]{bowditch:relhyp}, due to Bowditch, implies that this coned graph is $\delta$--hyperbolic for some $\delta$.  Theorem \ref{t:X QI to coned} then implies that the cube complex $X$ is $\delta$--hyperbolic for some (possibly different) $\delta$.  Thus, we have the first statement from Corollary \ref{cor:cube complex hyp}.

  Now we prove the statement about fixed point sets of infinite subgroups.  Let $I$ be a collection of orbit representatives of cells in $X$.  For $i\in I$, let $Q_i = \{g\in G\mid gi = i\}$, and let $\mc{Q} = \{Q_i\}_{i\in I}$.  Then $\mc{Q}$ is a finite collection of quasi-convex subgroups of $G$, so it has some finite height $k$ by Proposition \ref{p:finite height}.  If $H<G$ is infinite with nonempty fixed set and $\sigma$ is a cell meeting the fixed point set of $H$, then $H<Q_i^g$ where $\sigma = gi$.  Since the height of $\mc{Q}$ is $k$, at most $k$ such cells appear.

  In \cite{genevois:coning}, Genevois studies actions of groups on hyperbolic CAT$(0)$ cube complexes and shows in Theorem 8.33 that, in this setting, acylindricity is equivalent to the condition:
\begin{equation}\tag{G} \exists L, R, \forall x,y\in X^{(0)},\ d(x,y)\geq L \implies \#(\Stab(x)\cap\Stab(y))\leq R 
\end{equation}
  We take $R$ to be the maximum size of a finite subgroup and $L=k$.  Suppose $d(x,y) \geq L$.  Then the union of the combinatorial geodesics joining $x$ to $y$ contains finitely many (but at least $k+1$) vertices.  There is a finite index subgroup of $\Stab(x)\cap \Stab(y)$ which fixes all of these vertices.  This finite index subgroup fixes more than $k$ cells, so it is finite.  This implies $\Stab(x)\cap \Stab(y)$ is finite, as desired.
\end{proof}

\begin{remark}\label{rem:middleconclusion}
  In the context where $G$ is a finitely generated group acting cocompactly on a cube complex $X$ with strongly quasi-convex hyperplane stabilizers, the same proof of conclusion \eqref{cond:fixed} works as written, replacing the reference to Proposition \ref{p:finite height} with a reference to \cite[Theorem 1.2.(3)]{Tran}.
\end{remark}

\section{Conditions for quotients to be CAT$(0)$} \label{s:X mod K CAT(0)}

As noted in the introduction, Theorem \ref{omnibus} follows quickly from Theorem \ref{t:cell qc iff hyp}, Theorem \ref{t:fill to get proper action}, Agol's Theorem \cite[Theorem 1.1]{VH} and Wise's Quasi-convex Hierarchy Theorem \cite[Theorem 13.3]{WisePUP}.  Thus, other than Theorem \ref{t:globally QC} in Appendix \ref{app:QC} (which is independent of everything else in this paper), it remains to prove Theorem \ref{t:fill to get proper action}.  Therefore, we are interested in conditions on a group $G$ acting on a CAT$(0)$ cube complex $X$ and a normal subgroup $K \unlhd G$ which ensure that the quotient $\leftQ{X}{K}$ is a CAT$(0)$ cube complex.  In this section we develop criteria in terms of complexes of groups to ensure this.  In the next section, we translate these conditions into algebraic conditions on $K\unlhd G$.

Three conditions need to be ensured in order for the complex $Z = \leftQ{X}{K}$ to be a CAT$(0)$ cube complex:

\begin{enumerate}
\item $Z$ must be simply-connected;
\item $Z$ must be a cube complex (rather than a complex made out of cells which are quotients of cubes); and
\item $Z$ must be non-positively curved.
\end{enumerate}

We investigate these three properties in turn.

\subsection{Ensuring the quotient is simply-connected}
First, we give a sufficient condition for $\leftQ{X}{K}$ to be simply-connected.

Since $X$ is a finite dimensional cube complex, it has finitely many shapes, and we can use the following application of a theorem of Armstrong:
\begin{theorem}\label{thm:armstrong}
  Let $X$ be a simply connected metric polyhedral complex with finitely many shapes, and let $K$ be a group of isometries of $X$ respecting the polyhedral structure, generated by elements with fixed points.  Then $\leftQ{X}{K}$ is simply connected.
\end{theorem}
\begin{proof} (Sketch)
  A theorem of Armstrong, \cite[Theorem 3]{Armstrong65}, shows that $\leftQ{X}{K}$ is simply connected with the CW topology.  We have to show it is still simply connected with the metric topology.
  
  Because $X$ has finitely many shapes there is an equivariant triangulation $\mc{T}$ and an $\epsilon>0$ so that for every finite sub-complex $Y$, the $\epsilon$--neighborhood of $Y$ deformation retracts to $Y$.  If $f\co S^1\to X$ is any loop, then a compactness argument shows it lies in an $\epsilon$--neighborhood of some such finite complex.  We can then homotope $f$ to have image in $Y$ and apply the simple connectedness of $\leftQ{X}{K}$ with the CW topology.
\end{proof}

We remark that the hypothesis of finitely many shapes is necessary even when $X$ is CAT$(0)$ as the following example shows:
\begin{example}
  For $n\in \{2,3,\ldots\}$, let $D_n$ be the Euclidean cone of radius $1$ on a loop $\sigma_n$ of length $\frac{2\pi}{n^2}$.  For each $n$ mark a point on $\sigma_n$.  Let $Y$ be obtained from $\bigcup D_n$ by identifying the marked points.  Unwrapping all the cones to Euclidean discs gives a tree of Euclidean discs of radius $1$.  We call this CAT$(0)$ space $\widetilde{Y}$.  There is a discrete group of isometries $\Gamma=\langle \gamma_2,\gamma_3,\ldots\rangle$ acting on $\widetilde{Y}$ with quotient $Y$, so that each $\gamma_n$ fixes the center of some disc and rotates it by an angle of $\frac{2\pi}{n^2}$.  Nonetheless $Y$ is not simply connected, as the infinite concatenation of the loops $\sigma_n$ has finite length, but cannot be contracted to a point.
\end{example}

\subsection{Ensuring the quotient is a cube complex}\label{ss:cube}
We now turn to the question of when $\leftQ{X}{K}$ is a cube complex.

In order that the quotient $Z = \leftQ{X}{K}$ be a cube complex, there needs to be no element of $K$ which fixes a cell of $X$ set-wise but not point-wise.

Suppose that $\sigma$ is a cube of $X$.  The stabilizer $G_\sigma$ has a finite-index subgroup $Q_\sigma$ consisting of those elements which fix $\sigma$ pointwise.  Let $\{ \sigma_1, \ldots , \sigma_k \}$ be a set of representatives of $G$--orbits of cubes in $X$.  The following result is straightforward.

\begin{proposition} \label{p:quotient cubical}
Suppose that $G$ acts cocompactly on the cube complex $X$ and that $K$ is a normal subgroup of $G$ so that for each $i$ we have $G_{\sigma_i} \cap K \le Q_{\sigma_i}$.  Then the quotient $\leftQ{X}{K}$ is a cube complex and the links of vertices in $\leftQ{X}{K}$ inherit a cellular structure from the simplicial structure of cells in $X$.
\end{proposition}

\subsection{Ensuring the quotient is nonpositively curved}
The most complicated condition to ensure is that $\leftQ{X}{K}$ is nonpositively curved.

Throughout this subsection we suppose that $X$ is a CAT$(0)$ cube complex and that $\mc{X}$ is its idealization (see Definition \ref{def:cube idealization}).  We suppose further that $G$ is a group acting cocompactly on this cube complex.  The induced action of $G$ on $\mc{X}$ has quotient a scwol $\mc{Y}$.
Making choices as in Definition \ref{def:actionchoices}, we obtain a complex of groups $G(\mc{Y})$, with associated category $CG(\mc{Y})$.  Choosing a vertex $v_0\in \mc{Y}$, there is then an identification of $G$ with $\pi_1(CG(\mc{Y}),v_0)$.  Moreover, we choose a normal subgroup $K \unlhd G$ so that $\leftQ{X}{K}$ is a cube complex.  Throughout this section we denote $\mc{Z} = \leftQ{\mc{X}}{K}$.
  \footnote{We remark that we do not make any further assumptions than these about $G$ and $X$ in this subsection.  This may be an important observation for future applications.}

In Subsection \ref{ss:cube} we discussed how to find subgroups $K$ so that $\leftQ{X}{K}$ is a cube complex, but for this section we just assume that this is the case.

Let $\mc{C}_K$ be the cover of the category $CG(\mc{Y})$ corresponding to the subgroup $K$.  
Observe that $CG(\mc{Y})$--loops lift to $\mc{C}_K$ if and only if they represent elements of $K$.  (Basepoints are mostly omitted in this section, since we deal with a normal subgroup $K$.)

\begin{sa} \label{sa:abuse}
  Lemma~\ref{lem:nice_paths} tells us that any $CG(\mc{Y})$--path is homotopic to a concatenation of group and scwol arrows. 
  Throughout this section we assume all $CG(\mc{Y})$--paths have been homotoped to this form, though we do not assume that the arrows \emph{alternate} between group and scwol arrows.  We blur the distinction between the scwol arrow $(1,a_i)$ in $CG(\mc{Y})$ and the $\mc{Y}$-arrow $a_i$, and also between the group arrow $(g,\identity_v)$ and the element $g\in G_v$.
  
  Thus we may write a $CG(\mc{Y})$ path for example as a list $g_1 \cdot e_1 \cdot  e_2 \cdot  g_2 \cdot ...$, where each $g_i$ is an element of a local group $G_v$ and represents the edge $(g,\identity_v)^+$, corresponding to the group arrow $(g,\identity_v)$, and each $e_i$ is equal to some $a_i^{\pm}$ for a scwol arrow $(1,a_i)$.
   We implicitly assume that each concatenation we write defines a path, which  forces the group arrows labeled $g_i$ to be elements of particular local groups.  Whenever we consider a $CG(\mc{Y})$--path of length $1$ consisting of a single group arrow we are either explicit about the local group or else it is clear from the context.

In case we have an edge of the form $(1,\identity_v)^\pm$, we often implicitly (or explicitly) omit this arrow from our path.  Again this only changes the path by an elementary homotopy.  
\end{sa}

\begin{definition}[Link of a cube]
Given an $n$--cube $\sigma$ in a cube complex $Z$, let $b_\sigma$ be the barycenter.
For sufficiently small $\epsilon>0$, the sphere $\{x\in Z\mid d_Z(b_\sigma,x) = \epsilon\}$ is the join of an $(n-1)$--sphere with some simplicial complex (here we take the join of a $(-1)$--sphere with $K$ to be equal to $K$).  This simplicial complex is what we refer to as the \emph{link} of $\sigma$, or $\lk(\sigma)$.  It is naturally triangulated by simplices corresponding to inclusions of $\sigma$ as a face of some higher-dimensional cube.  In particular $\link(\sigma)$ is a \emph{$\Delta$--complex} \cite[Chapter 2.1]{hatcher:book} (though it may not be simplicial).  
\end{definition}
Though the link of a cube naturally has the structure of a piecewise spherical complex, we will think of it as just a combinatorial object.  When we refer to paths or loops in a link, these will always be concatenations of $1$--cells.  The \emph{length} of such a path is the number of $1$--cells traversed.

\begin{theorem} \label{t:link condition} (Gromov's Cubical Link condition, \cite[Theorem II.5.20]{bridhaef:book})
The cube complex $Z$ is a non-positively curved cube complex if and only if the link of each vertex in $Z$ is flag.
\end{theorem}

In this section, we provide a set of conditions on the subgroup $K$ which imply the link condition for $Z=\leftQ{X}{K}$.

We record two elementary observations:

\begin{lemma} \label{l:L simplicial}
  Let $v$ be a vertex of a cube complex, and let $L = \link(v)$.  Then $L$ is simplicial if and only if for every cell $\sigma$ containing $v$, the $1$--skeleton of $\link(\sigma)$ contains no immersed loop of length $1$ or $2$.
\end{lemma}
\begin{proof}
If $L$ fails to be simplicial, there is either a non-embedded simplex, or a pair of simplices which intersect in a set which is not a face of both.  If a simplex is non-embedded, we obtain a loop of length $1$ in $L$.  If two embedded simplices $\tau_1$ and $\tau_2$ of $L$ intersect in a set which is not a single face, let $F_1$ and $F_2$ be different maximal faces in the intersection, and let $f=F_1\cap F_2$.  For $i \in\{ 1,2\}$, let $v_i$ be a vertex in $F_i\minus f$.  Then the simplices spanned by $v_1\cup f$ and $v_2\cup f$ correspond to points in $\link(f)$ which lie on an immersed loop of length $2$.  But $f$ corresponds to some cube containing $\sigma$, and  $\link(f)\subset L$ is isomorphic to the link of that cube.
\end{proof}

\begin{lemma} \label{l:L flag}
  Let $v$ be a vertex of a cube complex, and suppose $L = \link(v)$ is simplicial.  Then $L$ is a flag complex if and only if for every cell $\sigma$ containing $v$, every loop of length $3$ in $\link(\sigma)$ is filled by a $2$--cell.
\end{lemma}
\begin{proof}
 If $\sigma$ is a cube and $\phi$ is a cube with 
$\sigma$ as a face, of dimension one higher, then $\phi$ corresponds to a vertex $f$ of the link $L$ of 
$\sigma$.  The link of $\phi$ is isomorphic to the link in $L$ of $f$.  The result now follows from 
\cite[Remark II.5.16.(4)]{bridhaef:book}.
\end{proof}

Therefore, in order to ensure $Z$ is nonpositively curved,  for each cell $\sigma$ in $Z$ we must rule out loops of length $1$ and $2$ in $\link(\sigma)$ and also ensure that any loop of length $3$ in $\link(\sigma)$ is filled by a $2$--cell.  We first explain how we translate between $1$--cells in links in $Z$ and $CG(\mc{Y})$--paths.  Then we develop the required conditions to rule out loops, finally dealing with loops of length $3$ which must be filled by $2$--cells.

\subsubsection{$CG(\mc{Y})$--paths associated to $1$--cells in $\link(\sigma)$}
Below we choose, for each cube $\sigma$ in $X$ and each oriented $1$--cell $\alpha$ in the link of $\sigma$, a $CG(\mc{Y})$--path $p_{\orbit{\alpha}}$ which is the label of an unscwolification of the idealization of $\alpha$.  As indicated by the notation, this label is the same for two such $1$--cells in the same $G$--orbit.  (In fact we will choose these paths for slightly more general objects than $1$--cells in links of cubes.)  If $\alpha$ is an oriented $1$--cell, we write $\backwards{\alpha}$ for the same $1$--cell with the opposite orientation.

We fix a cube $\sigma$ of $X$.  The second barycentric subdivision of the link of $\sigma$ embeds naturally in the geometric realization of $\mc{X}$.  The vertices of the image of $\link(\sigma)$ are precisely the length $\geq 2$ chains of cubes whose minimal element is $\sigma$.

We first consider an oriented $1$--cell $\beta$ of $\link(\sigma)$.
The $1$--cell $\beta$ corresponds to a triple of cubes $ {\epsilon}_1, {\epsilon}_2,  {\phi}$ in $X$ so that 
\[	\dim( {\sigma}) = \dim( {\epsilon}_i)-1 = \dim( {\phi}) - 2	,	\]
with ${\epsilon}_1,  {\epsilon}_2 \subset  {\phi}$ and $ {\epsilon}_1\cap  {\epsilon}_2 =  {\sigma}$.

In particular, $\beta$ has idealization an $\mc{X}$--path of length $4$, made up of arrows

\begin{equation}
	( \sigma \subset  \epsilon_1) {\longleftarrow} ( \sigma \subset  \epsilon_1 \subset  \phi) {\longrightarrow} ( \sigma \subset  \phi) {\longleftarrow} ( \sigma \subset  \epsilon_2 \subset  \phi) {\longrightarrow} ( \sigma \subset  \epsilon_2)	.		
\end{equation}

In order to formulate Definition~\ref{def:palpha} and Lemma~\ref{l:nice paths for edges} below we must consider a slightly more general situation:  we suppose 
$\epsilon_1, \epsilon_2,\phi$ are cubes in $X$ with $\epsilon_1, \epsilon_2$ codimension one sub-cubes of $\phi$, and $\gamma$ is a chain of cubes in $X$ so that each element of $\gamma$ is contained in each of $\epsilon_1,\epsilon_2,\phi$.  We can naturally extend $\gamma$ to chains which we denote $(\gamma \subset \epsilon_1), (\gamma \subset \epsilon_2)$, $(\gamma \subset \phi)$, and $(\gamma \subset \epsilon_i \subset \phi)$.
This sequence of chains corresponds to a $1$--cell $\alpha$ in an `iterated link' (a link of a cell in a link, etc.), which we fix from now through Definition~\ref{def:palpha}.
The $1$--cell $\alpha$ has idealization an $\mc{X}$--path of length $4$ as follows:
\begin{equation}\label{eq:arrowpath}
	( \gamma \subset  \epsilon_1) {\longleftarrow} ( \gamma \subset  \epsilon_1 \subset  \phi) {\longrightarrow} ( \gamma \subset  \phi) {\longleftarrow} ( \gamma \subset  \epsilon_2 \subset  \phi) {\longrightarrow} ( \gamma \subset  \epsilon_2)	.		
\end{equation}

The $\mc{X}$--path \eqref{eq:arrowpath} may not embed in $\mc{Y}$.
There are two ways this could happen.  The first is that there is an element of $\Stab(\gamma)$ which sends $\epsilon_1$ to $\epsilon_2$, but no such element fixes $\phi$.  In this case, the image in $\mc{Y}$ is a non-backtracking loop.  The second possibility is that there is an element $g\in G$ sending each of $\gamma$ and $\phi$ to itself, but exchanging $\epsilon_1$ and $\epsilon_2$.  If there is such a $g$, the idealization of the $1$--cell $\alpha$ backtracks in $\mc{Y}$, forming a `half-$1$--cell'.  Since we are assuming $Z$ is a cube complex, the second possibility does not arise for the image of~\eqref{eq:arrowpath} in $\mc{Z}$, though the first possibility may.

Let $y_{\orbit{\alpha}} = a_1^+ \cdot  a_2^- \cdot  a_3^+ \cdot  a_4^-$ be the $\mc{Y}$--path which is the image of the $\mc{X}$--path~\eqref{eq:arrowpath}.  For more compact notation, we define the following:
\[ \nu = \orbit{(\gamma\subset\phi)}\quad \mbox{and}\quad \mu_i = \orbit{( {\gamma} \subset  {\epsilon_i} \subset  {\phi})},\quad \xi_i = \orbit{( \gamma \subset {\epsilon_i})}\quad \mbox{for } i\in \{1,2\}.\]
Then we have the injective homomorphisms
\[	\psi_{a_2} \co G_{\mu_1} \to G_{\nu}	,	\]
and
\[	\psi_{a_3} \co G_{\mu_2} \to G_{\nu}	.	\]

The images of $\psi_{a_2}$ and $\psi_{a_3}$ are equal.  The projection to $\mc{Y}$ of the path~\eqref{eq:arrowpath} associated to $\alpha$ depends only on $\orbit{\alpha}$.  We
denote the common image of $\psi_{a_2}$ and $\psi_{a_3}$ in $G_\nu$ by $G_{\nu}^+$.  
Note that $G_{\nu}^+$ either has index $2$ in $G_\nu$ (in case there is a $g$ fixing $\gamma$ and $\phi$ and exchanging $\epsilon_1$ with $\epsilon_2$) or else $ G_{\nu}^+= G_\nu$ (if there is no such $g$).  
In case $G_\nu^+$ has index $2$ in $G_\nu$, we fix a choice of $g_\nu \in G_\nu \minus G_{\nu}^+$.  We make this choice once and for all for each orbit of $(\gamma,\epsilon_1,\epsilon_2,\phi)$, so that the choice depends only on the orbit and not on the representative.

In the sequel, we refer to the vertex groups by $G_{i(\orbit{\alpha})}$ (for $G_{\xi_1})$ and $G_{t(\orbit{\alpha})}$ (for $G_{\xi_2})$.  We further define ``edge-inclusions'' $\psi_{\orbit{\alpha}}\co G_{\nu}^+\to G_{t(\orbit{\alpha})}$ and
$\psi_{{\orbit{\backwards\alpha}}}\co G_{\nu}^+\to G_{i(\orbit{\alpha})}$ by

\[ \psi_{\orbit{\alpha}} = \psi_{a_4}\circ\psi_{a_3}^{-1}\mbox{, and }\psi_{\orbit{\backwards{\alpha}}} = \psi_{a_1}\circ \psi_{a_2}^{-1}.\]

Let $E_{\orbit{\alpha}}$ denote the image of $\psi_{\orbit{\alpha}}$ in $G_{t(\orbit{\alpha})}$, and $E_{\orbit{\backwards{\alpha}}}$ denote the image of $\psi_{\orbit{\backwards{\alpha}}}$ in $G_{i(\orbit{\alpha})}$.

\begin{definition}\label{def:palpha}
In case $G_{\nu}^+ \ne G_\nu$, associate to $\alpha$ the $CG(\mc{Y})$--path 
\[	p_\alpha = a_1^+ \cdot a_2^- \cdot g_\nu \cdot a_3^+ \cdot a_4^-	,	\]
where $g_\nu \in G_\nu \minus G_{\nu}^+$ is the element fixed above.
Note that in this case $a_2=a_3$ and $a_1=a_4$.  In case $G_\nu^+=G_\nu$ let 
\[ p_\alpha = a_1^+ \cdot a_2^- \cdot a_3^+ \cdot a_4^-.\]  
In either case some lift of $p_\alpha$ to $\cgyt$ is an unscwolification of the idealization of $\alpha$.
The $CG(\mc{Y})$--path $p_\alpha$ depends only on $\orbit{\alpha}$; if there is a $g$ with $g\sigma = \sigma$ and $g\alpha'=\alpha$, then $p_{\alpha'} = p_\alpha$.
Therefore, we have a well-defined $CG(\mc{Y})$--path $p_{\orbit{\alpha}}$.  
\end{definition}

Next we consider $CG(\mc{Y})$--paths whose scwolifications traverse $y_{\orbit{\alpha}}$ and whose lifts to $\cgyt$ have non-backtracking scwolifications.
(Since $Z = \leftQ{X}{K}$ is assumed to be a cube complex, these paths also have lifts in $\mc{C}_K$ whose scwolifications are non-backtracking in $\mc{Z}$.)
We observe in the following lemma that such paths can be converted by a sequence of elementary homotopies to paths which consist of a copy of $p_{\orbit{\alpha}}$ bookended by group arrows.

\begin{lemma}\label{l:nice paths for edges}
Suppose that $G_{\nu}^+ = G_\nu$.  Then
any $CG(\mc{Y})$--path $$g_0 \cdot a_1^+ \cdot g_1 \cdot a_2^- \cdot g_2 \cdot a_3^+ \cdot g_3 \cdot a_4^- \cdot g_4$$ is homotopic to a $CG(\mc{Y})$--path of the form
\[	g_0' \cdot p_{\orbit{\alpha}} \cdot g_1'	.	\]
Suppose that $G_{\nu}^+ \ne G_\nu$.  Then any $CG(\mc{Y})$--path
$$g_0 \cdot a_1^+ \cdot g_1 \cdot a_2^- \cdot g_2 \cdot a_3^+ \cdot g_3 \cdot a_4^- \cdot g_4$$ so that $g_2 \not\in E_{\orbit{\alpha}}$ is homotopic to a $CG(\mc{Y})$--path of the form
\[	g_0' \cdot p_{\orbit{\alpha}} \cdot g_1'	.	\]

  In both cases, the scwolification of the path is fixed during the homotopy.  Moreover any lift of the homotopy to a cover of $CG(\mc{Y})$ gives a sequence of paths with constant scwolification.
\end{lemma}

\begin{notation}
We fix some notation in order to study paths in $Z$ and also $\mc{Y}$--paths and $CG(\mc{Y})$--paths.
As above, we use $\orbit{.}$ to denote a $G$-orbit in $\mc{Z}$, which corresponds to its image in $\mc{Y}$ under the projection $\pi \co \mc{Z} \to \mc{Y}$.

Let $p_{\orbit{\alpha}}$ be one of the $CG(\mc{Y})$--paths fixed in Definition \ref{def:palpha}, corresponding to a $1$--cell $\alpha$ in some link (or iterated link) of a cube of $Z$.
The $CG(\mc{Y})$--path $p_{\orbit{\alpha}}$ has an underlying $\mc{Y}$--path, which we denote by $y_{\orbit{\alpha}}$.  Define $t(\orbit{\alpha}) = t(y_{\orbit{\alpha}})$ and $i(\orbit{\alpha})=i(y_{\orbit{\alpha}})$.  This is so we can denote the corresponding local groups as $G_{i(\orbit{\alpha})}$ and $G_{t(\orbit{\alpha})}$.

We will also need to refer to the subgroups $E_{\orbit{\alpha}}<G_{t(\orbit{\alpha})}$ and $E_{\orbit{\backwards{\alpha}}}<G_{i(\orbit{\alpha})}$ defined just  before Definition \ref{def:palpha}.  Each of these subgroups can be thought of as the pointwise stabilizer of some translate of a lift of $\alpha$ to $X$.
\end{notation}

\subsubsection{Loops in $\link(\check\sigma)$}
We are now ready to formulate the conditions on $K$ which characterize whether or not $\leftQ{X}{K}$ is non-positively curved.  We use Lemmas \ref{l:L simplicial} and \ref{l:L flag} repeatedly.

Recall that we have fixed a $K\lhd G$ so that $Z = \leftQ{X}{K}$ is a cube complex.  We also fix a cube $\check\sigma$ of $Z$, and a lift $\sigma$ of $\check\sigma$ to $X$.  If $\alpha$ is a $1$--cell in $\link(\sigma)$, then as described above there is a corresponding $\mc{X}$--path of length $4$ (its idealization).  Similarly a $1$--cell in the link of $\check\sigma$ has a corresponding $\mc{Z}$--path of length $4$.  We sometimes conflate a concatenation of $1$--cells with a concatenation of these idealizations.

We next give an algebraic characterization of $1$--cells in $\link(\sigma)$ which project to loops of length $1$ in $\link(\check\sigma)$.  Recall $\mc{C}_K = \leftQ{\cgyt}{K}$.

\begin{lemma} \label{l:loop of length 1}
Let $\alpha$ be a $1$--cell in $\link(\sigma)$, and $\check\alpha$ the projection of $\alpha$ to $\link(\check\sigma)$.  The endpoints of $\check\alpha$ are equal if and only if there is a $CG(\mc{Y})$--loop of the form $p_{\orbit{\alpha}} \cdot g$ that represents a conjugacy class in $K$.
\end{lemma}
\begin{proof}
  Thinking of $X$ as the geometric realization of $\mc{X}$, the $1$--cell $\alpha$ is the realization of a $\mc{X}$--path $q_\alpha$ of length $4$, which projects to a $\mc{Y}$--path $a_1^+ \cdot a_2^- \cdot a_3^+ \cdot a_4^-$.  Let $\widehat{q}_\alpha$ be an unscwolification of $q_\alpha$ in $\cgyt$, which we may choose to have label
  \begin{equation}\label{unscwol} a_1^+\cdot a_2^- \cdot g_1 \cdot a_3^+ \cdot a_4^-,
  \end{equation}
  for some group arrow $g_1$.

    Suppose first that the endpoints of $\check\alpha$ coincide.  Then the path \eqref{unscwol} projects in $\mc{C}_K$ to a path with endpoints separated by a group arrow, and so there is a $\mc{C}_K$--loop with label
    \[ a_1^+\cdot a_2^- \cdot g_1 \cdot a_3^+ \cdot a_4^-\cdot g_2. \]
  Lemma \ref{l:nice paths for edges} implies that this loop is homotopic to a loop with label $g_0 \cdot p_{\orbit{\alpha}}\cdot g_1$ for some $g_0, g_1$, and starting this loop at a different place gives the required $CG(\mc{Y})$--loop $p_{\orbit{\alpha}} \cdot g$ representing a conjugacy class of $K$.

  Conversely, suppose a conjugacy class in $K$ is represented by a $CG(\mc{Y})$--loop of the form $p_{\orbit{\alpha}} \cdot g$.  Then $p_{\orbit{\alpha}} \cdot g$ lifts to a loop in $\mc{C}_K$ whose scwolification is a projection of a translate of $q_\alpha$ by some element of $G$.  In particular, $q_\alpha$ must project to a loop, and so the endpoints of $\check\alpha$ coincide.
\end{proof}

\begin{definition}
  If $\o$ is an object of $\mc{Y}$, let $K_\o\lhd G_\o$ be $K\cap G_\o$.
\end{definition}

\begin{definition}
  A $CG(\mc{Y})$--path $p$ is \emph{$K$--non-backtracking} if for some (equivalently any) lift $\widehat{p}$ to $\mc{C}_K$, the scwolification $\Theta_K(\widehat{p})$ is non-backtracking.  A $CG(\mc{Y})$--loop can be thought of as a path starting at any of its vertices.  The loop is $K$--non-backtracking if all these paths are $K$--non-backtracking.
\end{definition}

\begin{lemma}\label{lem:Knonbacktrack}
A $CG(\mc{Y})$--path $g_0 \cdot p_{\orbit{\alpha_1}}\cdot g_1 \cdot p_{\orbit{\alpha_2}} \cdot \ldots \cdot g_{k-1} \cdot p_{\orbit{\alpha_k}} \cdot g_k$ is $K$--non-backtracking if and only if the first of the following two conditions holds.  A $CG(\mc{Y})$--loop with such a label is $K$--non-backtracking if and only if both conditions hold.
\begin{enumerate}
\item\label{pathcond} For $i \in \{ 1 , \ldots , k - 1\}$, if $\orbit{\alpha_{i+1}} = \orbit{\backwards{\alpha_i}}$ then $g_i \not\in E_{\orbit{\alpha_i}}K_{t(\orbit{\alpha_i})} \subset G_{t(\orbit{\alpha_i})}$; and
\item If $\orbit{\alpha_1} = \orbit{\backwards{\alpha_k}}$ then $g_kg_0 \not\in E_{\orbit{\alpha_k}}K_{t(\orbit{\alpha_k})} \subset G_{t(\orbit{\alpha_k})}$.
\end{enumerate}
\end{lemma}

The following result algebraically characterizes immersed loops of length $2$ in $\link(\sigma)$.

\begin{lemma} \label{l:no loop of length 2}
  Let $p$ be a path in $\link(\check{\sigma})$ which is a concatenation of two $1$--cells, $\check\alpha$ and $\check\beta$ which lift respectively to $1$--cells $\alpha$ and $\beta$ in $X$.  The following are equivalent:
  \begin{enumerate}
\item\label{cell2loop} There is a path $p' = \check\alpha'\cdot\check\beta'$ in $\link(\check\sigma)$ with $\orbit{\check\alpha'} = \orbit{\check\alpha}$ and $\orbit{\check\beta'} = \orbit{\check\beta}$ so that $p'$ is an immersed loop.
  \item There is a $K$--non-backtracking $CG(\mc{Y})$--loop $p_{\orbit{\alpha}} \cdot g_1 \cdot p_{\orbit{\beta}} \cdot g_2$ that represents a conjugacy class in $K$.
  \end{enumerate}
  Moreover, in case these conditions hold, the path $p'$ can be chosen to be the scwolification of a lift of $p_{\orbit{\alpha}} \cdot g_1 \cdot p_{\orbit{\beta}} \cdot g_2$ (and conversely $p_{\orbit{\alpha}} \cdot g_1 \cdot p_{\orbit{\beta}} \cdot g_2$ is the $CG(\mc{Y})$--path which labels the unscwolification of $p'$).
\end{lemma}
\begin{proof}
Suppose that there is an immersed loop $p' = \check\alpha'\cdot\check\beta'$ as in~\eqref{cell2loop}.  The idealization of  $p'$ is a $\mc{Z}$--path $q_{p'}$ of length $8$, labeled by a $\mc{Y}$--path $a_1^+ \cdot a_2^- \cdot a_3^+ \cdot a_4^- \cdot b_1^+ \cdot b_2^- \cdot b_3^+ \cdot b_4^-$ as discussed above.  Using Lemma \ref{l:nice paths for edges}, we can choose an unscwolification
$\widehat{q}_{p'}$ of $q_{p'}$ in $\mc{C}_K$ with label
\[  p_{\orbit{\alpha}}\cdot g_1 \cdot  p_{\orbit{\beta}}	,	\]
where $g_1$ is a group arrow.  But the unscwolification $\widehat{q}_{p'}$ has endpoints separated by a group arrow $g_2$, so there is a loop labeled $p_{\orbit{\alpha}}\cdot g_1 \cdot  p_{\orbit{\beta}}\cdot g_2$ as desired.  It is $K$--non-backtracking since its scwolification is the path $q_{p'}$.



Conversely, suppose that there is a $K$--non-backtracking $CG(\mc{Y})$--loop
\[	p_{\orbit{\alpha}} \cdot g_1 \cdot p_{\orbit{\beta}} \cdot g_2	,	\]
which represents an element of $K$.  Then $p_{\orbit{\alpha}} \cdot g_1 \cdot p_{\orbit{\beta}} \cdot g_2$ lifts to a loop in $\mc{C}_K$.  The scwolification of this loop gives a path $p'$ as in condition \eqref{cell2loop}.
\end{proof}

The following is elementary.

\begin{lemma} \label{l:length 3 no backtrack}
Let $Q$ be a complex so that there are no loops of length $1$ or $2$ in its $1$--skeleton.  Any loop in $Q$ of length $3$ is non-backtracking.
\end{lemma}
The utility of Lemma \ref{l:length 3 no backtrack} is that once we have found conditions to ensure that links in $\leftQ{X}{K}$ have no loops of length $1$ or $2$ then loops of length $3$ are automatically non-backtracking.

Given Lemma \ref{l:length 3 no backtrack}, the following is proved in the same way as Lemma \ref{l:no loop of length 2}.

\begin{lemma} \label{l:no loop of length 3 - dont fill}
Suppose that $\link(\check\sigma)$ is simplicial, and suppose that $p$ is a path in $\link(\check\sigma)$ which is a concatenation of three $1$--cells, $\check\alpha$, $\check\beta$ and $\check\gamma$ with lifts $\alpha$, $\beta$ and $\gamma$ to $X$.  The following are equivalent:
\begin{enumerate}
\item There is a path $p' = \check\alpha'\cdot\check\beta'\cdot\check\gamma'$ in $\link(\check\sigma)$ so that $\orbit{\check\alpha'} = \orbit{\check\alpha}$, $\orbit{\check\beta'} = \orbit{\check\beta}$ and $\orbit{\check\gamma'} = \orbit{\check\gamma}$, and $p'$ is an immersed loop.
\item There is a $CG(\mc{Y})$--loop of the form
 $p_{\orbit{\alpha}} \cdot g_1 \cdot p_{\orbit{\beta}} \cdot g_2 \cdot p_{\orbit{\gamma}} \cdot g_3$ that represents a conjugacy class in $K$.
\end{enumerate}
  Moreover, in case these conditions hold, the path $p'$ can be chosen to be the scwolification of a lift of  $p_{\orbit{\alpha}} \cdot g_1 \cdot p_{\orbit{\beta}} \cdot g_2 \cdot p_{\orbit{\gamma}} \cdot g_3$ (and conversely  $p_{\orbit{\alpha}} \cdot g_1 \cdot p_{\orbit{\beta}} \cdot g_2 \cdot p_{\orbit{\gamma}} \cdot g_3$ is the $CG(\mc{Y})$--path which labels the unscwolification of $p'$).
\end{lemma}

If $X$ has dimension greater than $2$, there are certainly some $\sigma$ so that there are loops of length $3$ in $\link(\sigma)$.  This introduces some subtleties, which we discuss in the next subsection.

\subsubsection{Loops of length $3$ filled by $2$--cells}
We assume for the rest of the section that cubes of $Z$ are embedded, so that objects of $\mc{Z}$ can be unambiguously described by chains of cubes of $Z$.
The phenomenon we are concerned with in this section is illustrated by the following example.
\begin{example}
  Let $\mc{Y}$ be a single $2$--simplex, and consider the complex of groups $G(\mc{Y})$ so that $G_v\cong \bZ$ for each vertex $v$, and all the other local groups are trivial.  Let $x,y,z\in \pi_1(G(\mc{Y}))$ generate the three vertex groups.
  The universal cover $X$ of $G(\mc{Y})$ is an infinite valence ``tree of triangles''.  Let $K = \llangle x^3,y^3,z^3,xyz\rrangle$.  Then $\leftQ{X}{K}$ can be realized as a subset of the Euclidean plane, consisting of every other triangle of a tessellation by equilateral triangles.  Moreover, if $\alpha \beta\gamma$ is the path in the $1$--skeleton of $\mc{Y}$ labeling the boundary of $\mc{Y}$, there are paths in $\leftQ{X}{K}$ projecting to $\alpha \beta\gamma$, but which are not filled by a $2$--cell in $\leftQ{X}{K}$.  The issue here, as we will see, is that $xyz\in K$ is not an element of $K_xK_yK_z$, where $K_x = K\cap \langle x \rangle$, and so on.

  Of course $X$ is not a cube complex, but it can be realized as the link of a vertex of a cube complex, covering a complex of groups in which $G(\mc{Y})$ is embedded.
\end{example}

\begin{definition} \label{def:tau,alpha corner}
  Let $\check\sigma$ be a cube of $Z$, and let $\check\tau$ be a $2$--cell in $\link(\check\sigma)$.  Then $\partial\check\tau$ is a loop composed of three oriented $1$--cells $\check\alpha\cdot\check\beta\cdot\check\gamma$, which we may lift to $1$--cells $\alpha$, $\beta$, and $\gamma$, forming the boundary of a lift $\tau$ of $\check\tau$ in $X$.  These $1$--cells are associated to $CG(\mc{Y})$--paths $p_{\orbit{\alpha}},p_{\orbit{\beta}},p_{\orbit{\gamma}}$ as in Definition \ref{def:palpha}.  Consider a $CG(\mc{Y})$--path of the form $q = p_{\orbit{\alpha}}\cdot g\cdot p_{\orbit{\beta}}$.  Let $\widehat{q}$ be a lift to $\mc{C}_K$.  The realization of $\Theta_K(\widehat{q})$ is a concatenation of two $1$--cells $\check\alpha'\cdot\check\beta'$.
  We say that $q$ \emph{$K$--bounds a $(\tau,\alpha)$--corner} if there is a cube $\check\sigma'$, a $2$--cell $\check\tau'$ in $\link(\check\sigma')$, and an $h\in G$ so that $\check\sigma' = h \check\sigma$, $\check\tau' = h \check\tau$, $\check\alpha' = h \check\alpha$ and $\check\beta' = h \check\beta$.  If there is some $(\tau,\alpha)$ for which the path $q$ $K$--bounds a $(\tau,\alpha)$--corner, we may just say $q$ \emph{$K$--bounds a corner}.
\end{definition}

In case there exists a path $q$ as above which $K$--bounds a $(\tau,\alpha)$--corner, there are $X$--cubes
$\epsilon, \phi_\alpha, \phi_\beta$ and $\psi$, all containing $\sigma$, so that $\epsilon \subset \phi_\alpha, \phi_\beta \subset \psi$, and $\dim(\psi) = \dim(\phi_\alpha) + 1 = \dim(\phi_\beta) + 1 = \dim(\epsilon) + 2 = \dim (\sigma) + 3$.  There is a copy of the link of $\epsilon$ contained in the link of $\sigma$.  The cubes $\epsilon, \phi_\alpha, \phi_\beta$ and $\psi$ determine an oriented $1$--cell $\zeta$ in this copy of the link of $\epsilon$.  Its idealization is shown in Figure \ref{fig:zeta}.  
\begin{figure}[htbp]
  \centering
  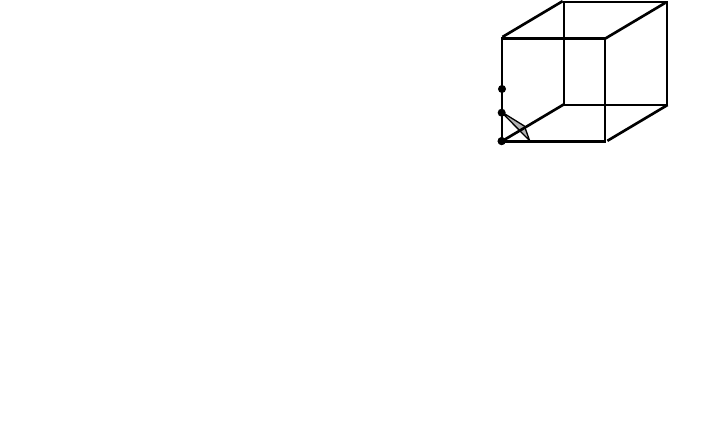
  \caption{A part of $\mc{X}$ representing part of the link of $\sigma$, containing the idealization of the $1$--cell $\zeta$ in green.  Directions of most arrows have been omitted.}
  \label{fig:zeta}
\end{figure}
The idealization of $\zeta$ begins at the object $(\sigma\subset \epsilon \subset \phi_\alpha)$ and ends at $(\sigma\subset \epsilon \subset \phi_\alpha)$.  Let $a_\alpha$ be the arrow pointing from $(\sigma\subset \epsilon \subset \phi_\alpha)$ to $(\sigma\subset \epsilon)$, and let $a_\beta$ be the arrow pointing from $(\sigma\subset \epsilon \subset \phi_\beta)$ to $(\sigma\subset \epsilon)$.  These arrows project to arrows $\orbit{a_\alpha}$ and $\orbit{a_\beta}$ in $\mc{Y}$, and the path $p_{\orbit{\zeta}}$ (defined as in Definition \ref{def:palpha}) travels from $i(\orbit{a_\alpha})$ to $i(\orbit{a_\beta})$.

\begin{lemma} \label{lem:g tau,alpha}
 The $CG(\mc{Y})$--loop
 \[	\orbit{a_{\alpha}}^+ \cdot p_{\orbit{\zeta}} \cdot \orbit{a_{\beta}}^-	,	\]
  (which is based at $\orbit{\sigma\subset \epsilon}$) represents an element of $G_{\orbit{\sigma\subset \epsilon}}$.
\end{lemma}
\begin{proof}
  All the chains which occur in this proof have the same minimal element $\sigma$, so we omit the prefix `$\sigma\subset$' from all chains until the end of the proof of the Lemma.  
  We therefore have a diagram in $\lk(\sigma)$ in the scwol $\mc{X}$ as follows:

\cd {
& & (\epsilon) & &\\
(\epsilon \subset \phi_\alpha) \ar[urr]^{a_\alpha} & & & & (\epsilon \subset \phi_\beta) \ar[ull]^{a_\beta} \\
& (\epsilon \subset \phi_\alpha \subset \psi) \ar[ul]^{a_1} \ar[dr]^{a_2} \ar[uur]^{b_1} & & (\epsilon \subset \phi_\beta \subset \psi) \ar[dl]^{a_3} \ar[ur]^{a_4} \ar[uul]^{b_3} \\
& & (\epsilon \subset \psi) \ar[uuu]^{b_2} \\
}

We have the following identities in the category $\mc{X}$:  $a_\alpha a_1 = b_1 = b_2a_2$ and $a_\beta a_4 = b_3 = b_2a_3$.  The path in the statement of the lemma is equal to:
\[	\orbit{a_{\alpha}}^+ \cdot \orbit{a_1}^+ \cdot \orbit{a_2}^- \cdot g_{(\epsilon \subset \psi)} \cdot \orbit{a_3}^+ \cdot \orbit{a_4}^- \cdot \orbit{a_{\beta}}^-	,\]
where $g_{(\epsilon \subset \psi)}$ is the element of $G_{\orbit{(\epsilon \subset \psi)}}$ chosen for the path $p_\zeta$ as in Definition \ref{def:palpha}.

Define the following elements of $G_{\orbit{\epsilon}}$:
\begin{eqnarray*}
 h_1 &=& z\left(\orbit{a_\alpha},\orbit{a_1}\right) z\left(\orbit{b_2},\orbit{a_2}\right)^{-1},\\
 h_2 &=& h_1 \psi_{\orbit{b_2}}(g_{(\epsilon \subset \psi)}),\\
 h_3 &=& h_2 z\left(\orbit{b_2},\orbit{a_3}\right) z\left(\orbit{a_\beta}, \orbit{a_4}\right)^{-1},
 \end{eqnarray*}
 where the $z(\orbit{a},\orbit{b})$ are the twisting elements determined by the complex of groups structure on $G(\mc{Y})$.
 
We now have the following sequence of elementary homotopies of $CG(\mc{Y})$--paths (all of which consist of applying the moves in Definition \ref{def:elem homotopy}, and the rule of arrow composition in $CG(\mc{Y})$ from Definition \ref{def:CG(Y)}).

\begin{eqnarray*}
\orbit{a_{\alpha}}^+ \cdot p_{\orbit{\zeta}} \cdot \orbit{a_{\beta}}^- &\simeq& \orbit{a_{\alpha}}^+ \cdot \orbit{a_1}^+ \cdot \orbit{a_2}^- \cdot g_{(\epsilon \subset \psi)} \cdot \orbit{a_3}^+ \cdot \orbit{a_4}^- \cdot \orbit{a_{\beta}}^- \\
&\simeq& z(\orbit{a_{\alpha}} , \orbit{a_1}) \cdot  \orbit{b_1}^+ \cdot \orbit{a_2}^- \cdot g_{(\epsilon \subset \psi)} \cdot \orbit{a_3}^+ \cdot \orbit{a_4}^- \cdot \orbit{a_{\beta}}^- \\
&\simeq& {\scriptstyle z(\orbit{a_{\alpha}} , \orbit{a_1})z(\orbit{b_2},\orbit{a_2})^{-1} \cdot \orbit{b_2}^+ \cdot \orbit{a_2}^+ \cdot \orbit{a_2}^- \cdot g_{(\epsilon \subset \psi)} \cdot \orbit{a_3}^+ \cdot \orbit{a_4}^- \cdot \orbit{a_{\beta}}^-}\\
&\simeq& h_1 \cdot \orbit{b_2}^+ \cdot g_{(\epsilon \subset \psi)} \cdot \orbit{a_3}^+ \cdot \orbit{a_4}^- \cdot \orbit{a_{\beta}}^-\\
&\simeq& h_1 \cdot \psi_{\orbit{b_2}}(g_{(\epsilon \subset \psi)}) \cdot  \orbit{b_2}^+ \cdot \orbit{a_3}^+ \cdot \orbit{a_4}^- \cdot \orbit{a_{\beta}}^- \\
&\simeq& h_2z(\orbit{b_2},\orbit{a_3}) \cdot \orbit{b_3}^+ \cdot \orbit{a_4}^- \cdot \orbit{a_{\beta}}^- \\
&\simeq& h_2z(\orbit{b_2},\orbit{a_3})z(\orbit{a_{\beta}},\orbit{a_4})^{-1} \cdot \orbit{a_{\beta}}^+ \cdot \orbit{a_4}^+ \cdot \orbit{a_4}^- \cdot \orbit{a_{\beta}}^- \\
&\simeq& h_3 \cdot \orbit{a_{\beta}}^+ \cdot \orbit{a_{\beta}}^- \\
&\simeq& h_3 
\end{eqnarray*} 
This proves the result.
\end{proof}

\begin{notation} \label{not:g_tau,alpha}
The element of $G_{\orbit{\epsilon}}$ represented by $\orbit{a_{\alpha}}^+  \cdot  p_{\orbit{\zeta}}  \cdot  \orbit{a_{\beta}}^-$ is denoted by $g_{\tau,\alpha}$.
\end{notation}

\begin{lemma} \label{l:tau-loop}
 A path $p_{\orbit{\alpha}}  \cdot  g  \cdot  p_{\orbit{\beta}}$ $K$--bounds a $(\tau,\alpha)$--corner if and only if there exists a $CG(\mc{Y})$--loop
 \begin{equation}\label{tau-loop}
   \orbit{a_{\alpha}}^+ \cdot g_1  \cdot  p_{\orbit{\zeta}}  \cdot  g_2  \cdot  \orbit{a_{\beta}}^-  \cdot  g^{-1}
 \end{equation}
 which represents an element of $K$.
\end{lemma}
\begin{proof}
 First suppose that there is a $CG(\mc{Y})$--loop of the form \eqref{tau-loop} representing an element of $K$.

Then the following two $CG(\mc{Y})$--paths differ by an element of $K$:
\[	p_{\orbit{\alpha}}  \cdot  g  \cdot  p_{\orbit{\beta}}	,	\]
and
\[	p_{\orbit{\alpha}}  \cdot  \orbit{a_{\alpha}}^+  \cdot g_1  \cdot  p_{\orbit{\zeta}}  \cdot  g_2  \cdot  \orbit{a_{\beta}}^-  \cdot  g^{-1}  \cdot  g  \cdot  p_{\orbit{\beta}}	.	\]
Thus they together form a loop which lifts to $\mc{C}_K$.

This second path is homotopic to a $CG(\mc{Y})$--path whose scwolification avoids the vertex $t(\orbit{\alpha})$ after $p_{\orbit{\alpha}}$ but instead travels across the first three edges of $p_{\orbit{\alpha}}$, traverses $p_{\orbit{\zeta}}$, and then travels across the final three edges of $p_{\orbit{\beta}}$.  The homotopy lifts to $\mc{C}_K$, and the image in $\mc{Z}$ of this homotopy under the scwolification $\Theta_K$ shows that there is a $2$--cell $\check\tau'$ between the $1$--cells $\check\alpha'$ and $\check\beta'$ whose idealizations are the scwolifications of the lifts of $p_{\orbit{\alpha}}$ and $p_{\orbit{\beta}}$ respectively.  This shows that the path $p_{\orbit{\alpha}}  \cdot  g  \cdot  p_{\orbit{\beta}}$ $K$--bounds a $(\tau,\alpha)$--corner.

Conversely, suppose that the $CG(\mc{Y})$--path $q = p_{\orbit{\alpha}}  \cdot  g  \cdot  p_{\orbit{\beta}}$ $K$--bounds a $(\tau,\alpha)$--corner.  Lift to a $\mc{C}_K$--path $\widehat{q}$ and consider the scwolification $\Theta_K(\widehat{q})$ in $\mc{Z}$.  As in Definition \ref{def:tau,alpha corner}, the realization of $\widehat{q}$ is the concatenation of two $1$--cells $\check\alpha',\check\beta'$ in $\link(\check\sigma')$ for some cube $\check\sigma'$ in the orbit of $\check\sigma$.  
Moreover, there is a $2$--cell $\check\tau'$ with $\check\alpha',\check\beta'$ in the boundary of $\check\tau'$ and an element $h$ of $G$ so that $\check\sigma' = h \check\sigma$, $\check\tau' = h \check\tau$, $\check\alpha' = h \check\alpha$ and $\check\beta' = h \check\beta$.  Let $v'$ be the vertex of $\link(\check\sigma')$ where $\check\alpha'$ and $\check\beta'$ meet.

Consider the loop $q_0 = \orbit{a_{\alpha}}^+  \cdot  p_{\orbit{\zeta}}  \cdot  \orbit{a_{\beta}}^-$ as in Lemma \ref{lem:g tau,alpha}.  This represents an element of $G_{t(\orbit{\alpha})}$, and there is a lift $\widehat{q}_0$ of $q_0$ to $\mc{C}_K$ so that $\Theta_K(\widehat{q}_0)$ is a loop based at $v'$ and traveling across the corner of $\check\tau'$ from $\check\alpha'$ to $\check\beta'$.  The paths $q_0$ and $q$ have lifts to $\mc{C}_K$ forming a sub-diagram:
\[
  \xy
  (0,0)*{\cdot} = "a";
  (10,10)*{\cdot} = "b";
  (20,20)*{\odot} = "c";
  (30,30)*{\cdot} = "d";
  (40,40)*{\cdot} = "e";
   (43,37)*{\cdot} = "k";
   (50,40)*{\cdot} = "f";
   (47,37)*{\cdot} = "l";
  (60,30)*{\cdot} = "g";
  (70,20)*{\odot} = "h";
  (80,10)*{\cdot} = "i";
  (90,0)*{\cdot} = "j";
  {\ar@/^/_g "e";"f"};
  {\ar "b";"a"};{\ar "b";"c"};{\ar "d";"c"};{\ar^{a_\alpha} "d";"e"};
  {\ar_{a_\beta} "g";"f"};{\ar "g";"h"};{\ar "i";"h"};{\ar "i";"j"};
  (33,27)*{\cdot} = "m";
  (38,24)*{\cdot} = "n";
  (45,23)*{\odot} = "o";
  (52,24)*{\cdot} = "p";
  (57,27)*{\cdot} = "q";
  (45,30)*{\widehat{q}_0};
  {\ar@{.>}@/_1pc/ (40,30)*{};(50,30)*{}};
  (10,25)*{\widehat{q}};
  {\ar@{.>} (10,18)*{};(17,25)*{}};
  {\ar "n";"m"};{\ar "n";"o"};{\ar "p";"o"};{\ar "p";"q"};{\ar "m";"k"};{\ar "q";"l"};
    \endxy
  \]
  The circled dots represent either single objects or pairs of objects separated by a group arrow, depending on whether the paths $p_{\orbit{x}}$ have length four or five for $x\in\{\alpha,\beta,\zeta\}$.  The scwolification of this diagram in $\mc{Z}$ looks like this:
\[
  \xy
  (0,0)*{\cdot} = "a2";
  (10,10)*{\cdot} = "b2";
  (20,20)*{\cdot} = "c2";
  (30,30)*{\cdot} = "d2";
  (40,40)*{\cdot} = "e2";
  (40,40)*{\cdot} = "f2";
  (50,30)*{\cdot} = "g2";
  (60,20)*{\cdot} = "h2";
  (70,10)*{\cdot} = "i2";
  (80,0)*{\cdot} = "j2";
  {\ar "b2";"a2"};{\ar "b2";"c2"};{\ar "d2";"c2"};{\ar^{a_\alpha} "d2";"e2"};
  {\ar_{a_\beta} "g2";"f2"};{\ar "g2";"h2"};{\ar "i2";"h2"};{\ar "i2";"j2"};
  (35,25)*{\cdot} = "n2";
  (40,23)*{\cdot} = "o2";
  (45,25)*{\cdot} = "p2";
  {\ar "n2";"d2"};{\ar "n2";"o2"};{\ar "p2";"o2"};{\ar "p2";"g2"};
 \endxy
\]

The edges which scwolify to $a_\alpha$ in $\widehat{q}$ and $\widehat{q}_0$ have sources connected by a group arrow labeled by some $g_1$.  Similarly the edges which scwolify to $a_\beta$ have sources connected by group arrow with some label $g_2$.  We thus obtain a loop in $\mc{C}_K$ of the form \eqref{tau-loop}.
\end{proof}

Given the criterion from Lemma \ref{l:tau-loop}, the following result is straightforward.   Recall the definition of the element $g_{\tau,\alpha}$ from Notation \ref{not:g_tau,alpha}.

\begin{proposition} \label{p:bound a corner}
Suppose that $\tau$ is a $2$--cell in $\link(\sigma)$ and that the boundary of $\tau$ is $\alpha.\beta.\gamma$.  For any $g \in G_{t\orbit{\alpha}}$ the $CG(\mc{Y})$--path $p_{\orbit{\alpha}}  \cdot  g  \cdot  p_{\orbit{\beta'}}$ $K$--bounds a $(\tau,\alpha)$--corner if and only if 
\begin{enumerate}
\item $\orbit{\beta'} = \orbit{\beta}$; and
\item $g \in E_{\orbit{\alpha}} g_{\tau,\alpha} E_{\orbit{\backwards{\beta}}} . K_{t(\orbit{\alpha})}$
\end{enumerate}
\end{proposition}
\begin{proof}
Recall from Notation \ref{not:g_tau,alpha} that $g_{\tau,\alpha}$ is the element of $G_{t\orbit{\alpha}}$ represented by the $CG(\mc{Y})$--loop
\[	\orbit{a_{\alpha}}^+  \cdot  p_{\orbit{\gamma}}  \cdot  \orbit{a_{\beta}}^- .	\]

Suppose that $p_{\orbit{\alpha}}  \cdot  g  \cdot  p_{\orbit{\beta}}$ $K$--bounds a $(\tau,\alpha)$--corner.  Then consider the path
 \[	\orbit{a_{\alpha}}^+  \cdot g_1  \cdot  p_{\orbit{\gamma}}  \cdot  g_2  \cdot  \orbit{a_{\beta}}^-  \cdot  g^{-1}	\]	
from Lemma \ref{l:tau-loop} which represents an element of $K$.

We have homotopies
\begin{eqnarray*}
 \orbit{a_{\alpha}}^+  \cdot g_1  \cdot  p_{\orbit{\gamma}}  \cdot   g_2  \cdot  \orbit{a_{\beta}}^-  \cdot  g^{-1} &\simeq& \psi_{\orbit{a_{\alpha}}}(g_1)  \cdot \orbit{a_{\alpha}}^+  \cdot  p_{\orbit{\gamma}}  \cdot  \orbit{a_{\beta}}^-  \cdot  \left( \psi_{\orbit{a_{\beta}}}(g_2)g^{-1}\right) \\
 &\simeq& \psi_{\orbit{a_{\alpha}}}(g_1)  \cdot  g_{\tau,\alpha}  \cdot  \left( \psi_{\orbit{a_{\beta}}}(g_2)g^{-1} \right)	\\
 &\simeq& \psi_{\orbit{a_{\alpha}}}(g_1)g_{\tau,\alpha}\psi_{\orbit{a_{\beta}}}(g_2)g^{-1} .
\end{eqnarray*}
Since $\psi_{\orbit{a_{\alpha}}}(g_1) \in E_{\orbit{\alpha}}$, $\psi_{\orbit{a_{\beta}}}(g_2) \in E_{\orbit{\beta}}$ and the whole expression above is an element of $K \cap G_{\orbit{v}} = K_{t(\orbit{\alpha})}$, we have
\[	g \in E_{\orbit{\alpha}} g_{\tau,\alpha} E_{\orbit{\beta}} K_{t(\orbit{\alpha})}	,	\]
as required.

In order to prove the other direction, this computation may be performed in reverse.
\end{proof}

\begin{lemma} \label{lem:bounds a 2-cell}
  Suppose that $\link(\check\sigma)$ is simplicial and contains $1$--cells $\check\alpha$, $\check\beta$, and $\check\gamma$ which lift respectively to $1$--cells $\alpha,\beta$ and $\gamma$ in $\link(\sigma)$ in $X$.  Let $$q = p_{\orbit{\alpha}}  \cdot  g_1  \cdot  p_{\orbit{\beta}}  \cdot  g_2  \cdot  p_{\orbit{\gamma}}  \cdot  g_3$$ be a $CG(\mc{Y})$--loop which represents an element of $K$.  Suppose $\check\eta\subset \link(\check\sigma)$ is the realization of the scwolification of some lift of $q$ to $\mc{C}_K$.

  If any one of $p_{\orbit{\alpha}}  \cdot  g_1  \cdot  p_{\orbit{\beta}}$, $p_{\orbit{\beta}}  \cdot  g_2  \cdot  p_{\orbit{\gamma}}$ or $p_{\orbit{\gamma}}  \cdot  g_3  \cdot  p_{\orbit{\alpha}}$ $K$--bounds a corner, then $\check\eta$ bounds a $2$--cell in $\link(\check\sigma)$.
\end{lemma}
\begin{proof}
  Note that since $q$ represents an element of $K$, any lift to $\mc{C}_K$ is a loop, and so the realization $\check\eta$ is also a loop.  Since $\link(\check\sigma)$ is simplicial, this loop is embedded of length $3$ in $\link(\check\sigma)$ by Lemma \ref{l:length 3 no backtrack}.

  Think of $q$ as given by a cyclic word in the arrows of $CG(\mc{Y})$, and suppose that one of the three given subpaths of $q$ $K$--bounds a corner.  By relabeling and cyclically rotating we can assume it is the subpath $p = p_{\orbit{\alpha}}\cdot g_1 \cdot p_{\orbit{\beta}}$, so there is some $2$--cell $\check\tau$ in $\link(\check\sigma)$ and lift $\tau$ to $\link(\sigma)$ and $p$ $K$--bounds a $(\tau,\alpha)$--corner.  It follows that some translate $\check\tau'$ of $\check\tau$ in $\link(\check\sigma)$ has boundary given by a path $\check\alpha'\cdot\check\beta'\cdot\check\gamma'$, where $\alpha'\cdot\beta'$ are the first two $1$--cells of the path $\check\eta$.  If the third $1$--cell of $\partial\check\tau'$ is not the third $1$--cell of $\check\eta$, we obtain $1$--cells in $\link(\check\sigma)$ with the same endpoints, contradicting the assumption that $\link(\check\sigma)$ is simplicial.  So $\eta$ bounds the $2$--cell $\check\tau'$.
\end{proof}

Since there are finitely many $\Stab(\check\sigma)$--orbits of $2$--cell in $\link(\check\sigma)$, we obtain the following.

\begin{proposition} \label{prop:no loop of length 3 - do fill}
  Suppose that $\link(\check\sigma)$ is simplicial.  There are finitely many $2$--cells $\check\tau_i$ in $\link(\check\sigma)$ (with boundary $\check\alpha_i\cdot\check\beta_i\cdot\check\gamma_i$, and lifts $\alpha_i,\beta_i, \gamma_i$ to $\link(\sigma)$) so that the following holds:
  
Every loop of length $3$ in $\link(\check\sigma)$ is filled by a $2$--cell if and only if 
for every $CG(\mc{Y})$-path
\begin{equation}\label{eq:form of path}
\tag{$\ast$}		p_{\orbit{\alpha}}  \cdot  g_1  \cdot  p_{\orbit{\beta}}  \cdot  g_2  \cdot  p_{\orbit{\gamma}}  \cdot  g_3 	
\end{equation}
which represents an element of $K$, there exists an $i$ so that
\begin{enumerate}
\item $\orbit{\alpha} = \orbit{\alpha_i}, \orbit{\beta} = \orbit{\beta_i}$ and $\orbit{\gamma} = \orbit{\gamma_i}$;
\item $g_1 \in E_{\orbit{\alpha_i}} g_{\tau_i,\alpha_i} E_{\orbit{\backwards{\beta_i}}} K_{t(\orbit{\alpha_i})}$;
\item $g_2 \in E_{\orbit{\beta_i}} g_{\tau_i,\beta_i} E_{\orbit{\backwards{\gamma_i}}} K_{t(\orbit{\beta_i})}$; and
\item $g_3 \in E_{\orbit{\gamma_i}} g_{\tau_i,\gamma_i} E_{\orbit{\backwards{\alpha_i}}} K_{t(\orbit{\gamma_i})}$.
\end{enumerate}
\end{proposition}
\begin{proof}
Choose  the $2$--cells $\check\tau_i$ to be representatives of the $\Stab(\check\sigma)$--orbits of $2$--cells (together with a fixed vertex to label the boundary -- so that a single orbit may appear up to three times in the list).

Suppose first that the condition about paths of form \eqref{eq:form of path} representing elements of $K$ is satisfied, and suppose that $p$ is a loop of length $3$ in $\link(\check\sigma)$ which is labeled by $1$--cells $\check\alpha', \check\beta', \check\gamma'$, in order.  By Lemma \ref{l:no loop of length 3 - dont fill} there exists a $CG(\mc{Y})$--path $\lambda$ of the form \eqref{eq:form of path} which is the label of an unscwolification of $p$.  Because of our hypothesis, there exists an $i$ so that conditions (1)--(4) are satisfied.  By Proposition \ref{p:bound a corner} the $CG(\mc{Y})$--path $\lambda$ $K$--bounds a corner at each of its three corners, and so by Lemma \ref{lem:bounds a 2-cell} the path $p$ bounds a $2$--cell, as required.

Conversely, suppose that every loop of length $3$ in $\link(\check\sigma)$ bounds a $2$--cell, and consider a $CG(\mc{Y})$--path $\lambda$ of the form \eqref{eq:form of path} which represents an element of $K$.  By Lemma \ref{l:no loop of length 3 - dont fill} the scwolification of a lift of $\lambda$ is (the idealization of) an immersed path of length $3$.  This immersed path must then bound a $2$--cell $\check\tau$.  Suppose that $\check\tau_i$ is the representative in the $\Stab(\check\sigma)$--orbit of the $2$--cell $\check\tau$, so condition (1) is satisfied.  According to Lemma \ref{l:tau-loop}, applied to all three corners of this $2$--cell, the path $\lambda$ satisfies conditions (2)--(4).  This finishes the proof.
\end{proof}

To summarize, given Lemma~\ref{l:L flag}, Lemmas~\ref{l:loop of length 1},~\ref{l:no loop of length 2},~\ref{l:no loop of length 3 - dont fill}, and Proposition~\ref{prop:no loop of length 3 - do fill} give descriptions of various types of $CG(\mc{Y})$--paths so that the cube complex $Z=\leftQ{X}{K}$ is nonpositively curved if and only if no such path lifts to $\mc{C}_K$.

\section{Algebraic translation} \label{s:algebra}

In this section, we continue to work in the context of a group $G$ acting cocompactly on a CAT$(0)$ cube complex $X$.  The induced action on the associated scwol $\mc{X}$ has quotient scwol $\mc{Y}$, the underlying scwol for a complex of groups structure $G(\mc{Y})$ on $G$.  We let $\mc{Q}(G)$ be the set of cube stabilizers for $G\acts X$; equivalently $\mc{Q}(G)$ is the set of conjugates of the local groups for the complex of groups $G(\mc{Y})$.

We translate the conditions from the previous section into algebraic statements about elements of $G$ and of $\mc{Q}(G)$, with an eye toward finding conditions on $K\lhd G$ so that $\leftQ{X}{K}$ is non-positively curved.  In Section~\ref{s:Dehn filling} we use hyperbolic Dehn filling to find $K$ which satisfy the conditions, under certain hyperbolicity assumptions on $G$ and $\mc{Q}(G)$.

We fix a basepoint $v_0$ for $\mc{Y}$ and an isomorphism $\pi_1(CG(\mc{Y}),v_0) \cong G$ as in Section~\ref{s:complex of groups}.  
The scwolification functor
\[ \Theta \co \cgyt\to \mc{X} \]
is $G$--equivariant.  Recall also that the objects of $\cgyt$ are homotopy classes of paths starting at $v_0$.

Fix also a maximal (undirected) tree $T$ in $\mc{Y}$.  For each object $v$ of $\mc{Y}$ which represents an orbit of cubes in $X$, let $c_v$ be the unique $\mc{Y}$--path in $T$ from $v_0$ to $v$.  By using scwol arrows, we also consider $c_v$ to be a $CG(\mc{Y})$--path in the natural way.  For an object $v$ of $\mc{Y}$ which represents a chain of cubes of length longer than $1$, we define a $\mc{Y}$--path $c_{v}$ from $v_0$ to $v$ as follows:  If $v$ is represented by $( \sigma_1 \subset \sigma_2 \subset \cdots \subset \sigma_k)$ (a nested chain of cubes in $X$) then define $c_{v}$ to be the concatenation of $c_{\orbit{\sigma_1}}$ with the path consisting of the arrows $(\sigma_1 \subset \cdots \subset \sigma_i) \to (\sigma_1 \subset \cdots \subset \sigma_{i+1})$, for $i = 1, 2, \ldots , k-1$.

We use the paths $c_v$ to define a map from (homotopy classes rel endpoints of) $CG(\mc{Y})$--paths to (homotopy classes of) $CG(\mc{Y})$--loops based at $v_0$ by
\[	p \mapsto c_{i(p)} \cdot p \cdot \backwards{c_{t(p)}}	.	\]
(Here and below we use the notation $\backwards{c}$ to denote the reverse of the $CG(\mc{Y})$--path $c$.)

Given a path $p$, let $\ell_p = \left[ c_{i(p)} \cdot p \cdot \backwards{c_{t(p)}} \right]\in \pi_1(CG(\mc{Y}),v_0)$.

The following results are all straightforward.

\begin{lemma}
For any $CG(\mc{Y})$--paths $p, p'$ so that $t(p) = i(p')$, we have
\begin{align*}
\ell_{\backwards{p}} &= \ell_p^{-1}\\
\ell_{p \cdot p'} &= \ell_p \ell_{p'}.
\end{align*}
\end{lemma}

\begin{lemma}
  Suppose that $p$ is a $CG(\mc{Y})$--path starting at $v_0$.  Let $[p]$ be the equivalence class of $p$ in $\cgyt$, and let $x = \Theta([p])$.  Then
\[	\Stab_G(x) = \left\{ [ p \cdot g \cdot \backwards{p} ] \mid g \in G_{\orbit{x}} \right\}	.	\]
\end{lemma}

\begin{definition} \label{def:Q_v}
Given an object $v$ of $\mc{Y}$, define
\[	Q_v = \left\{ [c_v \cdot g \cdot \backwards{c_v}] \mid g \in G_v \right\}	.	\]
\end{definition}
Definition \ref{def:Q_v} gives an explicit identification of the local groups of the complex of groups $G(\mc{Y})$ with finitely many elements of $\mc{Q}(G)$.

\subsection{Algebraic formulation of the link conditions}
Suppose that $K \unlhd G$.
In order for $Z = \leftQ{X}{K}$ to be non-positively curved,
there are five conditions that need to be ensured on links in $Z$.  Roughly speaking, they are:
\begin{enumerate}
\item No loop of length $1$,
\item No loop of length $2$ consisting of $1$--cells in different $G$--orbits,
\item No loop of length $2$ consisting of $1$--cells in the same $G$--orbit,
\item No loop of length $3$ whose image in $\mc{Y}$ does not bound a $2$--cell, and
\item No loop of length $3$ which does not bound a $2$--cell but whose image in $\mc{Y}$ does bound a $2$--cell.
\end{enumerate}
More precisely, the ``image in $\mc{Y}$'' means the image in $\mc{Y}$ of the idealization.  And we say this image $p$ ``bounds a $2$--cell'' if there is an unscwolification $\widehat{p}$ and a lift $\widetilde{p}$ of $\widehat{p}$ to $\cgyt$ so that the realization of the scwolification of $\widetilde{p}$ bounds a $2$--cell in some link of a cube in $X$.
  
If $\leftQ{X}{K}$ is a simply-connected cube complex and we ensure each of these conditions, then Lemmas \ref{l:L simplicial}, \ref{l:L flag} and \ref{l:length 3 no backtrack} imply that $\leftQ{X}{K}$ is CAT$(0)$.

In this subsection, we formulate five results which give algebraic conditions to enforce each of these five conditions in turn.  These results follow quickly from the results in Section \ref{s:X mod K CAT(0)} using the translation from the beginning of this section.  In each case, since $G$ acts cocompactly on a CAT$(0)$ cube complex, there are finitely many $G$--orbits of links and in each link finitely many $G$--orbits of each of the five kinds of paths in the above list, and we can rule out each orbit behaving badly in $\leftQ{X}{K}$ in turn.



 \begin{assumption} \label{SA}
   The group $G$ acts cocompactly on the CAT$(0)$ cube complex $X$, and $\mc{Q}(G)$ is the collection of cell stabilizers of the action.
 \end{assumption}

\begin{terminology} \label{term:tame}
  Under Assumption \ref{SA}, a normal subgroup $K \unlhd G$ is {\em co-cubical} if $\leftQ{X}{K}$ is a cube complex.
\end{terminology}

The following is a straightforward translation of Lemma \ref{l:loop of length 1}.  We spell out the proof since we use similar techniques for other more complicated results later in the section.

\begin{theorem} \label{t:NPC 1}
Under Assumption \ref{SA} there exists a finite set $F_1 \subset \mc{Q}(G) \times G$ so that for each $(Q,p) \in F_1$ we have $p \not\in Q$ and so that if (i) $K \unlhd G$ is co-cubical; and (ii) for each $(Q,p) \in F_1$ we have $p \not\in Q.K$, then no link in $\leftQ{X}{K}$ contains a loop of length $1$.
\end{theorem}
\begin{proof}
  Up to the action of $G$, there are finitely many pairs $(\widetilde\sigma,\widetilde\alpha)$, where $\widetilde\sigma$ is a cube of $X$ and $\widetilde\alpha$ is a $1$--cell in $\lk(\widetilde\sigma)$ whose endpoints are identified by some element of $G$.  For each such pair we will give a pair $(Q,p)$ as in the statement of the theorem.

  For such a pair, let $(\sigma,\alpha)$ be the image in $\leftQ{X}{K}$.  Since $K$ is assumed to act co-cubically, $\alpha$ is embedded in $\lk(\sigma)$, except that its endpoints may have been identified, making it a loop.  According to Lemma \ref{l:loop of length 1}, $\alpha$ is a loop if and only if there is a $CG(\mc{Y})$--loop of the form $p_{\orbit{\alpha}}.g$ that represents a conjugacy class in $K$.  In particular, this condition only depends on the orbit $\orbit{\alpha}$ and not on $\alpha$ itself.  We associate to $\alpha$ the element $p = \ell_{p_{\orbit{\alpha}}}$  and the subgroup $Q = Q_{t(\orbit{\alpha})}$, as described in the preamble to this section.

  Since $X$ itself is a CAT$(0)$ cube complex, the $1$--cell $\widetilde{\alpha}$ is not a loop.  Applying Lemma \ref{l:loop of length 1} in case $K = \{ 1 \}$ we see that $p \not\in Q$.  On the other hand, to say that $p \not\in Q.K$ is the same as saying there is no $CG(\mc{Y})$--loop of the form $p_{\orbit{\alpha}}.g$ which represents an element of $K$ (since in such a $CG(\mc{Y})$--loop the element $g$ must be in the local group $G_{t(\orbit{\alpha})}$).  This proves the result.
\end{proof}

The next result is an application of Lemma \ref{l:no loop of length 2} in case of paths of length $2$ consisting of $1$--cells in different $G$--orbits (since then the $K$--non-backtracking condition is vacuous).

\begin{theorem} \label{t:NPC 2}
Under Assumption \ref{SA} there exists a finite set $F_2 \subset \mc{Q}(G)^2 \times G^2$ so that for each
$(Q_1,Q_2,p_1,p_2) \in F_2$ we have
\[	1  \not\in p_1Q_1p_2Q_2	,	\]
and so that if (i) $K \unlhd G$ is co-cubical; and (ii) for each $(Q_1,Q_2,p_1,p_2) \in F_2$ we have 
\[	K \cap p_1Q_2p_2Q_2 = \emptyset	\]
then every loop of length $2$ in a link in $\leftQ{X}{K}$ consists of $1$--cells in the same $G$--orbit.
\end{theorem}
\begin{proof}
The proof is similar to the proof of Theorem \ref{t:NPC 1} above.   Lemma \ref{l:no loop of length 2} implies that it is enough to verify that no link in a cube of $\leftQ{X}{K}$ contains a pair of $1$--cells $\alpha$ and $\beta$ in
distinct $G$--orbits $\orbit{\alpha}, \orbit{\beta}$ so that there is a $CG(\mc{Y})$--loop $p_{\orbit{\alpha}} \cdot g_1 \cdot p_{\orbit{\beta}} \cdot g_2$ representing an element of $K$.

There are finitely many pairs of such orbits, and to each such pair we can associate the elements $p_1 = \ell_{p_{\orbit{\alpha}}}, p_2 = \ell_{p_{\orbit{\beta}}}, Q_1 = Q_{t(\orbit{\alpha})}, Q_2 = Q_{t(\orbit{\beta})}$.

Since $X$ is a CAT$(0)$ cube complex, there are no non-backtracking loops of length $2$ in any links in $X$, so applying Lemma \ref{l:no loop of length 2} with $K = \{ 1 \}$ we see that $1 \not\in p_1Q_1p_2Q_2$.  The result now follows from Lemma \ref{l:no loop of length 2} with our choice of $K$.
\end{proof}

For paths of length $2$ consisting of $1$--cells in the same $G$--orbit, the condition is slightly more complicated, as $K$--backtracking paths are possible.  
\begin{theorem} \label{t:NPC 3}
Under Assumption \ref{SA} there exists a finite set $F_3 \subset \mc{Q}(G)^2 \times G^2$ so that for each
$(Q_1,Q_2,p_1,p_2) \in F_3$ we have
\begin{equation}\label{eq:length2X}
	1 \not\in p_1 \left( Q_1 \minus Q_2^{p_2} \right) p_2 \left( Q_2 \minus Q_1^{p_1} \right)	,	
\end{equation}
and so that if (i) $K \unlhd G$ is co-cubical; (ii) no link in $\leftQ{X}{K}$ contains a loop of length $1$; and (iii) for every $(Q_1,Q_2,p_1,p_2) \in F_3$ we have
\begin{equation}\label{eq:length2} 
K \cap
  p_1 
  \left( Q_1 \minus
    \left( Q_2^{p_2} \left( K \cap Q_1 \right)\right)
  \right)  p_2	
  \left( Q_2 \minus
    \left( Q_1^{p_1} \left( K \cap Q_2 \right)\right)
  \right)
= \emptyset 
\end{equation}
then no link in $\leftQ{X}{K}$ contains an immersed loop of length $2$ consisting of $1$--cells in the same orbit.
\end{theorem}
\begin{proof}
  Because of assumptions (i) and (ii) we only need to be concerned with the following situation: There is some cube $\tilde{\sigma}$ of $X$ and some $1$--cell $\tilde{\alpha}$ in its link so that the following hold.
  \begin{enumerate}
  \item There is some $g\in G$ so that $g$ fixes $t(\tilde{\alpha})$ but not $\tilde{\alpha}$.
  \item There is some $h\in G$ so that $h(i(\tilde{\alpha}))=i(g\tilde{\alpha})$ but $h^{-1}g\tilde\alpha \neq \tilde\alpha$.
  \end{enumerate}
  There are finitely many orbits of pairs $(\tilde{\sigma},\tilde{\alpha})$  of this type.
  For each orbit we pick a representative, and describe an element of $\mc{Q}(G)^2\times G^2$ as in the theorem.  If Equation \eqref{eq:length2} is satisfied for this element, then $\leftQ{X}{K}$ will contain no immersed loop of length $2$ consisting of $1$--cells in the orbit of $\tilde{\alpha}$.

 We apply Lemma \ref{l:no loop of length 2} to a path of length $2$ of the form $\alpha.\alpha'$ where $\alpha$ is the image of $\tilde{\alpha}$ in $\leftQ{X}{K}$ and $\alpha'$ is the (oppositely oriented) image of a translate of $\tilde\alpha$ by an element of the stabilizer of $\tilde{\sigma}$.  Any immersed loop of the type we are trying to rule out gives rise to a $K$--non-backtracking $CG(\mc{Y})$--loop $p_{\orbit{\alpha}}\cdot g_1\cdot p_{\orbit{\backwards{\alpha}}}\cdot g_2$ representing a conjugacy class in $K$.  We let $p_1 = l_{p_{\orbit{\alpha}}}$, $p_2 = l_{ p_{\orbit{\backwards{\alpha}}}}$, $Q_1 = Q_{t(\orbit{\alpha})}$ and $Q_2 = Q_{i(\orbit{\alpha})}$.  
Using Lemma \ref{lem:Knonbacktrack}, the loop $p_{\orbit{\alpha}}\cdot g_1\cdot p_{\orbit{\backwards{\alpha}}}\cdot g_2$ is $K$--non-backtracking if and only if $g_1\notin E_{\orbit{\alpha}}K_{t(\orbit{\alpha})}$ and $g_2 \notin E_{\orbit{\backwards{\alpha}}}K_{i(\orbit{\alpha})}$.  The subgroup of $Q_1$ corresponding to $E_{\orbit{\backwards{\alpha}}}$ is equal to $Q_1\cap Q_2^{p_2}$, 
and the subgroup of $Q_2$ corresponding to $E_{\orbit{\backwards{\alpha}}}$ is $Q_2\cap Q_1^{p_1}$.  Thus
an element $p_1q_1p_2q_2$ of $p_1Q_1 p_2 Q_2$ comes from a $K$--non-backtracking $CG(\mc{Y})$--loop if and only if $q_1 \notin Q_2^{p_2}(K\cap Q_1)$ and $q_2 \notin Q_1^{p_1}(K\cap Q_2)$.  Applying Lemmas \ref{lem:Knonbacktrack} and \ref{l:no loop of length 2} in case $K = \{1\}$ and $\leftQ{X}{K} = X$ is CAT$(0)$, we see that our tuple satisfies Equation \eqref{eq:length2X}.  For an arbitrary $K$ we see that when Equation \eqref{eq:length2} is satisfied, there is no immersed loop of length $2$ in a link in $\leftQ{X}{K}$ consisting of images of translates of $\tilde{\alpha}$.
\end{proof}
In order to apply Lemma \ref{l:length 3 no backtrack}, in each of the following two results we make the extra assumption that $K$ is so that no link in $\leftQ{X}{K}$ contains a loop of length $1$ or $2$.  The following result is a translation of Lemma \ref{l:no loop of length 3 - dont fill}.

\begin{theorem} \label{t:NPC 4}
Under Assumption \ref{SA} there exists a finite set $F_4 \subset \mc{Q}(G)^3 \times G^3$ so that for each $(Q_1,Q_2,Q_3,p_1,p_2,p_3) \in F_4$ we have
\[	1 \not\in p_1Q_1p_2Q_2p_3Q_3	\]
and so that if (i) $K \unlhd G$ is co-cubical; (ii) no link in $\leftQ{X}{K}$ contains a loop of length $1$ or $2$; and (iii) for all $(Q_1,Q_2,Q_3,p_1,p_2,p_3) \in F_4$ we have
\[	K \cap p_1Q_1p_2Q_2p_3Q_3 = \emptyset	\]
then every loop of length $3$ in a link of $\leftQ{X}{K}$ has image in $\mc{Y}$ which bounds a $2$--cell.
 \end{theorem}
 \begin{proof}
 Condition (ii) and Lemma \ref{l:length 3 no backtrack} imply that it suffices to consider immersed loops of length $3$ in links in $\leftQ{X}{K}$.  For each choice of triple of $G$--orbits $\orbit{\alpha}, \orbit{\beta}, \orbit{\gamma}$ of $1$--cells in links in $X$ whose image in $\mc{Y}$ forms a loop, but whose image does not bound a $2$--cell in $\mc{Y}$ (in the sense described at the beginning of this subsection), we proceed as follows.  We associate the elements $p_1 = l_{p_{\orbit{\alpha}}}, p_2 = l_{p_{\orbit{\beta}}}$, $p_3 = l_{p_{\orbit{\gamma}}}$, $Q_1 = Q_{t(\orbit{\alpha})}$, $Q_2 = Q_{t(\orbit{\beta})}$, and $Q_3 = Q_{t(\orbit{\gamma})}$.
 
 Since $X$ is a CAT$(0)$ cube complex, we can apply Lemma \ref{l:no loop of length 3 - dont fill} to see that
 \[	1 \not\in p_1Q_1p_2Q_2p_3Q_3	.	\]
Now let $K \unlhd G$ be co-cubical, and satisfy conditions (i)--(iii) from the statement.  Condition (iii)  implies
that condition (2) from Lemma \ref{l:no loop of length 3 - dont fill} does not hold, and by that lemma there is no 
immersed loop of length $3$ in a link in $\leftQ{X}{K}$ whose image in $\mc{Y}$ is $\orbit{\alpha}, \orbit{\beta}, \orbit{\gamma}$.  
 
 Since there are finitely many such triples $\orbit{\alpha}, \orbit{\beta}, \orbit{\gamma}$, the theorem follows.
 \end{proof}

 Finally, we deal with loops of length $3$ in links in $\leftQ{X}{K}$ whose image in $\mc{Y}$ does bound a $2$-cell.

\begin{terminology}
Suppose that $A = (Q_1,Q_2,Q_3,p_1,p_2,p_3,h_1,h_2,h_3) \in \mc{Q}(G)^3 \times G^6$.  With indices read $\mathrm{mod}\ 3$, let
\[	A_i^- = Q_{i-1}^{p_{i-1}} \cap Q_i	,	\]
and let
\[	A_i^+ = Q_i \cap Q_{i+1}^{p_{i+1}}	.	\]
Furthermore, let
\[	B_i = A_i^- h_iA_i^+	.	\]
\end{terminology}
Using this terminology, we have the following translation of Proposition \ref{prop:no loop of length 3 - do fill}.
\begin{theorem} \label{t:NPC 5}
Under Assumption \ref{SA} there exists a finite set $F_5 \subseteq \mc{Q}(G)^3 \times G^6$ so that for each
$A = (Q_1,Q_2,Q_3,p_1,p_2,p_3,h_1,h_2,h_3)$ we have
\[	1 \not\in p_1 \left( Q_1 \minus B_1 \right) p_2
\left( Q_2 \minus B_2 \right)
p_3 \left( Q_3 \minus B_3 \right)	\]
and so that if (i) $K \unlhd G$ is co-cubical; (ii) no link in $\leftQ{X}{K}$ contains a loop of length $1$ or $2$; and (iii) for all $(Q_1,Q_2,Q_3,p_1,p_2,p_3,h_1,h_2,h_3) \in F_5$ we have
\[	K \cap  p_1 \left( Q_1 \minus B_1 \left( K \cap Q_1 \right) \right)
p_2 \left( Q_2 \minus B_2 \left( K \cap Q_2 \right) \right)
p_3 \left( Q_3 \minus B_3 \left( K \cap Q_3 \right) \right) = \emptyset	\]
 then no link in $\leftQ{X}{K}$ contains a loop of length $3$ which does not bound a $2$-cell but whose image in $\mc{Y}$ bounds a $2$--cell.
 \end{theorem}
\begin{proof}
  For each choice of triple of orbits $\orbit{\alpha}, \orbit{\beta}, \orbit{\gamma}$ whose image in $\mc{Y}$ bounds a $2$-cell (in the sense described at the beginning of this subsection), we proceed as follows.
Without loss of generality we choose representatives $\alpha, \beta, \gamma$ of these orbits so that there is a $2$--cell $\tau$ with boundary $\alpha \cdot \beta \cdot \gamma$.
We associate the elements $p_1 = \ell_{p_{\orbit{\alpha}}}, p_2 = \ell_{p_{\orbit{\beta}}}, p_3 = \ell_{p_{\orbit{\gamma}}}$, $Q_1 = Q_{t(\orbit{\alpha})}, Q_2 = Q_{t(\orbit{\beta})}$, $Q_3 = Q_{t(\orbit{\gamma})}$, and
$h_1 = [ c_{t(\orbit{\alpha})} \cdot g_{\tau,\alpha} \cdot \backwards{c_{t(\orbit{\alpha})}} ]$,
$h_2 = [ c_{t(\orbit{\beta})} \cdot g_{\tau,\beta} \cdot \backwards{c_{t(\orbit{\beta})}} ]$,
$h_3 = [ c_{t(\orbit{\gamma})} \cdot g_{\tau,\gamma} \cdot \backwards{c_{t(\orbit{\gamma})}} ]$.

Once again, since $X$ is a CAT$(0)$ cube complex, we can apply Proposition \ref{prop:no loop of length 3 - do fill} to see that 
$$1 \not\in p_1 \left( Q_1 \minus B_1 \right) p_2 \left( Q_2 \minus B_2 \right) p_3 \left( Q_3 \minus B_3 \right) . 	$$

When the conditions $g_1 \in E_{\orbit{\alpha_i}}g_{\tau_i,\alpha_i}E_{\orbit{\backwards{\beta_i}}} K_{t(\orbit{\alpha_i})}$, etc. from the statement of Proposition \ref{prop:no loop of length 3 - do fill} are translated into statements about the
group $G$ we get exactly $g_1 \in B_1 \left( K \cap Q_1 \right)$, etc., which gives the statement in the conclusion of the result.

Since there are finitely many such triples $\orbit{\alpha}, \orbit{\beta}, \orbit{\gamma}$, the theorem follows.
\end{proof}

\section{Dehn filling} \label{s:Dehn filling}
In this section we prove some results about group-theoretic Dehn filling.  Theorem \ref{t:meta} gives a `weak separability' of certain multi-cosets,
and generalizations of multi-cosets, and is used to find subgroups $K$ which satisfy the conditions from Theorems \ref{t:NPC 1}--\ref{t:NPC 5}.  Theorem \ref{t:meta} may be of independent interest, and we expect it to have applications beyond the scope of this paper.  The second main result of this section is Theorem \ref{t:fill to get CAT(0)}, from which Theorem \ref{t:fill to get proper action} from the introduction follows quickly by induction.

\subsection{Dehn fillings}
Let $(G,\mc{P})$ be a group pair, and let $\mc{N} = \{N_P\lhd P\mid P\in \mc{P}\}$ be a choice of normal subgroups of the peripheral groups.  The collection $\mc{N}$ determines a \emph{(Dehn) filling $(\overline{G},\overline{\mc{P}})$ of $(G,\mc{P})$}, where $\overline{G} = G/K$ for $K$ the normal closure of $\bigcup \mc{N}$, and $\overline{\mc{P}}$ equal to the collection of images of elements of $\mc{P}$ in $\overline{G}$.  The elements of $\mc{N}$ are called \emph{filling kernels}.  We sometimes write such a filling using the notation
\[ \pi\co (G,\mc{P})\to(\overline{G},\overline{\mc{P}}) ,\]
omitting mention of the particular filling kernels.

If $N_P \dotsub P$ (i.e. $N_P$ is finite index in $P$) for all $P\in \mc{P}$, we say that the filling is \emph{peripherally finite}.  If $H<G$ and for all $g \in G$,  $|H\cap P^g| = \infty$ implies $N_P^g\subseteq H$, then the filling is an \emph{$H$--filling}.  If $\mc{H}$ is a family of subgroups, the filling is an \emph{$\mc{H}$--filling} whenever it is an $H$--filling for every $H\in \mc{H}$.

A property $\mathsf{P}$ holds \emph{for all sufficiently long fillings} of $(G,\mc{P})$ if there is a finite set $S\subseteq \bigcup \mc{P}\minus\{1\}$ so that $\mathsf{P}$ holds whenever $(\bigcup\mc{N})\cap S = \emptyset$.  It is frequently useful to restrict attention to specific types of fillings (peripherally finite, $H$--fillings, etc.).  If $\mathsf{A}$ is a property of fillings we say that $\mathsf{P}$ holds for \emph{all sufficiently long $\mathsf{A}$--fillings} if, for all sufficiently long fillings, either $\mathsf{P}$ holds or $\mathsf{A}$ does not hold.

\subsection{Relatively hyperbolic group pairs}
We refer the reader to \cite{rhds} for a background on relatively hyperbolic groups.  In that paper, given a group pair $(G,\mc{P})$ (consisting of finitely generated groups) a space called the {\em cusped space} is built, which is $\delta$--hyperbolic (for some $\delta$) if and only if $(G,\mc{P})$ is relatively hyperbolic.
The cusped space is built by attaching \emph{combinatorial horoballs} to a Cayley graph for $G$.  Each combinatorial horoball $H$ has vertex set $tP\times\bZ_{\ge 0}$ for some coset $tP$ of some $P\in \mc{P}$, and is hyperbolic.  A vertex $(g,n)$ of such a horoball is said to have \emph{depth $n$}.  The depth $0$ vertices of the cusped space are exactly the vertices of the Cayley graph; if two vertices are connected by an edge, then their depths differ by at most one.
See \cite[Section 3]{rhds} for more information about the construction and geometry of the cusped space.
The following result is essentially contained in \cite[Theorem 7.11]{bowditch:relhyp}.

\begin{theorem}
 Suppose that $G$ is a hyperbolic group and that $\mc{P}$ is a finite collection of subgroups of $G$.  Then $(G,\mc{P})$ is relatively hyperbolic if and only if $\mc{P}$ is an almost malnormal family of quasi-convex subgroups.
\end{theorem}
Recall that $\mc{P} = \{ P_1 , \ldots , P_n \}$ is {\em almost malnormal} if whenever $P_i \cap P_j^g$ is infinite, we have $i = j$ and $g \in P_i$.

We can use the notion of height (see Definition \ref{d:height}) to measure how far away a family of subgroups is from being almost malnormal.

We now define the {\em induced peripheral structure} on $G$ associated to a finite collection of quasi-convex subgroups of a hyperbolic group, in analogy with the construction from \cite[Section 3.1]{agm}.

\begin{definition} \label{d:induced peripheral}
 Suppose that $G$ is a hyperbolic group and $\mc{H}$ is a finite collection of quasi-convex subgroups of $G$.  The {\em peripheral structure on $G$ induced by $\mc{H}$} is obtained as follows:
 
 Start by taking the collection of minimal infinite subgroups of the form \[H_1 \cap H_2^{g_2} \cap \ldots \cap H_k^{g_k}\] where the $H_i$ are in $\mc{H}$ and the cosets $\{ H_1, g_2H_2, \ldots , g_kH_k \}$ are all distinct.  Replace each element in this collection by its commensurator in $G$, and then choose one from each $G$--conjugacy class.  The resulting collection $\mc{P}$ is the induced peripheral structure.

If $H \in \mc{H}$ then the {\em induced peripheral structure on $H$ with respect to $\mc{H}$} is a choice of $H$-conjugacy representatives of intersections with $H$ of $G$--conjugates of elements of $\mc{P}$.
\end{definition}
We remark that the fact that there is a bound on the number $k$ of $g_iH_i$ as above follows from Proposition \ref{p:finite height}.

To state the next lemma we need a definition from \cite{agm}:
\begin{definition}
   Let $H<G$ and suppose $(H,\mc{D})$ and $(G,\mc{P})$ are relatively hyperbolic.  Suppose furthermore that every $D\in \mc{D}$ is conjugate into some $P\in\mc{P}$.  Then (see \cite[Lemma 3.1]{agm}) there is an induced $H$--equivariant map from the cusped space of $(H,\mc{D})$ to the cusped space of $(G,\mc{P})$.  The subgroup $(H,\mc{D})$ is \emph{relatively quasi-convex} if this induced map has quasi-convex image.
\end{definition}
 In \cite[Appendix A]{MM-P} it is proved that this is the same notion as the various notions of relative quasi-convexity discussed by Hruska in \cite{HruskaQC}.

The following can be proved in the same way as \cite[Proposition 3.12]{agm}.
\begin{lemma} \label{l:rh from induced}
 Suppose that $G$ is hyperbolic and $\mc{H}$ is a finite collection of quasi-convex subgroups of $G$.  
\begin{enumerate}
 \item The induced peripheral structure $\mc{P}$ is a finite collection of groups.   The pair $(G,\mc{P})$ is relatively hyperbolic.
\item If $H \in \mc{H}$ then the induced peripheral structure $\mc{D}$ of $H$ with respect to $\mc{H}$ is finite.  The pair $(H,\mc{D})$ is relatively hyperbolic.
\item For any $H \in \mc{H}$, the pair $(H,\mc{D})$ is full relatively quasi-convex in $(G,\mc{P})$.
\end{enumerate}
 \end{lemma}
A subgroup $H$ is {\em full} if whenever $P$ is a parabolic subgroup so that $H \cap P$ is infinite we have $H \cap P \dotsub P$.

\subsection{The appropriate meta-condition}
The goal of this subsection is to prove Theorem \ref{t:meta} below.
The special case that $n=1$ and $S_1 = \emptyset$ is \cite[Proposition 4.5]{agm}, which is about keeping elements out of full quasi-convex subgroups when performing long Dehn fillings.  Here we generalize to multi-cosets of full quasi-convex subgroups, possibly with some elements deleted.  Although the present result is more general, our proof is simpler, using the more appealing ``Greendlinger Lemma''--type Theorem \ref{t:greendlinger} below in place of the somewhat technical \cite[Lemmas 4.1 and 4.2]{agm}.\footnote{Using such a Greendlinger Lemma in place of the results of \cite{agm} was suggested to us by Alessandro Sisto while we were collaborating on \cite{GMS}.}
\begin{theorem}\label{t:meta}
  Let $(G,\mc{P})$ be relatively hyperbolic, and let $\mc{Q}$ be a collection of full relatively quasi-convex subgroups.  For $1\leq i\leq n$, let $p_i\in G$, $Q_i\in \mc{Q}$ and $S_i\subseteq Q_i$ be chosen to satisfy: 
  \begin{equation}
    \label{nontrivial}
 1\notin p_1(Q_1\minus S_1)\cdots p_n(Q_n \minus S_n)
  \end{equation}

  Then for sufficiently long $\mc{Q}$--fillings $G\to G/K$, the kernel $K$ contains no element of the form 
  \begin{equation}
    \label{forbidden}
 p_1 t_1 \cdots p_n t_n     
  \end{equation}
  where $t_i\in Q_i\minus \left( (K\cap Q_i) S_i \right)$.
\end{theorem}

The five conditions in the conclusions of Theorems \ref{t:NPC 1} -- \ref{t:NPC 5} each fall into the scheme of the conditions in Theorem \ref{t:meta}.  Therefore, we may apply Theorem \ref{t:meta} to obtain the following result.  We remark that the following result is stated in the generality of relatively hyperbolic groups acting cocompactly on cube complexes with full relatively quasi-convex subgroups.  This is greater generality than is strictly required for the proof of Theorem \ref{t:fill to get proper action}.  However, we believe that this extra generality will be of use in future work, and should be of independent interest.

\begin{corollary} \label{t:quotient CAT(0)}
Suppose that $(G,\mc{P})$ is relatively hyperbolic and that $G$ acts cocompactly on the CAT$(0)$ cube complex $X$.  Suppose every parabolic element of $G$ fixes some point of $X$, and that cell stabilizers are full relatively quasi-convex.  Let $\sigma_1,\ldots,\sigma_k$ be representatives of the $G$--orbits of cubes of $X$.  For each $i$ let $Q_i$ be the finite index subgroup of $\Stab(\sigma_i)$ consisting of elements which fix $\sigma_i$ pointwise.  Let $\mc{Q}=\{Q_1,\ldots,Q_k\}$.

For sufficiently long $\mc{Q}$--fillings 
\[	G \to \overline{G} = G(N_1, \ldots , N_m)	\]
of $(G,\mc{P})$, with kernel $K$, the quotient $\leftQ{X}{K}$ is a CAT$(0)$ cube complex.
\end{corollary}
\begin{proof}
  The kernels of Dehn fillings are always generated by parabolic elements, and the parabolic elements act elliptically by assumption.  Thus the kernel of {\em any} Dehn filling is generated by elliptic elements, so $\leftQ{X}{K}$ is simply-connected by Theorem \ref{thm:armstrong}.  For sufficiently long $\mc{Q}$--fillings the fact that $G_{\sigma_i} \cap K \le Q_i$ follows from \cite[Proposition 4.4]{agm}, so by Proposition \ref{p:quotient cubical} for such fillings $\leftQ{X}{K}$ is a cube complex.
Therefore, we may assume that the subgroup $K$ is co-cubical (in the sense of Terminology \ref{term:tame}).

It remains to show that for sufficiently long $\mc{Q}$--fillings $\leftQ{X}{K}$ is non-positively curved.  It follows from Theorems \ref{t:NPC 1}--\ref{t:NPC 3} and \ref{t:meta} that for sufficiently long $\mc{Q}$--fillings each link of each cell in $\leftQ{X}{K}$ is simplicial.  Thus it follows from Theorem \ref{t:NPC 4}, \ref{t:NPC 5} and \ref{t:meta} that for sufficiently long $\mc{Q}$--fillings, each link of each cell in $\leftQ{X}{K}$ is also flag, which means that $\leftQ{X}{K}$ is non-positively curved by Theorem \ref{t:link condition}.
\end{proof}

To prove Theorem \ref{t:meta}, we use the following ``Greendlinger Lemma'' (cf. \cite[Lemma 2.26]{GMS}):
\begin{theorem}\label{t:greendlinger}
  Let $C_1,C_2>0$.  Suppose that $(G,\mc{P})$ is relatively hyperbolic, with cusped space $X$.  For all sufficiently long fillings $G\to G/K$, and any geodesic $\gamma$ in $X$ joining $1$ to $g\in K\minus\{1\}$, there is a horoball $A$ so that 
  \begin{enumerate}
  \item $\gamma$ contains a depth $C_1$ vertex of $A$, and
  \item there is an element $k$ of $K$ stabilizing $A$, so that, for two points $a,b$ in $A$ and lying on $\gamma$ at depth at least $C_1$, $d(a,kb)<d(a,b)-C_2$ (in particular $d(1,kg)<d(1,g) - C_2$).
  \end{enumerate}
\end{theorem}
\begin{proof}
Let $\delta>0$ be such that $X$ is $\delta$--hyperbolic, and so are the cusped spaces for sufficiently long fillings (that there exists such a $\delta$ is \cite[Proposition 2.3]{agm}).  We only consider such fillings, without further mention of this assumption.

  Now choose $L,\epsilon$ so that every $L$--local $(1,C_2)$--quasi-geodesic lies within an $\epsilon$--neighborhood of any geodesic with the same endpoints.  (Such $L$, $\epsilon$ only depend on $\delta$ and $C_2$. See \cite[Ch. 3]{cdp}.)

  Now choose a filling long enough so that every $(2L+C_1+2\epsilon)$--ball centered on the Cayley graph embeds in the quotient cusped space.  Let $K$ be the kernel of the filling, and choose $g\in K\minus \{1\}$.  Let $\gamma$ be a geodesic from $1$ to $g$; let $\overline{\gamma}$ be the projection to the cusped space $\leftQ{X}{K}$ for $G/K$.  Within an $(L+C_1+2\epsilon)$--neighborhood of the Cayley graph, $\overline{\gamma}$ is an $L$--local geodesic.  But $\overline{\gamma}$ cannot be an $L$--local $(1,C_2)$--quasi-geodesic everywhere, since it is a loop with diameter larger than $\epsilon$.  

In particular, there is a subsegment $\sigma$ of $\overline{\gamma}$ of length $l\leq L$ so that the endpoints $\overline{a}$ and $\overline{b}$ of $\sigma$ are less than $l-C_2$ apart.  This subsegment $\sigma$ must moreover lie in the image of a single horoball.

The corresponding points $a$ and $b$ on $\gamma$ lie at depth at least $C_1$ in a horoball $A$ of $X$.  Since $d(\overline{a},\overline{b})<l-C_2$, there is some element $k\in K$ stabilizing $A$ so that $d(a,kb)<l-C_2$, as desired.
\end{proof}

The following result follows immediately from \cite[A.6]{MM-P}.

\begin{lemma} \label{l:intersect parabolic}
Suppose that $(G,\mc{P})$ is relatively hyperbolic with cusped space $X$ and that $(H,\mc{D}) \le (G,\mc{P})$ is a full relatively quasi-convex subgroup.  There exists a constant $\kappa$ satisfying the following:

Suppose that $g \in G$ and that $x_1, x_2 \in gH$.  Suppose that $\gamma$ is a geodesic in $X$ between $x_1$ and $x_2$.  Further, suppose that $aP$ (for $a \in G$ and $P \in \mc{P}$) is a coset so that $\gamma$ intersects the horoball corresponding to $aP$ to depth at least $\kappa$.  Then $P$ is infinite and $P^a \cap H^g$ has finite-index in $P^a$.
\end{lemma}

\begin{proof}[Proof of Theorem \ref{t:meta}]
Let $X$ be the cusped space associated to $(G,\mc{P})$ and suppose that $X$ is $\delta$--hyperbolic.  Let $C_2$ be any positive number, and let $C_1 = \max \{ |p_i|, \kappa \} + 2 (n+100)\delta$, where $\kappa$ is the constant from Lemma \ref{l:intersect parabolic} above.  Suppose that $K$ is the kernel of a filling which is long enough to satisfy the conclusion of Theorem \ref{t:greendlinger} with these constants.  

In order to obtain a contradiction, suppose that there is an element $g \in K$ which is of the form
\[	g = p_1 t_1 \cdots p_n t_n ,\]
where $t_i \in Q_i \minus \left( (K \cap Q_i ) S_i \right)$, and suppose that $g$ is chosen so that $d_X(1,g)$ is minimal amongst all such choices.

Since for each $i$ we have $Q_i \minus \left( (K \cap Q_i ) S_i \right) \subseteq Q_i \minus S_i$, the assumption of the theorem implies that $g \ne 1$.  We can represent the equation $g = p_1t_1 \cdots t_np_n$ by a geodesic $(2n+1)$--gon in $X$, joining the appropriate elements of the Cayley graph in turn by $X$--geodesics.  Let $\gamma$ be the geodesic for $g$, $\rho_i$ the geodesic for $p_i$ and $\tau_i$ the geodesic for $t_i$.

Since $g \in K \minus \{ 1 \}$, by Theorem \ref{t:greendlinger} there exist a horoball $A$ in $X$, an element $k \in K$ stabilizing $A$, and points $a, b$  on $\gamma$ at depth at least $C_1$ so that $k$ stabilizes $A$ and $d(a,kb) < d(a,b) - C_2$.  In particular, we have $d(x,kgx) < d(x,gx) - C_2$.  The geodesic $(2n+1)$--gon is $(2n-1)\delta$--thin, so $b$ lies within distance $(2n-1)\delta$ of some side other than $\gamma$.
The paths $\rho_i$ do not go deeply enough into any horoballs to be this close to $b$, so $b$ lies within $(2n-1)\delta$ of some point $b'$ on some $\tau_i$.  By the choice of $C_1$, $b'$ lies at depth at least $\kappa$ in $A$.  

Write $A = aP$ for some $P \in \mc{P}$.  Note that $\tau_i$ is a geodesic between two points in 
the coset ${p_1t_1 \cdots p_{i}} Q_i$.
By Lemma \ref{l:intersect parabolic}, $P^a \cap Q_i^{p_1t_1 \cdots p_i}$ has finite-index in $P^a$.  Since the filling is a $\mc{Q}$--filling, we have that $k \in Q_i^{p_1t_1 \cdots p_i}$.

Let $k' = k^{(p_1t_1 \cdots p_i)^{-1}}$, and let $t_i' = k't_i$.  Then $k' \in K \cap Q_i$.

Note that $kg = p_1t_1 \cdots p_i (k't_i) p_{i+1} \cdots p_nt_n$.
Since $t_i \not\in (K \cap Q_i)S_i$, we have that $t_i' \not\in (K \cap Q_i)S_i$.  Therefore, the element $kg$ is another element of the required form, contradicting the choice of $g$ as the shortest such.  This completes the proof of Theorem \ref{t:meta}.
\end{proof}

\subsection{Dehn fillings which induce CAT$(0)$ quotient cube complexes}

\begin{theorem} \label{t:fill to get CAT(0)}
Suppose that the hyperbolic group $G$ acts cocompactly on the CAT(0) cube complex $X$, and that cell stabilizers are virtually special and quasi-convex.  Let $\sigma_1, \ldots , \sigma_k$ be representatives of the $G$--orbits of cubes of $X$, and for each $i$ let $Q_i$ be the finite-index subgroup of $\Stab(\sigma_i)$ consisting of elements which fix $\sigma_i$ pointwise.  Let $\mc{Q} = \{ Q_1, \ldots,  Q_k\}$, and let $\mc{P}$ be the peripheral structure on $G$ induced by $\mc{Q}$, as in Definition \ref{d:induced peripheral}.

If some element of $\mc{Q}$ is infinite, then there exists a Dehn filling
\[	G \onto \overline{G} = G(N_1, \ldots , N_m)	\]
of $(G,\mc{P})$, with kernel $K_{\mc{P}}$ so that
\begin{enumerate}
\item\label{Gbar hyp} $\overline{G}$ is hyperbolic;
\item\label{Qbar VS QC} $\overline{\mc{Q}}$ consists of virtually special quasi-convex subgroups of $\overline{G}$.
\item\label{K_P gen ell} $K_{\mc{P}}$ is generated by elements in cell stabilizers.
\item\label{K_P in Q_i} For each $i$, we have $K_{\mc{P}} \cap \Stab(\sigma_i) \le Q_i$;
\item\label{height} $\height(\overline{\mc{Q}}) < \height(\mc{Q})$.
\item\label{quotient CAT(0)} $\leftQ{X}{K_{\mc{P}}}$ is a CAT(0) cube complex;
\end{enumerate}
\end{theorem}
\begin{proof}
Let $G, X, \mc{Q}$ and $\mc{P}$ be as in the statement of the theorem.  By Lemma \ref{l:rh from induced}, $(G,\mc{P})$ is relatively hyperbolic.  Moreover, for each $Q \in \mc{Q}$, the induced structure $\mc{D}_Q$ on $Q$ makes $(Q,\mc{D}_Q)$ relatively hyperbolic, and $Q$ is full relatively quasi-convex in $(G,\mc{P})$.  Note that the assumption that some element of $\mc{Q}$ is infinite implies (by the definition of $\mc{P}$) that some element of $\mc{P}$ is infinite.

Property  \eqref{Gbar hyp} holds for sufficiently long peripherally finite fillings of $(G,\mc{P})$ by the basic result of relatively hyperbolic Dehn fillings \cite[Theorem 1.1]{osin:peripheral}.  We always assume that we have taken a filling so that $\overline{G}$ is hyperbolic.

We remark that, because each element of $\mc{Q}$ is finite-index in a cell stabilizer, each element of $\mc{Q}$ is hyperbolic and virtually special.  Moreover, since each element of $\mc{P}$ has a finite-index subgroup which is a quasi-convex subgroup of some element of $\mc{Q}$ by construction, each element of $\mc{P}$ is also hyperbolic and virtually special.  In particular, each element of $\mc{P}$ is residually finite.  We choose particular fillings with $N_i \dotsub P_i$, and residual finiteness guarantees the existence of the fillings that we seek.

We now explain how to ensure the properties of the conclusion of the result.

Suppose that $Q \in \mc{Q}$.  Since $\mc{P}$ is the peripheral structure induced by $\mc{Q}$, we can choose finite-index subgroups of elements of $\mc{P}$ which induce $\mc{Q}$--fillings, and any such filling $\overline{G}$ of $G$ naturally induces a filling $\overline{Q}$ of $Q$.  By the Malnormal Special Quotient Theorem \cite[Theorem 12.3]{WisePUP} (see also \cite[Corollary 2.8]{AGM_msqt}) for each $P_i \in \mc{P}$ there is a subgroup $\dot{P}_i(Q) \dotnorm P_i$ so that if each filling kernel $N_i$ satisfies $N_i \le \dot{P}_i(Q)$ then the induced filling $\overline{Q}$ is virtually special (and hyperbolic).  Let $\dot{P}_i$ be the intersection of the $\dot{P}_i(Q)$ for all $Q \in \mc{Q}$.  Thus, if we choose filling kernels $N_i \le \dot{P}_i$ then each of the induced fillings of each element of $\mc{Q}$ is virtually special.
By \cite[Proposition 4.6]{Hwide}, the natural map from $\overline{Q}$ to $\overline{G}$ is injective for all sufficiently long fillings.\footnote{\cite[Lemma 3.7]{Hwide} ensures that sufficiently long $\mc{Q}$--fillings are sufficiently wide, in the terminology of that paper.}  If we choose a sufficiently long peripherally finite filling of $(G,\mc{P})$ with $N_i \le \dot{P}_i$ then \cite[Proposition 4.5]{Hwide} implies that each $\overline{Q}$ is quasi-convex in $\overline{G}$.
This ensures Property \eqref{Qbar VS QC}.

For the remaining properties, we show that they hold for sufficiently long peripherally finite $\mc{Q}$--fillings of $(G,\mc{P})$.  Therefore, to ensure that all of the properties hold, it suffices to take a sufficiently long $\mc{Q}$--filling with each $N_i \dotsub \dot{P}_i$.

Property \eqref{K_P gen ell} holds automatically for any $\mc{Q}$--filling, since $K_\mc{P}$ is generated by conjugates of elements in $\mc{Q}$, and each such conjugate lies in a cell stabilizer.

We now explain how to ensure each of the remaining properties in turn for sufficiently long $\mc{Q}$--fillings.

For property \eqref{K_P in Q_i}, suppose that $\mc{F}_i \sqcup \{ 1 \}$ is a set of coset representatives for $Q_i$ in $\Stab(\sigma_i)$.  To ensure that \eqref{K_P in Q_i} holds, it suffices to keep (the image of) each element of $\mc{F}_i$ out of the image of $\Stab(\sigma_i)$ in $\overline{G}$.  This is true for sufficiently long $\mc{Q}$--fillings by \cite[Theorem A.43.4]{VH}, because $Q_i$ has finite index in $\Stab(\sigma_i)$.

Property \eqref{height} holds for sufficiently long peripherally finite $\mc{Q}$--fillings of $(G,\mc{P})$ by an entirely analogous argument to that of \cite[Theorem A.47]{VH}.

Finally, Property \eqref{quotient CAT(0)} holds for sufficiently long $\mc{Q}$--fillings by Corollary \ref{t:quotient CAT(0)}.
\end{proof}

The group $\overline{G}$ as above acts isometrically on $\overline{X} = \leftQ{X}{K_{\mc{P}}}$ with quotient naturally isomorphic (as a topological space, but not as a complex of groups) to $\leftQ{X}{G}$.  Therefore, if the action of $\overline{G}$ on $\overline{X}$ is not proper, we can apply Theorem \ref{t:fill to get CAT(0)} to this action, to obtain a further quotient.  By induction on height, we obtain the following result from the introduction.

\fillproper*

\appendix
\section{A quasi-convexity criterion} \label{app:QC}

In this appendix, we give a criterion (Theorem \ref{t:globally QC}) for a possibly infinite union of quasi-convex sets in a hyperbolic space to be quasi-convex.  This criterion is used in the forward direction of  Theorem \ref{t:cell qc iff hyp}:  quasi-convex cell stabilizers imply quasi-convex hyperplane stabilizers.  This criterion may be of independent interest.

Since any subset is a union of points, clearly some assumptions are needed.

We begin with a basic lemma about finite unions of quasi-convex subsets.

\begin{lemma} \label{l:QC constant is log}
Suppose that $Y$ is $\delta$--hyperbolic, and $P \subset Y$ is a union of $k$ $\epsilon$--quasi-convex subsets $P_1 , \ldots , P_k$ so that $P_i \cap P_{i+1} \ne \emptyset$ for each $i$.  Then $P$ is $\rho$--quasi-convex where
\[	\rho = \delta ( \log_2(k) + 1) + \epsilon	.	\]
\end{lemma}
\begin{proof}
Consider a pair of points $x \in P_r$, $y \in P_s$.  Without loss of generality, assume that $r < s$ (the case $r = s$ being straightforward).  

Now choose a sequence of points $p_i \in P_i \cap P_{i+1}$ for $r \le i < s$, let $\sigma$ be a geodesic between $x$ and $y$ and let $u$ be a point on $\sigma$.  Our task is to bound the distance from $u$ to $P$.

Consider the broken geodesic
$\gamma = [x,p_r,p_{r+1}, \ldots , p_{s-1},y]$.  Since the $P_i$ are $\epsilon$--quasi-convex, $\gamma$ is contained in an $\epsilon$--neighborhood of $P_r \cup \ldots \cup P_s \subset P$.

Consider the geodesic polygon with one side the geodesic $\sigma = [x,y]$ and the other sides the geodesics forming $\gamma$.  Let $r_0 = \lfloor \frac{r+s}{2} \rfloor$, and consider the geodesic triangle $\sigma$, $[x,p_{r_0}], [p_{r_0},y]$.  By $\delta$--hyperbolicity, $u$ lies within $\delta$ of one of $[x,p_{r_0}]$ and $[p_{r_0},y]$.
Suppose it is $[x,p_{r_0}]$ (the other case being entirely similar), and suppose that $u_1 \in [x,p_{r_0}]$ is within $\delta$ of $u$.

Now let $r_1 = \lfloor \frac{r+r_0}{2} \rfloor $ and consider the geodesic triangle $[x,p_{r_1}],[p_{r_1},p_{r_0}], [p_{r_0},x]$.
By $\delta$--hyperbolicity, $u_1$ is within $\delta$ of one of $[x,p_{r_1}],[p_{r_1},p_{r_0}]$, so there is $u_2$ on one of these sides within $\delta$ of $u_1$ and within $2\delta$ of $u$.

We proceed in this manner, in each case making the interval of indices half as long.  After $t$ steps of this argument we find a point $u_t$ which is within within $t\delta$ of $u$.

After at most $d = \log_2(k) + 1$ steps, we have a geodesic triangle where two sides are $[p_l, p_{l+1}], [p_{l+1},p_{l+2}]$ (or maybe one endpoint $x$ or $y$), and we have $u_d$ within $d\delta$ of $u$, but also within $\epsilon$ of $P$.  This proves the lemma.
\end{proof}

The following straightforward instance of ``linear-beats-log'' is tailored for use in the proof of Theorem \ref{t:globally QC}.

\begin{lemma} \label{l:linear beats log}
Fix $\delta, \epsilon>0$, and let $g(x) = \delta\left( \log_2 (x + 1)+1\right) + \epsilon$.
For any $m > 0$ and $c \geq 0$ there exists a natural number $R_{m,\epsilon,\delta}$ so that for all $R_0 > R_{m,\epsilon, \delta}$, we have
\[	g(R_0) < \frac{1}{200} m \left( \frac{1}{4}R_0 - \frac{2g(R_0)+1}{m}-3c \right)	.	\]
\end{lemma}

The next result states that under appropriate hypotheses, the union of an arbitrary number of quasi-convex subsets is itself quasi-convex, with constant not depending on the number of such subsets.

\begin{theorem} \label{t:globally QC}
Suppose that $\Upsilon$ is a $\delta$--hyperbolic space and that $m, \epsilon > 0$ and $c \ge 0$ are real numbers.  There exists a constant $\epsilon'$ so that for any (finite or countably infinite) collection of subsets $\{ X_i \}_{i=1}^\Lambda$ of $\Upsilon$ for which
\begin{enumerate}
\item Each $X_i$ is $\epsilon$--quasi-convex;
\item For each $i$ we have $X_i \cap X_{i+1} \ne \emptyset$; and
\item \label{i:reg progress} For any $i,j,$ if $x \in X_i$ and $y \in X_j$ we have $d(x,y) \ge m \left(| i - j |-c \right)$,
\end{enumerate}
the set $\mathbf{X}=\cup X_i$ is $\epsilon'$--quasi-convex.
\end{theorem}
\begin{proof}
 Let $g(x) = \delta\left( \log_2 (x + 1)+1\right) + \epsilon$, and let
$R = R_{m,\epsilon,\delta}$ be the number from Lemma \ref{l:linear beats log}.  Without loss of generality we may assume that $R \ge 1$.  
 
 If $\Lambda \le 100R$ then Lemma \ref{l:QC constant is log} implies $\mathbf{X}$ is $\rho$-quasi-convex with $$\rho = \delta\left( \log_2(100R)+1 \right) +\epsilon = g(100R-1).$$

On the other hand, suppose that $\Lambda > 100R$ and fix $u,v\in \mathbf{X}$.  Let $j,k$ be so that $u \in X_j, v \in X_k$ and without loss of generality suppose that $j \le k$.  It suffices to show that any geodesic $[u,v]$ stays uniformly close to $X_j \cup \cdots \cup X_k$.  If $|k-j| \le 100 R$ then this follows from Lemma \ref{l:QC constant is log}, so suppose that $|k-j| > 100R$.  Let $Y = X_j \cup \cdots \cup X_k$.

Our strategy is to build a path between $u$ and $v$ which is (i) uniformly quasi-geodesic; and (ii) stays uniformly close to $Y$.  The theorem then follows by quasi-geodesic stability.  Choose a sequence of indices $t_0 = j, t_1 , \ldots , t_{s-1}, t_s = k$ so that for each $0 \le r \le s-2$ we have
\[	 t_{r+1} - t_r = 100R	,	\]
and
\[	t_s - t_{s-1} \in \bZ\cap \left[ 100R, \ldots , 200R	\right]	.	\]
Moreover, for each $0 \le r \le s$ choose some $u_r \in X_{t_r}$.  We require $u_0 = u$ and $u_s = v$.

For $r\in \{0,\ldots,s-1\}$, let $\gamma_r$ be a geodesic between $u_r$ and $u_{r+1}$.  Let $$K = g(200R)=\delta (\log_2(200R+1)+1) + \epsilon.$$  Since we assume $R \ge 1$ we know that $K > \delta$.

Since we know that for each $r\in \{0,\ldots,s-1\}$ we have $t_{r+1} - t_r \le 200R$ we know that the set 
\[	Y_r = \bigcup_{k = t_r}^{t_{r+1}} X_k	,	\]
is $K$--quasi-convex, by Lemma \ref{l:QC constant is log}.  In particular the geodesic $\gamma_r$ lies in a $K$--neighborhood of $Y_r$.

For each $r\in \{0,\ldots,s-1\}$ and each $x \in \gamma_r$, let $\pi_r(x)$ denote the set of closest points on $Y_r$ to $x$.  Furthermore, let $I_r(x)$ be the set of indices $l$ so that $\pi_r(x) \cap X_l \ne \emptyset$.

\begin{claim} \label{claim:progress}
For any $v \in \{ t_r, \ldots , t_{r+1} \}$ there exists $x_v \in \gamma_r$ so that
\[	d_\N \left( v , I_r(x_v) \right) \le \frac{1}{2}\left(\frac{2K+1}{m}+c \right)	.	\]
\end{claim}
\begin{proof}[Proof of Claim \ref{claim:progress}]
For any $y \in \pi_r(x)$ we have $d(x,y) \le K$.  Now, if $x$ and $x'$ are adjacent vertices and $y \in \pi_r(x)$ with $y \in X_k$ and $z \in \pi_r(x')$ with $z \in X_l$ then
$m( \left| k-l \right| - c ) \le d(y,z) \le d(y,x) + d(x,x') + d(x',z) \le 2K+1$, so $| k -l| \le \frac{2K+1}{m} + c$.

The claim now follows immediately from the fact that $t_r \in I_r(u_r)$ and $t_{r+1} \in I_r(u_{r+1})$, letting $x$ and $x'$ run over adjacent pairs of vertices in $\gamma_r$.  This finishes the proof of Claim \ref{claim:progress}.
\end{proof}

Suppose $0\leq r\leq s-1$.
Using Claim \ref{claim:progress}, we can choose a point $x_r \in \gamma_r$ and a point $y_r \in \pi_r(x_r)$ so that $y_r \in X_{k_r}$ and
\begin{equation} \label{kr in middle}
\left| k_r - \frac{t_r + t_{r+1}}{2} \right| \le \frac{2K+1}{2m}+c	. \tag{$\dagger$}
\end{equation}

Now, for each $r\in \{1,\ldots,s-1\}$, let $\sigma_r$ be a geodesic between $y_{r-1}$ and $y_r$.  Further, let $\sigma_0$ be a geodesic from $u$ to $y_0$ and let $\sigma_s$ be a geodesic
from $y_{s-1}$ to $v$ (note that there is no point $y_s$).  See Figure \ref{fig:sigmas}.
\begin{figure}[htbp]
  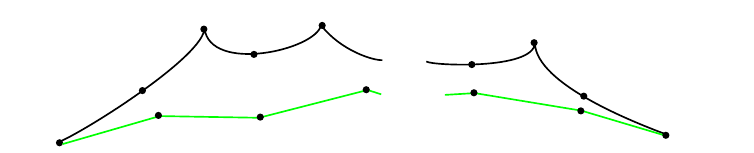
  \caption{The $\sigma_i$ forming a broken geodesic.}
  \label{fig:sigmas}
\end{figure}
We bound the Gromov product between $\sigma_t$ and $\sigma_{t+1}$ for each $t$.  
(There is no reason to expect such a bound on the Gromov product between $\gamma_r$ and $\gamma_{r+1}$.)

Though we have no control on the lengths of the segments $\sigma_0$ and $\sigma_s$, the lengths of the other segments can be bounded below:
\begin{claim} \label{sigma long}
Suppose $0< r < s$.  The length of $\sigma_r$ is at least $200K$. 
\end{claim}
\begin{proof}
 By the choice of the index $k_r$ in Equation \eqref{kr in middle} we have
\begin{eqnarray*}
 k_r - k_{r-1} &\ge& \frac{ t_{r+1} - t_{r-1}}{2} - \frac{ 2K + 1}{m}-2c \\
 &=& 100R -  \frac{ 2K + 1}{m}-2c\\
 & > & 50R - \frac{ 2K + 1}{m}-2c
\end{eqnarray*}
(the equality follows from the choice of $t_r$).

Below, we apply Lemma \ref{l:linear beats log} with $R_0 = 200R$, noting that $K = g(200R)$, where $g$ is the function from that lemma.  We have
 
\begin{eqnarray*}
 |\sigma_r| &=& d_G(y_{r-1},y_r) \\
 & \ge & m(k_r - k_{r-1}-c) \\
 & > & m \left( 50R -   \frac{ 2K + 1}{m}-3c \right)\\
 & \ge & 200K
\end{eqnarray*}
The second inequality above follows from the fact that $y_{i} \in X_{k_{i}}$ so such points are at least distance $m(k_r - k_{r-1}-c)$ apart.  The final inequality follows from the promised use of Lemma \ref{l:linear beats log}.  This completes the proof of Claim \ref{sigma long}.
\end{proof}

\begin{claim}
  Let $0\leq r\leq s-1$.
  The Gromov product of $\sigma_r$ and $\sigma_{r+1}$ is at most $8K$.
\end{claim}
\begin{proof}
We first handle the case that $0<r<s-1$.

For $i\in \{r,r+1\}$, the path $\sigma_i$ is one side of a pentagon.  The other sides are (A) two sides of length at most $K$ at either end of $\sigma_i$, and (B) two `halves' of adjacent geodesics: the second `half' of $\gamma_{i-1}$ and the first `half' of $\gamma_i$, joined at $u_i$.  See Figure \ref{fig:sigmapents}.
\begin{figure}[htbp]
  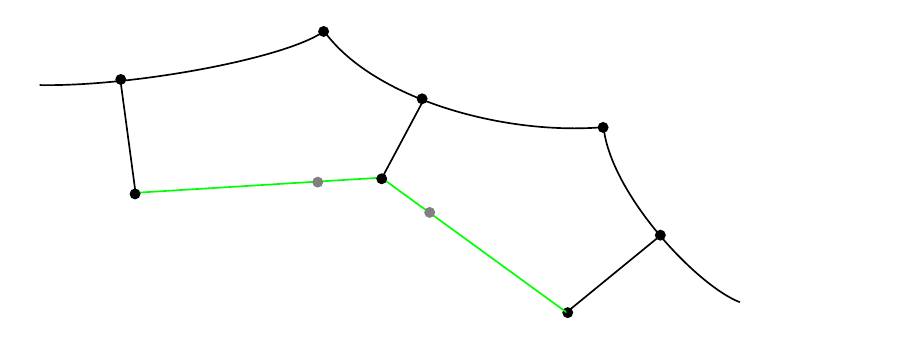
  \caption{Computing the Gromov product of $\sigma_r$ and $\sigma_{r+1}$.}
  \label{fig:sigmapents}
\end{figure}

By Claim \ref{sigma long} the geodesics $\sigma_r$ and $\sigma_{r+1}$ have length at least $200K$.
Let $z$ be the point on $\sigma_r$ at distance exactly $8K$ from $y_r$.  

Since geodesic pentagons are $3\delta$--slim, we know that $z$ must be distance at most $3\delta$ from some point on one of the other four sides.  However, it cannot be within distance $3\delta$ of the geodesic between $x_r$ and $y_r$ since that geodesic has length at most $K$.  Similarly, since $|\sigma_r| \ge 200K$, $z$ cannot be within $3\delta$ of the geodesic between $x_{r-1}$ and $y_{r-1}$.  We claim that $z$ also cannot be within $3\delta$ of the part of $\gamma_{r-1}$ contained in the pentagon.  

Indeed, suppose $w\in \gamma_{r-1}$, and choose $i_w\in I_{r-1}(w)\subset [t_{r-1},t_r]$.  There is a point $w'$ of $\pi_{r-1}(w)$ in $X_{i_w}$; thus $d(w,w')\leq K$.  The point $x_r$ is likewise within $K$ of some $X_{k_r}$ where $k_r$ satisfies the inequality \eqref{kr in middle}.  This implies that 
\[ |k_r - i_w| \geq \frac{t_{r+1}-t_r}{2} - \frac{2K+1}{2m}-c, \]
and so
\begin{eqnarray*} 
d(x_r,w)& \geq &  m\left(\frac{t_{r+1}-t_r}{2} - \frac{2K+1}{2m}-2 c\right)-2K\\
 & \geq & m\left(50 R - \frac{2K+1}{m} - 3c \right)-2K\\
 & \geq & 198K,
\end{eqnarray*}
using Lemma \ref{l:linear beats log} again.
But this contradicts $d(x_r,w)\leq d(x_r,z)+d(z,w)\leq 9K + 3\delta \leq 12K$.

We have shown that there is some point $w$ on $\gamma_r$ between $u_r$ and $x_r$ within $3\delta$ of $z$.  Note that $d(x_r,w)\geq d(y_r,z) - K-3\delta \geq 4K$, since $K\geq \delta$.

Now consider the pentagon formed with $\sigma_{r+1}$ on one side, and the point $z'$ on $\sigma_{r+1}$ which is distance exactly $8K$ from $y_r$.  An entirely analogous argument to the above shows that there is some $w'$ between $x_r$ and $u_{r+1}$ on $\gamma_r$ so that $d(z',w')\leq 3\delta$, and $d(x_r,w')\geq 4K$.  Since $\gamma_r$ is geodesic, we have
\[ d(w,w') = d(w,x_r)+d(x_r,w) \geq 8K. \]
It follows that $d(z,z')\geq 8K-6\delta \geq 2K>\delta$.  It follows that the Gromov product $(y_{r-1},y_{r+1})_{y_r}$ is strictly less than $d(z,y_r)=d(z',y_r)=8K$, whenever $0<r<s$.

The cases $r=0$ and $r=s-1$ are symmetric, so it suffices to handle the case $r=0$.  See Figure \ref{fig:endcase}.
\begin{figure}[htbp]
  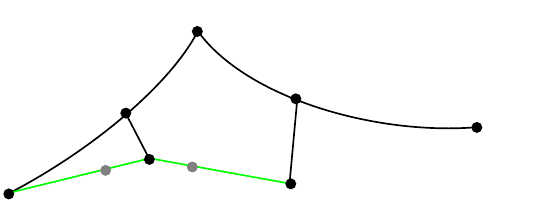
  \caption{Computing the Gromov product of $\sigma_0$ and $\sigma_1$.}
  \label{fig:endcase}
\end{figure}
We are trying to show that $(u,y_1)_{y_0}\leq 8K$, so we may suppose without loss of generality that $d(y_0,u)>8K$.  Thus there is a point $z$ on $\sigma_0$ at distance exactly $8K$ from $y_0$.  Since $d(x_0,y_0)\leq K$, this point is within $\delta$ of a point $w$ on $\gamma_0$ between $u$ and $x_0$.  

For the point $z'$ on $\sigma_1$ at distance $8K$ from $y_0$, we argue as before.  We are again able to deduce that $d(z,z')>\delta$, and so $(u,y_1)_{y_0}\leq 8K$.
\end{proof}

Thus, we have a collection of arcs $\sigma_i$ which form a broken geodesic between $u$ and $v$ with segments of length at least $200K$ (except possibly the first and last) and all Gromov product at most $8K$ at the corners.  Thus the union of the $\sigma_i$ forms a global quasi-geodesic with uniformly bounded parameters.  However, each $\sigma_i$ lies within a $(3\delta + K)$--neighborhood of the union of the $\gamma_i$, which in turn lie in a $K$--neighborhood of the union of the $X_i$.
As explained above, this suffices to prove that the union of the $X_i$ is $\epsilon'$--quasi-convex with the constant $\epsilon'$ depending on the quantities $\delta$, $m$, and $\epsilon$, but not on the number of the $X_i$, as required.  This completes the proof of Theorem \ref{t:globally QC}.
\end{proof}

\bibliographystyle{abbrv}

\end{document}